\documentclass[reqno]{amsart}
\usepackage[utf8]{inputenc}
\usepackage[usenames,dvipsnames,svgnames,table]{xcolor}
\usepackage{amsmath, amssymb, amsthm, verbatim, hyperref, datetime}
\usepackage{mathtools,bm,bbm,framed,subcaption,dsfont,booktabs}
\usepackage{enumitem}
\usepackage{stmaryrd}
\usepackage{comment}
\usepackage{ytableau}
\usepackage[margin=1in]{geometry}
\usepackage{tikz}
\usepackage{tikz-cd}

\newtheorem{theorem}{Theorem}[section]
\newtheorem{prop}[theorem]{Proposition}
\newtheorem{conj}[theorem]{Conjecture}
\newtheorem{definition}[theorem]{Definition}
\newtheorem{lem}[theorem]{Lemma}
\newtheorem{cor}[theorem]{Corollary}
\theoremstyle{definition}
\newtheorem{remark}[theorem]{Remark}
\newtheorem{example}[theorem]{Example}
\newtheorem{question}[theorem]{Question}

\newtheorem{maintheorem}{Theorem}

\numberwithin{equation}{section}

\newenvironment{verlong}{}{}

\includecomment{verlong}
\excludecomment{vershort}
\excludecomment{noncompile}

\newenvironment{private}{}{}
\excludecomment{private}

\usepackage{xargs}
\usepackage{todonotes}
\newcommand{\franco}[2][]{\todo[size=\tiny,color=ForestGreen!30,#1]{#2 \\ \hfill --- Franco}}
\newcommand{\sarah}[2][]{\todo[size=\tiny,color=teal!30,#1]{#2 \\ \hfill --- Sarah}}

\newcommand{\darij}[2][]{\todo[size=\tiny,color=cyan!30,#1]{#2 \\ \hfill --- Darij}}

\newcommand{\symm}{\mathfrak{S}}
\newcommand{\frakt}{\mathfrak{t}}
\newcommand{\syt}{\mathsf{SYT}}

\newcommand{\id}{\operatorname{id}}
\newcommand{\Des}{\operatorname{Des}}
\newcommand{\End}{\operatorname{End}}
\newcommand{\charr}{\operatorname{char}}
\newcommand{\sh}{\operatorname{sh}}
\newcommand{\content}{\mathsf{c}}

\newcommand{\qbinom}[2]{\genfrac{[}{]}{0pt}{}{#1}{#2}}
\newcommand{\ralgebra}{\mathfrak{R}}

\newcommand{\kk}{\mathbf{k}}
\newcommand{\ZZ}{\mathbb{Z}}
\newcommand{\QQ}{\mathbb{Q}}

\newcommand{\basis}{\mathfrak{B}}
\newcommand{\eigen}{\mathcal{E}}

\DeclareMathOperator{\noninv}{noninv}
\DeclareMathOperator{\aff}{aff}
\DeclareMathOperator{\im}{im}

\newcommand{\sm}{\setminus}

\newcommand{\set}[1]{\left\{ #1 \right\}}
\newcommand{\abs}[1]{\left| #1 \right|}
\newcommand{\tup}[1]{\left( #1 \right)}
\newcommand{\ive}[1]{\left[ #1 \right]}

\DeclareMathOperator{\HH}{\mathcal{H}}

\DeclareMathOperator{\B}{\mathcal{B}}
\DeclareMathOperator{\R}{\mathcal{R}}
\DeclareMathOperator{\RR}{\mathbb{R}}
\DeclareMathOperator{\Eig}{\mathrm{Eig}}

\DeclareMathOperator{\Rt}{\widetilde{\mathcal{R}}}
\DeclareMathOperator{\M}{\mathcal{M}}
\DeclareMathOperator{\X}{\mathcal{X}}

\DeclareMathOperator{\sfX}{\mathsf{X}}
\DeclareMathOperator{\sfY}{\mathsf{Y}}
\DeclareMathOperator{\sfA}{\mathsf{A}}
\DeclareMathOperator{\sfB}{\mathsf{B}}

\title{The $q$-deformed random-to-random family in the Hecke algebra}
\author{Sarah Brauner, Patricia Commins, Darij Grinberg, Franco Saliola}
\date{\today, \currenttime}


\begin{document}

\maketitle
\begin{abstract}
We generalize Reiner--Saliola--Welker's well-known but mysterious family of
\emph{$k$-random-to-random shuffles} from Markov chains on symmetric groups
to Markov chains on Type-$A$ Iwahori--Hecke algebras. We prove that these
operators pairwise commute, and that their eigenvalues are polynomials in $q$
with non-negative integer coefficients; we express these eigenvalues
in a closed form as certain symmetric function evaluations
(which is new even for $q=1$).
We show that these operators are simultaneously diagonalizable for
almost all $q$, and combinatorially express the dimension of the
subalgebra they generate.
Our work generalizes and extends work of
Reiner--Saliola--Welker and Lafreni\`ere for the symmetric group, and
simplifies all known proofs in this case.
\end{abstract}
\section{Introduction}

In this paper, we study and generalize a well-known family of Markov chains that emerge at the intersection of probability theory, algebraic combinatorics, and representation theory. This family, called the \emph{$k$-random-to-random shuffles} $\R_{n,k}(1)$, models a method of shuffling cards, wherein one removes $k$ cards (uniformly) at random from a deck of $n$ cards, and then reinserts them in new (uniform) random positions. The case that $k{\,=\,}1$, also known as \emph{random-to-random shuffling}, was introduced in the early 1990s by Diaconis and Saloff-Coste \cite{DiaconisSaloffCoste1993};
the general family was defined by Reiner--Saliola--Welker in 2014~\cite{RSW2014}.

It is the spectral properties of the $k$-random-to-random shuffles that have made them an enduring figure in combinatorics. In particular, the $\R_{n,k}(1)$ were shown to pairwise commute in \cite{RSW2014}, and conjectured in \cite[$k=1$]{UyemuraReyes2002} and \cite[$k>1$]{RSW2014} to have eigenvalues in $\ZZ_{\geq 0}$ (after suitable normalization). These conjectures were affirmed in work of Dieker--Saliola \cite[$k=1$]{DiekerSaliola} and Lafreni\`ere \cite[$k>1$]{Lafreniere}. Yet many mysteries endured: for example, Lafreni\`ere conjectured a formula for the second-largest eigenvalue of the $\R_{n,k}(1)$, but the formulas for the spectra of $\R_{n,k}(1)$ were recursive and highly inexplicit.

Here, we define a deformation of the $k$-random-to-random shuffles to the \emph{Type $A$ Iwahori--Hecke algebra} $\HH_n(q)$. These elements encode Markov chains on $\HH_n(q)$; in this setting, if $q\geq 1$ and real, we can think of $q^{-1}$ as a probability (see Remark~\ref{rmk:markovchain}). Indeed, our deformation is motivated in part by recent interest in Markov chains on Hecke algebras; see for example \cite{r2r1, ayyer2026q,diaconis2000analysis, diaconis2023double}, as well as work of Bufetov~\cite{Bufetov2020}, who showed that many interacting particle systems arising in statistical mechanics can be understood as random walks on $\HH_n(q)$. His work implies that questions of convergence of such Markov chains are well-posed; from this perspective, it is of significant interest to understand the eigenvalues of these processes.

Our main results are as follows:
\begin{maintheorem}
\label{thm:main-thm}
Let $\R_{n,k}(q)$ be the $k$-random-to-random shuffle in $\HH_n(q)$. 
\begin{enumerate}
    \item \textsf{Commutativity (Theorem~\ref{thm:commut}):} For a fixed $n$, the $\R_{n,k}(q)$ pairwise commute; \smallskip
    \item \textsf{Positive integrality (Theorem~\ref{lem:positive}):} The eigenvalues of $\R_{n,k}(q)$ are in $\ZZ_{\geq 0}[q]$; \smallskip
    \item \textsf{Closed-form expressions (Theorem~\ref{thm:big.eigenvalues}):}
    All eigenvalues of $\R_{n,k}(q)$ are indexed by horizontal strips $\lambda \sm \mu$ with $|\lambda| =n$. The eigenvalue corresponding to $\lambda \sm \mu$ is the evaluation of the symmetric function
    \[ \eigen_{\lambda \sm \mu}^{(k)}(q):= \sum_{j=0}^k e_j(\sfX) \, h_{k-j}(\sfY)\]
    at the multisets (unordered tuples)
    \begin{align*}
        \sfX:= ( q^{n-k+1}[j-i]_q: (i,j) \in \lambda \sm \mu )
        \quad \text{ and } \quad
        \sfY:= ( q^{n-k+1-\ell}[\ell]_q : |\mu| < \ell \leq n-k+1 ).
    \end{align*}
Here, $(i,j)$ indicates the row and column of a cell in $\lambda \sm \mu$, and $j-i$ is the content of that cell. \smallskip
\item \textsf{Second-largest eigenvalue (Theorem~\ref{prop:max2}):} When $q \in \RR_{>0}$, the second-largest eigenvalue of $\R_{n,k}(q)$ is 
    \[ \eigen_{(n-1,1) \sm (1,1)}^{(k)}(q)=  [k]!_q \qbinom{n-2}{k}_q \qbinom{n+1}{k}_q
    \qquad \text{ whenever $n \geq 2$ and $1 \leq k \leq n$.}
    \]
\end{enumerate}
\end{maintheorem}

\enlargethispage{2pt}

We further prove (\S~\ref{ssec.content.heckesemisimple}) that the $\R_{n,k}(q)$ are diagonalizable whenever $q$ is ``$2n$-safe'' (in characteristic $0$, this means that $q$ is neither zero nor a nontrivial root of unity of order $\leq 2n$), and construct an eigenbasis when $q \in \RR_{>0}$. See \S~\ref{ssec.charpoly} for more about eigenvalue multiplicities.

We note that Theorem~\ref{thm:main-thm} (3) gives the first closed-form, non-recursive expression for the eigenvalues of the symmetric group family $\R_{n,k}(1)$, and Theorem~\ref{thm:main-thm} (4) resolves a conjecture of  Lafreni\`ere for $q=1$ \cite{Lafreniere}. Furthermore, parts (1) and (2) in Theorem~\ref{thm:main-thm} generalize all known results about the $\R_{n,k}(1)$ in \cite{DiekerSaliola,Lafreniere,RSW2014, UyemuraReyes2002} to $\R_{n,k}(q)$ and simplify the arguments significantly by discovering a novel (but subtle) connection between the Jucys-Murphy elements and how $\R_{n,k}$ relates to $\R_{n-1,k-1}$ and $\R_{n-1,k}$ (see Theorem~\ref{thm:mainthm3}). 

Bernstein--Nestoridi proved in \cite{bernstein2019cutoff} that the random-to-random shuffles in the symmetric group exhibit remarkable cutoff behavior, and our work naturally motivates the question of whether this phenomenon is preserved by the $q$-deformation. This problem was recently solved by Arfaee--Wong in \cite{arfaee2026cutoffqdeformedclassicalcard} in the case that $k=1$ for fixed $q>1$; the case of general $k$ has not yet been studied, even when $q=1$. Our work simplifies the spectral formulas for $\R_{n,k}(q)$ significantly, thus making the general problem tractable for the first time.

While Theorem~\ref{thm:main-thm} is interesting both when $q$ is an indeterminate and in the specialization to $q=1$, there are deeper algebraic reasons why the Hecke algebra seems to be a more natural setting for the $\R_{n,k}(q)$-family than the symmetric group algebra. The formula $\eigen_{\lambda \sm \mu}^{(k)}(q)$ in Theorem~\ref{thm:main-thm} (3) involves the \emph{content} of a box in $\lambda \sm \mu$. We now consider the (multi)set of the contents of every box in $\lambda \sm \mu$:
\[ \content(\lambda \sm \mu):= ( j-i: (i,j) \in \lambda \sm \mu ). \]

Since content plays such a dominant role in the formula for $\eigen_{\lambda \sm \mu}^{(k)}$, one would hope that horizontal strips with \emph{different} content sets would yield distinct eigenvalues, particularly for the most important member of the family\footnote{We consider the $k=1$ case to ``determine'' the other $k$ because the image of $\R_{n,k}(q)$ is contained in that of $\R_{n,1}(q)$ (see \eqref{eq:Rnimagecontainment}), and the simultaneous eigenbasis for the $\R_{n,k}$ is constructed from that of $\R_{n,1}$ (see Corollary~\ref{thm:eigenbasis}). Furthermore, in the closely related case of $k$-random-to-bottom $\B_{n,k}^*(q)$ (see Definition~\ref{def:BR}), we have $\big \langle \B_{n,k}^{*}(q) \, \big \rangle_{k}  = \kk[\B_{n,1}^*(q)]$; see \cite{BCR-flags}. }, $\R_{n,1}(q)$. However, this is not necessarily the case: for example, when $q=1$, we have 
\[ \eigen_{(2,1,1) \sm (1,1)}^{(1)}(1) = \eigen_{(3,1) \sm (2,1)}^{(1)}(1) = 6,\] even though 
\[ \content((2,1,1) \sm (1,1)) = (1,-2 ) \qquad \qquad \textrm{and} \qquad \qquad \content((3,1) \sm (2,1)) = ( 2 ). \] 

Moreover, since the $\R_{n,k}$ pairwise commute, it is natural to consider the (commutative) subalgebra $\ralgebra_n(q)$ that they generate. Ideally, one would have a description of $\ralgebra_n(q)$ that makes this commutativity manifest.

We show that excluding certain values of $q$---including the case $q=1$---completely fixes the former problem, and we lay some conceptual groundwork for the latter. We say that $q$ is \emph{integral over $\ZZ$} if it is the solution to a monic polynomial equation in $\ZZ[x]$; these are the values of $q$ we would like to omit.

\begin{maintheorem}\label{thm:mainthm2}
    Suppose that $q$ is not integral over $\ZZ$.
    \begin{enumerate}
        \item \textsf{Content sets separate eigenvalues (Theorem~\ref{thm.A1}):}  For horizontal strips $\lambda \sm \mu$ and $\alpha \sm \beta$ with $|\lambda| = |\alpha|$, 
        \[ \eigen_{\lambda \sm \mu}^{(1)}(q) = \eigen_{\alpha \sm \beta}^{(1)}(q)
  \textrm{  if and only if  } \content(\lambda \sm \mu) = \content(\alpha \sm \beta). \]
    \item \textsf{$\R_{n,1}(q)$ generates the random-to-random algebra (Theorem~\ref{thm.A3}):} 
The algebra generated by the $\R_{n,k}(q)$ is precisely the polynomial algebra in $\R_{n,1}(q)$: 
    \[ \ralgebra_n(q):=\kk[ \R_{n,1}(q), \R_{n,2}(q), \ldots, \R_{n,n}(q)] = \kk[\R_{n,1}(q)].\]
    \end{enumerate}
   
\end{maintheorem}

Reiner has a long-standing open problem of finding a conceptual, non-recursive explanation for the commutativity of the $\R_{n,k}(1)$; see \cite[Problem 9]{reiner_favorite_problems}. While our proof that the $\R_{n,k}(q)$ commute (Theorem~\ref{thm:main-thm}(1)) is recursive, Theorem~\ref{thm:mainthm2}(2) suggests that the deformation to the Hecke algebra (or potentially, a further generalization) is the correct setting to answer this question.

\subsection*{Outline of Proof Strategy}
The key to the proof of both Theorems~\ref{thm:main-thm} and \ref{thm:mainthm2} is to recursively relate $\R_{n,k}$ to $\R_{n-1,k-1}$ and $\R_{n-1,k}$. We first do so purely algebraically, as in  Theorem~\ref{thm:mainthm3}(1) below. Then, with some significant work, we are able to translate this algebraic recursion to a recursion on eigenvalues, as shown in Theorem~\ref{thm:mainthm3}(2) below.

\begin{maintheorem}[]
    \label{thm:mainthm3}
For any $1 \le k \le n$, we have the following:
\begin{enumerate}
    \item \textsf{Crucial algebraic recursion (Theorem~\ref{thm:crucialrecursion}):} The following identity holds in $\HH_n(q)$: 
    \begin{equation*}
     \B_n  \R_{n, k}  = \Big(q^k \R_{n-1, k}  + \Big( \ive{n+1-k}_q + q^{n + 1 - k}J_n \Big) \R_{n - 1, k - 1} \Big)\B_n .
\end{equation*}

    \item \textsf{Crucial eigenvalue recursion (Theorem~\ref{prop:eval-recursion-1-step}, Corollary~\ref{cor:big.eigenvalues}):}
    Suppose $\lambda \sm \mu$ is a horizontal strip with $|\lambda| = n$ and that $\lambda' \supseteq \mu$ is a partition obtained by removing a box from $\lambda$. Then
        \[ \eigen_{\lambda \sm \mu}^{(k)}(q) = q^{k} \eigen_{\lambda' \sm \mu}^{(k)}(q) + [n+1-k + \content(\lambda \sm \lambda ')]_q \,  \, \eigen_{\lambda' \sm \mu}^{(k-1)}(q).\]

\end{enumerate}

\end{maintheorem}
Theorem~\ref{thm:mainthm3}(1) makes precise a subtle and important relationship between the $\R_{n,k}$ and the \emph{Jucys-Murphy elements} of the Hecke algebra, and generalizes a simpler recursion for the case $k=1$ in \cite{r2r1}. It provides the main ingredients to prove Theorem~\ref{thm:main-thm}(1), that the $\R_{n,k}$ pairwise commute.

To move from the crucial algebraic recursion to the crucial eigenvalue recursion, we must use the representation theory of $\HH_n(q)$. Recall that when $\HH_n(q)$ is semisimple, its irreducible representations---like those of the symmetric group---can be described as Specht modules $S^\lambda$ for $\lambda$ an integer partition of $n$. 

We use Theorem~\ref{thm:mainthm3}(1) to recursively construct eigenvectors for $\R_{n,k}$ by projecting to the appropriate Specht module. Specifically, let $\lambda'$ be a partition obtained from $\lambda$ by removing a box. Theorem~\ref{thm:mainthm3}(2) is proved by starting with a joint eigenvector of $\R_{n-1,k}(q)$ and $\R_{n-1,k-1}(q)$ in $S^{\lambda'}$ and applying Theorem~\ref{thm:mainthm3}(1) to lift it to an eigenvector of $\R_{n,k}(q)$ in $S^\lambda$; the corresponding eigenvalue in this process is given by the recursion in Theorem~\ref{thm:mainthm3}(2). 

One can thus understand the eigenvalue $\eigen_{\lambda \sm \mu}^{(k)}(q)$ as being built from a chain in Young's lattice on integer partitions, beginning at the partition $\mu$ and ending at the partition $\lambda$, as in Figure~\ref{fig:pathsinyoungslattice}; this path corresponds to a sequence of eigenvectors, beginning with kernel vectors of $\R_{|\mu|,k}(q)$ in $S^\mu$ and ending with eigenvectors of $\R_{|\lambda|,k}(q)$ in $S^\lambda$. (Note that it takes work to prove that all eigenvectors can be obtained in this way, and that the choice of path does not matter; this was proved in \cite{r2r1} in the analysis of $\R_{n,1}(q)$.)

\colorlet{ydgray}{gray!30}
\colorlet{ydred}{red!75!black}
\colorlet{ydorange}{orange!90!black}
\colorlet{ydgreen}{green!45!black}
\colorlet{ydblue}{blue!75!black}
\colorlet{ydmagenta}{magenta!75!black}
\colorlet{ydteal}{teal}
\colorlet{ydviolet}{violet}

\newcommand{\ydcell}{0.36}   

\tikzset{
  ydsq/.style={draw=black, line width=0.5pt, minimum size=\ydcell cm, inner sep=0pt, outer sep=0pt},
  ydlab/.style={font=\footnotesize, inner sep=2pt},
}

\newcommand{\ydshape}[5]{%
  \begin{scope}[shift={(#2)}, local bounding box=#1]
    \begin{scope}[x=\ydcell cm, y=\ydcell cm, shift={(-#3,-#4)}]
      \foreach \xx/\yy/\cc in {#5} { \node[ydsq, fill=\cc] at (\xx,\yy) {}; }
    \end{scope}
  \end{scope}}

\begin{figure}
  \centering
  \begin{tikzpicture}[>=stealth]

    \ydshape{ydB}{0,0}      {0}  {0.5}{0/0/ydgray, 0/1/ydgray}
    \ydshape{ydL}{-1.25,1.75} {0}  {1}  {0/0/white, 0/1/ydgray, 0/2/ydgray}
    \ydshape{ydR}{1.25,1.75}  {0.5}{0.5}{0/0/ydgray, 0/1/ydgray, 1/1/white}
    \ydshape{ydT}{0,3.5}    {0.5}{1}  {0/0/white, 0/1/ydgray, 0/2/ydgray, 1/2/white}

    \draw[->, line width=0.6pt, shorten <=2pt, shorten >=4.5pt] (ydB) -- (ydL);
    \draw[->, line width=0.6pt, shorten <=2pt, shorten >=4.5pt] (ydB) -- (ydR);
    \draw[->, line width=0.6pt, shorten <=4.5pt, shorten >=2pt] (ydL) -- (ydT);
    \draw[->, line width=0.6pt, shorten <=4.5pt, shorten >=2pt] (ydR) -- (ydT);

    \node[ydlab, anchor=north, yshift=-5pt] at (ydB.south) {%
      $\eigen_{(1^2)\sm (1^2)}^{(0)}(q) = {\color{ydred}1}, \quad
       \eigen_{(1^2)\sm (1^2)}^{(1)}(q) = {\color{ydorange}0}, \quad
       \eigen_{(1^2)\sm (1^2)}^{(2)}(q) = {\color{ydgreen}0}$};

    \node[ydlab, anchor=east, xshift=-3pt] at (ydL.west) {%
      \setlength{\arraycolsep}{0.75pt}%
      $\begin{array}{rclcl}
        \eigen_{(1^3)\sm (1^2)}^{(1)}(q) &=& {\color{ydblue}1} &=& q\cdot {\color{ydorange}0} + [1]_q \cdot {\color{ydred}1}\\[6pt]
        \eigen_{(1^3)\sm (1^2)}^{(2)}(q) &=& {\color{ydmagenta}0} &=& q^2\cdot {\color{ydgreen}0} + [0]_q \cdot {\color{ydorange}0}
      \end{array}$};

    \node[ydlab, anchor=west, xshift=3pt] at (ydR.east) {%
      \setlength{\arraycolsep}{0.75pt}%
      $\begin{array}{rclcl}
        \eigen_{(2,1)\sm (1^2)}^{(1)}(q) &=& {\color{ydteal}[4]_q} &=& q\cdot {\color{ydorange}0} + [4]_q \cdot {\color{ydred}1}\\[6pt]
        \eigen_{(2,1)\sm (1^2)}^{(2)}(q) &=& {\color{ydviolet}0} &=& q^2\cdot {\color{ydgreen}0} + [3]_q \cdot {\color{ydorange}0}
      \end{array}$};

    \node[ydlab, anchor=south, yshift=6pt] at (ydT.north) {%
      $\eigen_{(2,1,1)\sm (1,1)}^{(2)}(q)
        = q^2 \cdot {\color{ydmagenta}0} + [4]_q \cdot {\color{ydblue}1}
        = q^2 \cdot {\color{ydviolet}0} + [1]_q \cdot {\color{ydteal}[4]_q}
        = [4]_q$};
  \end{tikzpicture}
  \caption{To obtain the eigenvalue $\eigen_{(2,1,1) \sm (1,1)}^{(2)}(q)$, one can apply the crucial eigenvalue recursion along either path from $(1,1)$ to $(2,1,1)$ in Young's lattice.}
  \label{fig:pathsinyoungslattice}
\end{figure}

The bulk of our remaining results can be proved once Theorem~\ref{thm:mainthm3}(2) is obtained. In particular, the closed-form expression for $\eigen_{\lambda \sm \mu}^{(k)}(q)$ in Theorem~\ref{thm:main-thm}(3) follows by realizing---somewhat miraculously---that the supersymmetric functions 
\[ e_k( \mathsf{S} / \mathsf{T}) = \sum_{j=0}^k e_{j}(\mathsf{S}) \, h_{k-j}(\mathsf{T})\]
evaluated in any multisets $\mathsf{S}, \mathsf{T}$ satisfy an analogous recursion. Likewise, one can prove the positivity of the $\eigen_{\lambda \sm \mu}^{(k)}(q)$ in Theorem~\ref{thm:mainthm3}(2) by arguing that at each step in the eigenvalue recursion, the contributions remain in $\ZZ_{ \geq 0}[q]$. Theorem~\ref{thm:mainthm2} follows from further analysis of the crucial eigenvalue recursion.

\subsection*{Organization of the paper} The remainder of the paper proceeds as follows:
\begin{itemize}[topsep=1ex, itemsep=1ex]
    \item In Section~\ref{section.background}, we review the definition
    of the $k$-random-to-random shuffles in the group algebra of the symmetric group and the Hecke algebra. 
    \item  Section~\ref{section.recursion} introduces and subsequently proves the main algebraic recursion (Theorem~\ref{thm:mainthm3}(1)).
    This lays the groundwork for the proofs in the remaining sections.
    \item We then
    prove the commutativity of the $\R_{n,k}$ in Section~\ref{section.commutativity} (Theorem~\ref{thm:main-thm}(1)).
    \item In Section~\ref{section.eigenvectorsandkernels}, we move to representation theory, and use Theorem~\ref{thm:mainthm3}(1) to prove Theorem~\ref{thm:mainthm3}(2), the crucial eigenvalue recursion. Along the way, we compute the kernel of the $\R_{n,k}(q)$ in terms of the combinatorics of \emph{derangements} and \emph{desarrangement tableaux} (see \S~\ref{ssec.kernelofrnk}).
 We build on results from \cite{r2r1} to compute an eigenbasis for the $\R_{n,k}(q)$ when $q$ is a positive real number (Corollary~\ref{thm:eigenbasis}).
    \item Section~\ref{sec.eigenvalues} studies the eigenvalues of the $\R_{n,k}$. We prove the eigenvalue formula in Theorem~\ref{thm:main-thm}(3) in \S~\ref{ssec.eigenvalueexpression}, the positivity of the eigenvalues from Theorem~\ref{thm:main-thm}(2) in \S~\ref{ssec.positivity}, and derive multiplicative eigenvalue formulas for the special case when $\lambda$ is a hook in \S~\ref{ssec.hooks}, including the second-largest eigenvalue formula in Theorem~\ref{thm:main-thm}(4). We also compute the characteristic polynomial from the action of $\R_{n,k}(q)$ on $\HH_n$ in \S~\ref{ssec.charpoly} (Corollary~\ref{cor:charpoly}).

    \item Finally in Section~\ref{section.ralgebra} we study the random-to-random algebra $\ralgebra_n(q)$ and prove Theorem~\ref{thm:mainthm2}.
    
    \item 
       The polynomial expansion of the nonzero eigenvalues of $\R_{n, k}$ for $3 \leq n \leq 5$ is included in Appendix~\ref{sec:Appendix-Raw-Data}.
\end{itemize}

\subsection*{Notation}
Henceforth, to simplify notation, we omit ``$(q)$'' for nonspecific $q$, writing
$\HH_n$ for $\HH_n(q)$, $\R_{n,k}$ for $\R_{n,k}(q)$, $\eigen_{\lambda \sm \mu}^{(k)}$ for $\eigen_{\lambda \sm \mu}^{(k)}(q)$, and so on.
The specialization to the $q=1$ case will be denoted, for example, by $\R_{n,k}(1)$.

\section{Background and key definitions}  \label{section.background}

In this section, we define the $\R_{n,k}(1)$ from the introduction in \S~\ref{section:symmetricgroup}, introduce the Hecke algebra in \S~\ref{section:HHbackground}, and then define the $\R_{n,k}(q)$ in \S~\ref{section.shufflinginhecke}.

We note that the definitions needed for our actual work begin in \S~\ref{section:HHbackground}.
A reader not interested in card shuffling and the $q=1$ background can thus start reading from \S~\ref{section:HHbackground}.

\subsection{$k$-random-to-random in the symmetric group}\label{section:symmetricgroup}
Fix a nonnegative integer $n$, and consider the symmetric group $\symm_n$.
Our story begins with one of the best-studied Markov chains on the symmetric group, called 
\emph{random-to-bottom}, given by
right multiplication in $\ZZ\ive{\symm_n}$ by the element
\begin{align}
    \B^\ast_n(1)
    &= \sum_{i=1}^n \tup{i, i+1, \ldots, n} \qquad \tup{\text{in cycle notation}} \label{eq:Bncycle} \\
     &= \sum_{i=1}^n s_{i} s_{i+1} \cdots s_{n-2} s_{n-1} \qquad \tup{\text{using generators $s_i = (i,i+1)$}}. \label{eq:Bngens}
\end{align}

In probabilistic language, if we view permutations in one-line notation as
decks of $n$ cards labeled $1,2,\ldots,n$, then $\B^\ast_n(1)$
corresponds to picking a card at random and moving it to the bottom of the
deck, where the bottom card of the deck is the $n$-th position.

The \emph{random-to-random
shuffling element} is obtained by composing this element $\B^\ast_n(1)$ with the
\emph{bottom-to-random shuffling element} $\B_n(1)$:
\begin{equation*}
    \R_n(1) = \B^\ast_n(1) \B_n(1)
    \qquad
    \text{with}
    \qquad
    \B_n(1)
    = \sum_{i=1}^n \tup{n, n-1, \ldots, i+1, i}.
\end{equation*}

Reiner, Saliola and Welker introduced in their monograph \cite{RSW2014} a family of elements they called \emph{$k$-random-to-random shuffling elements} $\R_{n,k}(1) \in \ZZ[\symm_n]$. This family included $\R_n(1) =: \R_{n,1}(1)$, and each element $\R_{n,k}(1)$ can be understood as a
(rescaled) shuffling process as follows.
Again, associate a deck of $n$ cards with a
word $w$ of length $n$ with letters $1, 2, \ldots, n$. Right multiplication by $\R_{n,k}(1)$ acts on $w$
by randomly selecting $k$ distinct letters $w_{i_1}, \ldots, w_{i_k}$ and placing them
in new random positions. More precisely,
\[
\R_{n,k}(1) = \tfrac{1}{k!} \B^\ast_{n,k}(1) \B_{n,k}(1) \in \ZZ\ive{\symm_n},
\]
where the actions of $\B_{n,k}^*(1)$ and $\B_{n,k}(1)$ are defined as follows:
\begin{enumerate}[itemsep=1ex, after=\smallskip]
    \item $\B^\ast_{n,k}(1)$ selects $k$ letters from a word $w$ and moves them to the end of the word;
    \item $\B_{n,k}(1)$ moves the last $k$ letters in a word $u$ to new, uniformly random positions.
\end{enumerate}

\begin{example}\rm
    Consider $n = 4$ and $k = 2$. Using one-line notation, $\B_{4,2}^\ast(1)$ acts on $1234$ by
    \begin{align*}
       1234 \cdot \B^\ast_{4,2}(1) 
        = 34{\color{red}12} &+ 34{\color{red}21} + 24{\color{red}13} + 24{\color{red}31} + 23{\color{red}14} + 23{\color{red}41}\\
        &+ 14{\color{red}23} + 14{\color{red}32} + 13{\color{red}24} + 13{\color{red}42} + 12{\color{red}34} + 12{\color{red}43} ,
    \end{align*}
    which is the sum of all permutations that can be obtained from $1234$ by
    moving two letters to the end of the word; the red color indicates the
    moved letters.
    Similarly,
    \begin{align*}
      1234\cdot  \B_{4,2}(1) 
        = 12{\color{red}3}{\color{red}4} + 12{\color{red}4}{\color{red}3} &+ 1{\color{red}3}2{\color{red}4} + 1{\color{red}3}{\color{red}4}2 + 1{\color{red}4}2{\color{red}3} + 1{\color{red}4}{\color{red}3}2
        \\&+ {\color{red}3}12{\color{red}4} + {\color{red}3}1{\color{red}4}2 + {\color{red}3}{\color{red}4}12 + {\color{red}4}12{\color{red}3} + {\color{red}4}1{\color{red}3}2 + {\color{red}4}{\color{red}3}12
    \end{align*}
    is the sum of all permutations obtained from $1234$ by moving
    its last two letters.

Suppose we are interested in the coefficient of $2314$ in the image of $\R_{4,2}(1)$ acting on $1234$.
There are eight ways (shown below) to obtain $2314$ by first moving
    two letters from $1234$ to the end of the word, and then moving those
    letters to new positions:
    {
        \def\r#1{{\color{red}#1}}
        \begin{align*}
            \r{1}23\r{4} \xrightarrow{}  23\r{1}\r{4} \xrightarrow{} 23\r{1}\r{4} \hspace{1cm} & \hspace{1cm} \r{1}23\r{4} \xrightarrow{} 23\r{4}\r{1} \xrightarrow{} 23\r{1}\r{4} \\
            \r{1}2\r{3}4 \xrightarrow{}  24\r{1}\r{3} \xrightarrow{} 2\r{3}\r{1}4 \hspace{1cm} & \hspace{1cm} \r{1}2\r{3}4 \xrightarrow{} 24\r{3}\r{1} \xrightarrow{} 2\r{3}\r{1}4 \\
            \r{1}\r{2}34 \xrightarrow{}  34\r{1}\r{2} \xrightarrow{} \r{2}3\r{1}4 \hspace{1cm} & \hspace{1cm} \r{1}\r{2}34 \xrightarrow{} 34\r{2}\r{1} \xrightarrow{} \r{2}3\r{1}4 \\
            1\r{2}\r{3}4 \xrightarrow{}  14\r{2}\r{3} \xrightarrow{} \r{2}\r{3}14 \hspace{1cm} & \hspace{1cm} 1\r{2}\r{3}4 \xrightarrow{} 14\r{3}\r{2} \xrightarrow{} \r{2}\r{3}14
        \end{align*}
    }
Hence the coefficient of $2314$ in $1234 \cdot \R_{4,2}(1)$ is $8/2! = 4$.
\end{example}

\begin{remark}
To obtain a true Markov chain, one normalizes by $k! \binom{n}{k}^2$. In other words, writing 
\[ \widetilde{\R}_{n,k}(1):=  \frac{1}{k! \binom{n}{k}^2 } \R_{n,k}(1) = \frac{1}{\left(k! \binom{n}{k}\right)^2} \B_{n,k}^*(1) \B_{n,k}(1), \]
we have that for $\sigma, \tau \in \symm_n$, 
\[ [ \sigma] \big( \tau \cdot \widetilde{\R}_{n,k}(1) \big)  = p_{\tau, \sigma}, \]
where $p_{\tau, \sigma}$ is the probability of obtaining $\sigma$ from $\tau$ after one application of the shuffling process. 

Normalizing $\R_{n,k}(1)$ by $k! \binom{n}{k}^2$ simply scales its eigenvalues, so that the eigenvalues of $\widetilde{\R
}_{n,k}(1)$ lie in $[0,1]$ by construction, and the largest eigenvalue is 1. Thus in all statements about the eigenvalues of $\R_{n,k}(1)$, one can immediately recover the same information for $\widetilde{\R
}_{n,k}(1)$.
\end{remark}
As a probabilist, one is interested in the spectra of the $\R_{n,k}(1)$; but the following remarkable properties make them interesting also to combinatorialists. 

\begin{theorem}\label{thm:q=1} Let $0 \leq k \leq n$.
\begin{enumerate}
    \item (Dieker--Saliola \cite[$k=1$]{DiekerSaliola}, Lafreni\`ere \cite[$k>1$]{Lafreniere}) The eigenvalues of $\R_{n,k}(1)$ are in $\ZZ_{\geq 0}$.
    \item (Reiner--Saliola--Welker \cite{RSW2014}, Lafreni\`ere \cite{Lafreniere}) For a fixed $n$, the $\R_{n,k}(1)$ pairwise commute.
\end{enumerate}
\end{theorem}

Theorem~\ref{thm:q=1}~(1) was conjectured as early as 2002 by Uyemura-Reyes for $k=1$ \cite{UyemuraReyes2002}; this conjecture was extended in 2014 by Reiner--Saliola--Welker \cite{RSW2014}. Its subsequent proof used the representation theory of the symmetric group. 
All these results are recovered and generalized by our present work (Theorem~\ref{lem:positive} and Theorem~\ref{thm:commut}).

In fact, the $\B_{n,k}(1)$ are examples of a more general family of shuffles introduced by Bidigare, Hanlon and Rockmore in their seminal work \cite{BHR}, which
studies the theory of random walks on hyperplane arrangements. Such random walks are thus known as \emph{BHR random walks}.
This study was developed further by Brown and Diaconis \cite{Brown2000,BrownDiaconis1998}.

The $\R_{n,k}(1)$, too, are instances of a more general construction introduced by Reiner, Saliola and Welker called \emph{symmetrized shuffling operators}.
These operators are obtained by composing a BHR random walk
with its transpose. Reiner--Saliola--Welker introduced one such
operator for every integer partition $\lambda \vdash n$, each
corresponding to a Markov chain on the symmetric group $\symm_n$. Despite 
the monograph devoted to their study \cite{RSW2014}, many questions about symmetrized shuffling operators remain. 

For instance, \cite{RSW2014} identified just \emph{two} families of partitions
giving rise to pairwise commuting elements in $\RR[\symm_n]$. 
The first family is precisely the $k$-random-to-random shuffling elements $ \R_{n,k}(1)$
which are the focus of our study---and generalization---here. (The element $\R_{n,k}(1)$ corresponds to the partition $(n-k,1^k)$ in \cite{RSW2014}.)
A forthcoming paper will resolve conjectures regarding the second family (corresponding to the family of partitions $\{ (2^{k},1^{n-2k}) \}$).
It is suspected, but still unproved (\cite[Conjecture 1.7]{RSW2014}), that there are no further nontrivial commutativities between symmetrized shuffling operators.

\subsection{The Hecke algebra $\HH_n$}
\label{section:HHbackground}
We will deform the $k$-random-to-random shuffles to the Hecke algebra, whose theory we briefly review below.

Let $n$ be a nonnegative integer.
The symmetric group $\symm_n$ is the group of all permutations of the set $\ive{n} := \set{1,2,\ldots,n}$, and is generated by the simple transpositions $s_i = \left(i,i+1\right)$ for all $i \in \ive{n-1}$.
We will use cycle notation for permutations unless declared otherwise.
Permutations are multiplied by the rule $\tup{\alpha \beta}\tup{i} = \alpha\tup{\beta\tup{i}}$ for all $\alpha, \beta \in \symm_n$ and $i \in \ive{n}$.

Let $\kk$ be a field (a commutative ring will suffice for Sections \ref{section.background}--\ref{section.commutativity}). We fix $q \in \kk$.
The \emph{Hecke algebra} $\HH_n$ is the (unital, associative) $\kk$-algebra generated by elements $T_1, T_2, \ldots, T_{n-1}$ subject to the relations
\begin{align}
T_i^2 &= \tup{q-1} T_i + q \qquad \text{for all } i \in \ive{n-1};
\label{eq:Hndef:1} \\
T_i T_j &= T_j T_i \qquad \text{whenever } \abs{i-j} > 1;
\label{eq:Hndef:2} \\
T_i T_{i+1} T_i &= T_{i+1} T_i T_{i+1} \qquad \text{for all } i \in \ive{n-2}.
\label{eq:Hndef:3}
\end{align}
As a $\kk$-vector space, $\HH_n$ has a basis $\tup{T_w}_{w \in \symm_n}$, defined by the formula
\begin{align}
T_w &= T_{i_1} T_{i_2} \cdots T_{i_k}
\label{eq:Hndef:Tw=}
\end{align}
if $s_{i_1} s_{i_2} \cdots s_{i_k}$ is a reduced expression\footnote{A \emph{reduced expression} for a permutation $w \in \symm_n$ means a way of writing $w$ as a product of the form $s_{i_1} s_{i_2} \cdots s_{i_k}$ with minimum $k$. This minimum $k$ is called the \emph{(Coxeter) length} of $w$, and equals the number of inversions of $w$.} for $w$.
In particular, $T_{\id} = 1$ and $T_{s_i} = T_i$ for each $i \in \ive{n-1}$.
We refer to \cite{Mathas} and \cite[Section 7.4]{Humphreys} for more in-depth introductions to Hecke algebras.

There is a unique algebra anti-automorphism of $\HH_n$
that sends each basis vector $T_{w}$ to $T_{w^{-1}}$ and thus, in
particular, preserves all the generators $T_{1},T_{2},\ldots ,T_{n-1}$.
Following \cite[Lemma 2.3]{Murphy95}, we
shall denote this anti-automorphism by $a\mapsto a^\ast$.
Moreover, the map
$a\mapsto a^\ast$ is an involution, i.e., we have $a^{\ast \ast
}=a$.

In what follows, we will often assume that $q$ is an indeterminate\footnote{That is, $\kk$ will be a field of rational functions $\mathbf{m}\left(q\right)$ over a smaller field $\mathbf{m}$.}. However, for any characteristic polynomial formula or polynomial identity, we can always specialize our indeterminate $q$ to any specific value in $\kk$. 
If we impose special conditions on $q$, we will always state it explicitly. For the specialization $q = 1$, the basis elements $T_w$ of $\HH_n$ become the permutations $w \in \symm_n$ in the group algebra $\kk\left[\symm_n\right]$.
When we specialize an element $a \in \HH_n$ by setting $q = 1$, we will denote the resulting element by $a\tup{1}$, to distinguish it from the general case.

We will use the $q$-integers 
\[ [k]_q = \tup{1-q^k} / \tup{1-q} \qquad \text{ for all } k \in \ZZ, \]
and the $q$-factorials
\[
\ive{k}!_q = \ive{1}_q \ive{2}_q \cdots \ive{k}_q
\qquad \text{ for all } k \geq 0.
\]
Note that if $k \geq 0$, then $\ive{k}_q = q^0 + q^1 + \cdots + q^{k-1}$ is a polynomial in $q$, whereas for negative $k$, the element $[k]_q$ is no longer a polynomial, but rather a Laurent polynomial in $q$, and explicitly equals $-q^{-1} - q^{-2} - \cdots - q^k$.

We implicitly consider the Hecke algebras of all symmetric groups to form a chain of subalgebras:
\[
\HH_0 \subseteq \HH_1 \subseteq \HH_2 \subseteq \cdots
\]
by identifying homonymous generators $T_i$ of different Hecke algebras.

Like the symmetric group algebra, the Hecke algebra contains a special family of elements called the \emph{Young--Jucys--Murphy elements} $J_1, J_2, J_3, \ldots$ in the respective Hecke algebras $\HH_1, \HH_2, \HH_3, \ldots$ defined when $q\neq 0$ by
\begin{equation}\label{eq:JMelementdef}
J_1 = 0, \qquad \qquad \textrm{and} \qquad \qquad J_k = \sum_{i=1}^{k-1} q^{i-k} T_{(i,k)} \qquad \text{ for all } k \geq 1,
\end{equation}
where $(i,k)$ denotes the transposition that swaps $i$ with $k$.
For instance, $J_3 = q^{-2} T_{(1,3)} + q^{-1} T_{(2,3)}$.

The anti-automorphism $a\mapsto a^\ast$ of the
Hecke algebra $\HH_n$ preserves $T_w$ for
each transposition $w$ (since $w^{-1} = w$), and thus also
preserves the Young--Jucys--Murphy elements, i.e., we have
\begin{equation}
\left( J_k \right) ^\ast = J_k
\qquad \text{for all } k\geq 1.
\label{eq:star:J}
\end{equation}

\subsection{The $k$-random-to-random shuffles in $\HH_n$}\label{section.shufflinginhecke}

We now describe how to deform the $k$-random-to-bottom and $k$-random-to-random shuffles in \S~\ref{section:symmetricgroup} to the Hecke algebra. 

We begin with the simplest elements: random-to-bottom $\B_n^*(1):= \B_{n,1}^*(1)$ and bottom-to-random $\B_n(1):= \B_{n,1}(1)$. Recall the definition of $\B_{n}^*(1)$ in  \eqref{eq:Bngens} in terms of simple transpositions $s_i$. Our definition of $\B_{n}^*$ simply replaces each $s_i$ with a $T_i$ (and likewise for $\B_n$).

\begin{definition}
\label{def.B-B*}
For any $n \ge 0$, we define the elements
\begin{align}
\B_n &= \sum_{i=1}^n T_{n-1} T_{n-2} \cdots T_i \qquad \text{ and}
\label{eq.Bn=}\\
\B^*_n &= \big(\B_n\big)^\ast = \sum_{i=1}^n T_i T_{i+1} \cdots T_{n-1}
\label{eq.B*n=}
\end{align}
in the Hecke algebra $\HH_n$.
We understand $\B_n$ and $\B^*_n$ to be $0$ (like any empty sum) whenever $n \le 0$.
\end{definition}

\begin{example}
Suppose $n \geq 4$. Using cycle notation for permutations, we have:
\begin{align*}
    \B_4 &= 1 + T_3 + T_3 T_2 + T_3 T_2 T_1 = T_{\id} + T_{(4,3)} + T_{(4,3,2)} + T_{(4,3,2,1)}; \\
    \B^*_4 &= 1 + T_3 + T_2 T_3 + T_1 T_2 T_3 = T_{\id} + T_{(3,4)} + T_{(2,3,4)} + T_{(1,2,3,4)}.
\end{align*}
\end{example}

Next, we define elements $\B_{n,k}$, $\B^*_{n,k}$ and $\R_{n,k}$ that are called (respectively) the \emph{$k$-bottom-to-random shuffle}, the \emph{$k$-random-to-bottom shuffle} and the \emph{$k$-random-to-random shuffle} in the Hecke algebra $\HH_n$
(although the names are swapped in some of the literature):

\begin{definition}
\label{def:BR}
For any $n, k \ge 0$, we define the following elements of $\HH_n$:
\begin{align}
        \B_{n,k}  & := \B_{n-(k-1)}  \cdots \B_{n-1}  \B_n \qquad \text{and}
		\label{eq:Bnk:def} \\
        \B^*_{n,k}  & := \big(\B_{n, k}\big)^\ast = \B^*_n  \B^*_{n-1}  \cdots \B^*_{n-(k-1)} \qquad \text{and}
		\label{eq:Bastnk:def} \\
		\R_{n,k} & := \frac{1}{\ive{k}!_q} \B^*_{n,k} \B_{n,k}.
        \label{eq:Rnk:def}
\end{align}
\end{definition}
Note that, as the notation suggests, $(\B_{n,k})^* = \B_{n,k}^*$. Thus 
\begin{equation}
    (\R_{n,k})^* =  \left( \frac{1}{\ive{k}!_q} \B^*_{n,k} \B_{n,k} \right)^* = \frac{1}{\ive{k}!_q} (\B_{n,k})^* (\B_{n,k}^*)^* = \frac{1}{\ive{k}!_q} \B_{n,k}^* \B_{n,k} =  \R_{n,k}. \label{eq:star:R}
\end{equation}

We will soon see that the product $\B^\ast_{n,k} \B_{n,k}$ is a multiple of $[k]!_q$ (Lemma~\ref{lem:divisible}),
which implies that $\R_{n,k}$ is well-defined for any value of $q$, even if the denominator in \eqref{eq:Rnk:def} might vanish; in this case, we expand $\R_{n,k}$ for an indeterminate $q$ and then substitute our desired value of $q$
(or simply take the formula \eqref{eq:uI:Rnk=} from Lemma~\ref{lem:divisible} as its definition).

Both $\B_{n,0}$ and $\B^*_{n,0}$ equal $1$, being empty products.
Hence, $\R_{n,0} = 1$ as well.
Note that $\B_{n,1} = \B_n$ and $\B^*_{n,1} = \B^*_n$.
On the other hand, recall that $\B_m$ and $\B^*_m$ are $0$ whenever $m \le 0$; thus, Definition~\ref{def:BR} yields that
\begin{equation}
\B_{n,k} = \B^*_{n,k} = 0 \qquad \text{ for all } k > n,
\label{eq:Bnk:zero}
\end{equation}
and consequently
\begin{equation}
\R_{n,k} = 0 \qquad \text{ for all } k > n.
\label{eq:Rnk:zero}
\end{equation}

Note that our notations agree with those in \cite{r2r1}, but
we omit the ``$(q)$'' from notations like ``$\HH_n(q)$'' and ``$\B_n(q)$'', and we abbreviate $T_{s_i}$ by $T_i$. Further,
$\R_{n,1}$ is the $\mathcal{R}_n(q)$ from \cite{r2r1}.

The \emph{$k$-top-to-random shuffle} and \emph{$k$-random-to-top shuffle} are obtained from $\B_{n,k}$ and $\B^*_{n,k}$ by applying the involutive algebra endomorphism of $\HH_n$ that sends each $T_i$ to $T_{n-i}$.

\begin{example}
For $n = 3$, we have (using cycle notation for permutations)
\begin{align*}
\B_{3,0} &= \B^*_{3,0} = \R_{3,0} = 1
\end{align*}
and
\begin{align*}
\B_{3,1} &= \B_3 = 1 + T_2 + T_2 T_1 = T_{\id} + T_{(3,2)} + T_{(3,2,1)}; \\
\B^*_{3,1} &= \B^*_3 = 1 + T_2 + T_1 T_2 = T_{\id} + T_{(2,3)} + T_{(1,2,3)}; \\
\R_{3,1} &= \dfrac{1}{[1]!_q} \B^*_{3,1} \B_{3,1}
= [3]_q T_{\id} + q[2]_q T_{(1,2)} + [2]_q T_{(2,3)} + q T_{(1,2,3)} + q T_{(1,3,2)} + \tup{q-1} T_{(1,3)}
\end{align*}
and
\begin{align*}
\B_{3,2} &= \B_2 \B_3 = \sum_{w \in \symm_3} T_w; \qquad \qquad
\B^*_{3,2} = \B^*_3 \B^*_2 = \sum_{w \in \symm_3} T_w; \qquad \qquad
\R_{3,2} = [3]_q
\sum_{w \in \symm_3} T_w
\end{align*}
and
\begin{align*}
\B_{3,3} &= \B^*_{3,3} = \sum_{w \in \symm_3} T_w; \qquad \qquad
\R_{3,3} = \sum_{w \in \symm_3} T_w.
\end{align*}
Note that the $\tup{q-1} T_{(1,3)}$ term in $\R_{3,1}$ is invisible for $q=1$.
\end{example}

\begin{example}
\label{exa:R42}
For $n = 4$ and $k = 2$, we have
\begin{equation*}
    \begin{aligned}
        \R_{4, 2}
        & = \left(q^{4} + q^{3} + 2 q^{2} + q + 1\right) T_{\id}
        + \left(q^{3} + 2 q^{2} + q + 1\right) T_{(3,4)}
        + \left(q^{4} + q^{3} + q^{2} + q + 1\right) T_{(2,3)}
        \\ &
        + \left(q^{3} + q^{2} + q + 1\right) T_{(2,3,4)}
        + \left(q^{3} + q^{2} + q + 1\right) T_{(2,4,3)}
        + \left(q^{3} + q + 1\right) T_{(2,4)}
        \\ &
        + \left(q^{4} + q^{3} + 2 q^{2} + q\right) T_{(1,2)}
        + \left(q^{3} + 2 q^{2} + q\right) T_{(1,2)(3,4)}
        + \left(q^{4} + q^{3} + q^{2} + q\right) T_{(1,2,3)}
        \\ &
        + \left(q^{3} + q^{2} + q\right) T_{(1,2,3,4)}
        + \left(q^{3} + q^{2} + q\right) T_{(1,2,4,3)}
        + \left(q^{3} + q\right) T_{(1,2,4)}
        \\ &
        + \left(q^{4} + q^{3} + q^{2} + q\right) T_{(1,3,2)}
        + \left(q^{3} + q^{2} + q\right) T_{(1,3,4,2)}
        + \left(q^{4} + q^{3} + q^{2} + q - 1\right) T_{(1,3)}
        \\ &
        + \left(q^{3} + q^{2} + q - 1\right) T_{(1,3,4)}
        + \left(q^{3} + q\right) T_{(1,3)(2,4)}
        + \left(q^{3} + q - 1\right) T_{(1,3,2,4)}
        \\ &
        + \left(q^{3} + q^{2} + q\right) T_{(1,4,3,2)}
        + \left(q^{3} + q\right) T_{(1,4,2)}
        + \left(q^{3} + q^{2} + q - 1\right) T_{(1,4,3)}
        \\ &
        + \left(q^{3} + q - 1\right) T_{(1,4)}
        + \left(q^{3} + q - 1\right) T_{(1,4,2,3)}
        + \left(q^{3} + q - 2\right) T_{(1,4)(2,3)}.
    \end{aligned}
\end{equation*}
Note that all $24$ basis elements $T_w$ of $\HH_n$ appear
in this sum.
However,
the coefficient $q^3 + q - 2$ in front of $T_{(1,4)(2,3)}$
vanishes at $q = 1$,
so that $\R_{4,2}\tup{1}$ contains only $23$ permutations
$w \in \symm_n$.
\end{example}

\begin{remark}\label{rmk:markovchain}
The reader interested in probability may be confused about how the element $\R_{n,k}$ is in fact a Markov chain on $\HH_n$ for $0 \leq k \leq n$. To realize it as such, there are several steps:
\begin{enumerate}
    \item First, as in the case of the shuffles in the symmetric group algebra, we think of the element $\R_{n,k}$ as acting on $\HH_n$ by right multiplication. (Equivalently, since $\HH_n$ has a linear basis indexed by permutations, one could think of this as an action on $\ZZ[q^{\pm}][\symm_n]$).
    \item Second, one needs the multiplication action on $\HH_n$ to be stochastic, i.e. to have coefficients that sum to 1. To achieve this, we utilize a rescaling defined in \cite{diaconis2000analysis}: fix $q \geq 1$ and real so that $0  < q^{-1} \leq 1$. Define 
    \[ \widetilde{T}_i:= q^{-1}T_i, \]
    so that $\widetilde{T}_w = q^{-\ell(w)} T_w$. It then follows that 
    \begin{equation}
    \widetilde{T}_w \cdot \widetilde{T}_i = \begin{cases}
    \widetilde{T}_{w \cdot s_i} & \text{if } \ell(w \cdot s_i) > \ell(w); \\
    q^{-1} \widetilde{T}_{w \cdot s_i} + (1-q^{-1}) \widetilde{T}_{w} & \text{if } \ell(w \cdot s_i) < \ell(w).
    \end{cases}
    \end{equation}
    Hence, this defines a stochastic process whenever $q^{-1}$ is a probability, i.e. for $q \in \RR_{\geq 1}$. Thus, to realize $\R_{n,k}$ as a Markov chain, one takes $q \in \RR_{\geq 1}$ and substitutes each instance of $T_w$ with $q^{\ell(w)} \widetilde{T}_w$. 
    \item Finally, as in the case of the symmetric group, we must normalize $\R_{n,k}$ by an appropriate scalar so that its coefficients sum to 1. Once $\R_{n,k}$ has been written in terms of the $\widetilde{T}_i$ generators, it is not hard to deduce (see Corollary~\ref{prop:evals-n}) that its coefficients sum to 
    \[ \qbinom{n}{k}^2 [k]!_q.\]
    We thus define the Markov chain for any real $q \geq 1$ and $n \geq k \geq 0$ by
    \[ \widetilde{\R}_{n,k}:= \frac{1}{\qbinom{n}{k}^2 [k]!_q} \R_{n,k} = \frac{1}{\big( \qbinom{n}{k} [k]!_q \big)^2} \B_{n,k}^* \B_{n,k}. \]
Once again, the eigenvalues of $\widetilde{\R}_{n,k}$ are precisely the normalized eigenvalues of $\R_{n,k}$.
\end{enumerate}
\end{remark}

\begin{remark}\label{rmk:hepworth}
An equivalent form of $\B^*_n$ has appeared in Hepworth's work \cite[\S 6]{Hepworth2022} on homological stability for Hecke algebras.
He constructed a chain complex $\mathcal D(n)$ whose top two
degrees are isomorphic copies of $\HH_n(q)$; its top differential
$\partial^{n-1}$ is right multiplication by
$\Psi(\B_n^*(q^{-1}))$, where
$\Psi : \HH_n(q^{-1}) \to \HH_n(q)$ is the $\kk$-algebra
isomorphism sending each generator $T_i$ to $-q^{-1}T_{n-i}$.
We equip $\HH_n(q)$ with the symmetric bilinear form
$\langle T_u,T_v\rangle=\delta_{u,v}q^{\ell(u)}$
(which has the property that right multiplication by an
element $a \in \HH_n(q)$ is adjoint to right multiplication
by $a^*$), and equip $\HH_n(q^{-1})$ with the analogous form
(using $q^{-1}$ instead of $q$); then, $\Psi$ is an isometry,
and thus the top Laplacian of $\mathcal D(n)$ is
right multiplication by $\Psi(\R_{n,1}(q^{-1}))$.
This interpretation of random-to-random shuffling was known for $q=1$ since the work of Hanlon and Hersh \cite{HanlonHersh2004}.
\end{remark}

\subsubsection{Alternative descriptions of the shuffles}
We now give an alternative description of the $\R_{n,k}$ that elucidates certain properties that will be useful in proving Theorem~\ref{thm:mainthm3}(1), in Section~\ref{section.recursion}. We first note the elements $\B_n$ and $\B^*_n$ can be rewritten in more combinatorial terms, as proved in \cite[Lemma 3.12]{r2r1}.
\begin{prop}\cite{r2r1}
\label{prop:rewr}
Let $n\ge 1$. Then,
\begin{align}
\B_n &= \sum_{i=1}^n T_{(n, n-1, \ldots, i)} = \sum_{\substack{w \in \symm_n;\\ w^{-1}\left(1\right) < w^{-1}\left(2\right) < \cdots < w^{-1}\left(n-1\right)}} T_{w};
\label{eq:prop:rewr:B} \\
\B^*_n &= \sum_{i=1}^n T_{(i, i+1, \ldots, n)} = \sum_{\substack{w \in \symm_n;\\ w\left(1\right) < w\left(2\right) < \cdots < w\left(n-1\right)}} T_{w},
\label{eq:prop:rewr:Bast}
\end{align}
where $(n, n-1, \ldots, i)$ and $(i, i+1, \ldots, n)$ are understood as cycles.
\end{prop}

\begin{remark}
More generally, it can be shown that for all $n \geq k$, we have
\begin{align}
\B_{n,k} &= \sum_{\substack{w \in \symm_n;\\ w^{-1}\left(1\right) < w^{-1}\left(2\right) < \cdots < w^{-1}\left(n-k\right)}} T_{w};
\\
\B^*_{n,k} &= \sum_{\substack{w \in \symm_n;\\ w\left(1\right) < w\left(2\right) < \cdots < w\left(n-k\right)}} T_{w}.
\end{align}
Indeed, the former equality is \cite[Lemma 3.12]{r2r1} (applied to $j = n-k$), whereas the latter follows by applying the anti-isomorphism $*$.
For $q = 1$, this implies that the element $\B^*_{n,k}(1)$ lies in the subalgebra of the symmetric group algebra known as \emph{Solomon's descent algebra} \cite{schocker-descent}.
However, the descent algebra does not seem to generalize to the Hecke algebra with general $q$ (the relevant basis elements no longer span a subalgebra, even for $n = 3$), rendering any use of the descent algebra a dead end at our level of generality.
\end{remark}

For each subset $I$ of $\ive{n}$, let $u_I \in \symm_n$ be the unique permutation that
\begin{enumerate}
    \item is strictly increasing on $I$ and on $\ive{n} \setminus I$, and
    \item satisfies $u_I\tup{I} = \left\{1,2,\ldots,\abs{I}\right\}$.
\end{enumerate}
Thus, $u_I$ sends the elements of $I$ in increasing order to the numbers $1,2,\ldots,\abs{I}$, and sends the remaining elements of $\ive{n}$ in increasing order to the remaining numbers $\abs{I}+1,\abs{I}+2,\ldots,n$.

Given any $i \in \set{0,1,\ldots,n}$, we write $\symm_{\ive{n}-\ive{i}}$ for the subgroup of $\symm_n$ consisting of all permutations $w \in \symm_n$ that fix the numbers $1,2,\ldots,i$.
This subgroup is generated by $s_{i+1},s_{i+2},\ldots,s_{n-1}$, and is isomorphic to $\symm_{n-i}$.

\begin{prop}
\label{prop:uI}
Let $n, k \geq 0$, and set
\[
\X_{n,k} := \sum_{\substack{I \subseteq \ive{n}; \\ \abs{I} = n-k}} T_{u_I}
\qquad \text{ and } \qquad
\M_{n,k} := \sum_{w \in \symm_{\ive{n}-\ive{n-k}}} T_w
\]
when $k \leq n$.
(For $k>n$, set both $\X_{n,k}$ and $\M_{n,k}$ to be $0$.)
Then,
\begin{align}
\B_{n,k} &= \M_{n,k} \X_{n,k}; \label{eq:uI:Bnk=} \\
\M_{n,k}^2 &= \ive{k}!_q \M_{n,k}; \label{eq:uI:Mnk2} \\
\M_{n,k}^* &= \M_{n,k}. \label{eq:uI:Mnk*}
\end{align}
Furthermore, if $j$ is a positive integer such that $n-k < j < n$, then
\begin{align}
T_j \M_{n,k} &= q \M_{n,k}; \label{eq:uI:TjMnk} \\
T_j \B_{n,k} &= q \B_{n,k}. \label{eq:uI:TjBnk}
\end{align}
\end{prop}

\begin{proof}
Equality \eqref{eq:uI:Bnk=} was proved in \cite[Proposition 3.10]{r2r1}, with slightly different notation\footnote{In particular, $\X_{n,k}$ is $x_{(n-k,k)}$ in \cite{r2r1} and $\M_{n,k}$ is $m_{(1^{n-k},k)}$ in \cite{r2r1}.}.
The others are easy; see Appendix B in the arXiv version of this paper for more details.
\end{proof}

We can now prove something we have claimed right after Definition~\ref{def:BR}:

\begin{lem}
\label{lem:divisible}
Let $n,k\geq 0$. Then we can write
\begin{equation}
\R_{n,k} = \X^*_{n,k} \M_{n,k} \X_{n,k}. \label{eq:uI:Rnk=}
\end{equation}
Thus $\R_{n,k}$ as in \eqref{eq:Rnk:def} is
well-defined for all values of $q$, since the product $\B^*_{n,k} \B_{n,k}$ is a multiple of $\ive{k}!_q$.
\end{lem}

\begin{proof}

From \eqref{eq:uI:Mnk*}, we obtain
\begin{equation*}
\M^*_{n,k} \M_{n,k}
= \M_{n,k} \M_{n,k}
= \M_{n,k}^2
= \ive{k}!_q \M_{n,k}
\end{equation*}
by \eqref{eq:uI:Mnk2}.
Now, \eqref{eq:uI:Bnk=} yields
\begin{align}
\B_{n,k}^\ast \B_{n,k}
&= \left( \M_{n,k} \X_{n,k}\right)^* \M_{n,k} \X_{n,k}
= \X^*_{n,k} \underbrace{\M^*_{n,k} \M_{n,k}}_{=
\ive{k}!_q \M_{n,k}} \X_{n,k}
\nonumber \\
&= \ive{k}!_q \X^*_{n,k} \M_{n,k} \X_{n,k}.
\label{eq:BstarB-div}
\end{align}
Dividing this by $\ive{k}!_q$, we obtain
\[
\R_{n,k} = \X^*_{n,k} \M_{n,k} \X_{n,k},
\]
because of \eqref{eq:Rnk:def}.
This proves \eqref{eq:uI:Rnk=}.

The equality \eqref{eq:BstarB-div} also shows that $\B^*_{n,k} \B_{n,k}$ is a multiple of $[k]!_q$.
This completes our proof.
\end{proof}

The proof shows that $\R_{n,k}$ is actually a linear combination of basis elements $T_w$, with each coefficient being a universal polynomial in $q$ with integer coefficients.
Hence, any identity between $\R_{n,k}$'s that we prove for an infinite set of $q$'s will automatically be true for every $q$.
In particular, we can thus assume the invertibility of $[n]!_q$ and of $q$ in our proofs without having to require it in our claims (unless $[n]!_q$ or $q$ appear as denominators in said claim).

As another byproduct of Proposition~\ref{prop:uI}, we show that the $q=1$ specialization of $\R_{n,n-k}$ agrees with the element $\nu_{(k,1^{n-k})} \in \kk\left[\symm_n\right]$ from \cite[Section 1.2]{RSW2014}, defined on the right-hand-side of Proposition~\ref{prop:q=1-R2R} below.

\begin{prop}
\label{prop:q=1-R2R}
Assume that $q=1$. Let $n \ge k \ge 0$.
For each $w \in \symm_n$, let $\noninv_k w$ denote the number of all $k$-element subsets of $\ive{n}$ on which $w$ is strictly increasing (i.e., the number of choices of $1 \leq i_1 < i_2 < \cdots < i_k \leq n$ satisfying $w\left(i_1\right) < w\left(i_2\right) < \cdots < w\left(i_k\right)$).
Then,
\[
\R_{n,n-k}(1) = \sum_{w \in \symm_n} \tup{\noninv_k w} w \in \kk\left[\symm_n\right].
\]
\end{prop}

\begin{proof}
Applying \eqref{eq:uI:Rnk=} to $n-k$ instead of $k$, we obtain
\begin{align*}
\R_{n,n-k}
&= \X^*_{n,n-k} \M_{n,n-k} \X_{n,n-k}
= \Bigg(\sum_{\substack{I \subseteq \ive{n}; \\ \abs{I} = k}} T_{u_I}\Bigg)^*
\Bigg(\sum_{w \in \symm_{\ive{n}-\ive{k}}} T_w\Bigg)
\Bigg(\sum_{\substack{I \subseteq \ive{n}; \\ \abs{I} = k}} T_{u_I}\Bigg)
\end{align*}
by the definitions of $\X_{n,n-k}$ and $\M_{n,n-k}$.
In the specialization $q = 1$, the element $T_w$ maps to $w \in \symm_n$, so the above simplifies to
\begin{align*}
\R_{n,n-k}(1)
&= \Bigg(\sum_{\substack{I \subseteq \ive{n}; \\ \abs{I} = k}} u_I\Bigg)^*
\Bigg(\sum_{w \in \symm_{\ive{n}-\ive{k}}} w\Bigg)
\Bigg(\sum_{\substack{I \subseteq \ive{n}; \\ \abs{I} = k}} u_I\Bigg) \\
&= \sum_{\substack{I \subseteq \ive{n}; \\ \abs{I} = k}}\ \ %
\sum_{w \in \symm_{\ive{n}-\ive{k}}}\ \ %
\sum_{\substack{J \subseteq \ive{n}; \\ \abs{J} = k}}
u_I^{-1} w u_J
\qquad \big(\text{since } u_I^* = u_I^{-1}\big) \\
&= \sum_{\substack{J \subseteq \ive{n}; \\ \abs{J} = k}}\ \ %
\sum_{\substack{I \subseteq \ive{n}; \\ \abs{I} = k}}
\Bigg(\sum_{w \in \symm_{\ive{n}-\ive{k}}} u_I^{-1} w u_J\Bigg) .
\end{align*}
The sum inside the parentheses here is the sum of all permutations in $\symm_n$ that send the $k$ elements of $J$ to the $k$ elements of $I$ in increasing order (because $u_J$ sends the $k$ elements of $J$ to the $k$ numbers $1,2,\ldots,k$ in increasing order, then $w \in \symm_{\ive{n}-\ive{k}}$ leaves these $k$ numbers $1,2,\ldots,k$ unchanged, and finally $u_I^{-1}$ sends the numbers $1,2,\ldots,k$ to the $k$ elements of $I$ in increasing order).
In other words, it is the sum of all permutations $w \in \symm_n$ that are strictly increasing on $J$ and satisfy $w\tup{J} = I$.
Hence, the above equality can be rewritten as
\begin{align*}
\R_{n,n-k}(1)
&= \sum_{\substack{J \subseteq \ive{n}; \\ \abs{J} = k}}\ \ %
\sum_{\substack{I \subseteq \ive{n}; \\ \abs{I} = k}}
\Bigg(\sum_{\substack{w \in \symm_n; \\ w \text{ is increasing on } J; \\ w\tup{J} = I}} w\Bigg)
= \sum_{\substack{J \subseteq \ive{n}; \\ \abs{J} = k}}\ \ %
\sum_{\substack{w \in \symm_n; \\ w \text{ is increasing on } J}} w \\
&= \sum_{w \in \symm_n} \underbrace{\sum_{\substack{J \subseteq \ive{n}; \\ \abs{J} = k; \\ w \text{ is increasing on } J}} w}_{= \tup{\noninv_k w} w}
= \sum_{w \in \symm_n} \tup{\noninv_k w} w .
\qedhere
\end{align*}
\end{proof}

\begin{remark}
If $w$ is the permutation $w_0 \in \symm_n$ that sends $1,2,\ldots,n$ to $n,n-1,\ldots,1$, then $\noninv_k w = 0$ for all $k > 1$.
Hence, $w_0$ does not appear on the right-hand side of Proposition~\ref{prop:q=1-R2R} for $k > 1$.
However, $T_{w_0}$ can well appear in $\R_{n,n-k}$, as we saw in Example~\ref{exa:R42}.
This shows that Proposition~\ref{prop:q=1-R2R} is likely a special feature of the $q = 1$ case.
\end{remark}

\section{The crucial recursion}\label{section.recursion}

We shall next prove the crucial recursive formula in Theorem~\ref{thm:mainthm3}(1) expressing
$\B_{n} \R_{n,k}$ in terms of $\R_{n-1,k} \B_{n}$, $\R_{n-1,k-1} \B_{n}$ and
$J_n$:

\begin{theorem}[Crucial Recursion]
\label{thm:crucialrecursion}
For any $1 \le k \le n$, we have
\begin{align*}
     \B_n  \R_{n, k}  &= \Big(q^{k}\R_{n-1, k}  + \Big( \ive{n+1-k}_q + q^{n + 1 - k}J_n \Big) \R_{n - 1, k - 1} \Big)\B_n .
\end{align*}
\end{theorem}

The proof, of which the remainder of this section is devoted, involves technical algebraic computations in $\HH_n$. For readers more interested in the big picture, it is the following:
\begin{enumerate}
    \item \textsf{What Theorem~\ref{thm:crucialrecursion} is good for:} we build an eigenbasis for $\R_{n,k}$ by recursively relating it to $\R_{n-1,k-1}$ and $\R_{n-1,k}$ and multiplying successive copies of $\B_n$ to move a vector from $\HH_{n-1}$ to $\HH_n$. To accomplish this, we must understand how ``far'' $\B_n$ is from commuting with $\R_{n,k}$, which is precisely described in Theorem~\ref{thm:crucialrecursion}.
    \item \textsf{How to prove Theorem~\ref{thm:crucialrecursion}:} We use the fact that $\R_{n,k} = \frac{1}{[k]!_q} \B_{n,k}^* \B_{n,k}$ to prove several more complicated relations related to how ``far'' $\B_n$ and $\B_{n,k}^*$ are from commuting, and then reduce modulo the kernel of $\B_{n,k}$.
\end{enumerate}

We begin in \S~\ref{ssec.prelimrecursions} 
by proving some preliminary recursive relations between the elements appearing in Theorem~\ref{thm:crucialrecursion}. Then in \S~\ref{ssec.proofcrucialrecursion} we prove Theorem~\ref{thm:crucialrecursion}.

\subsection{Preliminary recursions}\label{ssec.prelimrecursions}
The following recursions for our shuffles $\B_{n,k}$, $\B^*_{n,k}$ and $\R_{n,k}$ follow easily from their definitions.

\begin{prop}
\label{prop:trivrec}
    Let $n$ and $k$ be positive. Then,
    \begin{align}
        \B_{n,k} &= \B_{n-1,k-1} \B_n ;
        \label{eq:trivrec1}
        \\
        \B^*_{n,k} &= \B^*_n \B^*_{n-1,k-1} ;
        \label{eq:trivrec2}
        \\
        \B_{n,k} &= \B_{n-k+1} \B_{n,k-1} ;
        \label{eq:trivrecb1}
        \\
        \B^*_{n,k} &= \B^*_{n,k-1} \B^*_{n-k+1} ;
        \label{eq:trivrecb2} \\
        \ive{k}_q  \R_{n, k} &= \B_n^\ast   \R_{n- 1, k - 1}  \B_n .
        \label{eq:trivrec3}
    \end{align}
\end{prop}

\begin{proof}
    First, we prove \eqref{eq:trivrec1}.
    By definition, we have $\B_{n,k}  = \B_{n-(k-1)}  \cdots \B_{n-1}  \B_n$.
    Likewise,
    $\B_{n-1,k-1}  = \B_{n-(k-1)}  \cdots \B_{n-2}  \B_{n-1}$,
	since $\tup{n-1} - \tup{\tup{k-1}-1} = n - \tup{k-1}$.
    Thus,
    \[
    \B_{n,k} = \B_{n-(k-1)}  \cdots \B_{n-1}  \B_n
    = \left( \B_{n-(k-1)}  \cdots \B_{n-2} \B_{n-1} \right) \B_n = \B_{n-1,k-1} \B_n.
    \]
    This proves \eqref{eq:trivrec1}. The proof of \eqref{eq:trivrec2} is analogous,
	and the proofs of \eqref{eq:trivrecb1} and \eqref{eq:trivrecb2} are similar.

    Finally, to prove \eqref{eq:trivrec3}, we observe that
    \begin{align*}
        \B_n^\ast  \R_{n- 1, k - 1} \B_n  & =
        \B_n^* \tup{\frac{1}{\ive{k-1}!_q} \B^*_{n-1,k-1}  \B_{n-1,k-1} } \B_n
		\qquad \left(\text{by \eqref{eq:Rnk:def}}\right) \\
        &= \frac{1}{\ive{k-1}!_q} (\B_n^* \B^*_{n-1,k-1}) ( \B_{n-1,k-1}  \B_n)
        \\
        &= \frac{1}{\ive{k-1}!_q} \B^*_{n,k} \B_{n,k} \qquad \left(\text{by \eqref{eq:trivrec2} and \eqref{eq:trivrec1}}\right) \\
        &= \frac{\ive{k}_q}{\ive{k}!_q} \B^*_{n,k} \B_{n,k} \qquad \left(\text{since } \ive{k}!_q = \ive{k-1}!_q \ive{k}_q\right) \\
        &= \ive{k}_q \R_{n,k} \qquad \left(\text{by \eqref{eq:Rnk:def}}\right).
        \qedhere
    \end{align*}
\end{proof}

We will now state two well-known properties of Young--Jucys--Murphy elements that will aid our computations later on.
They both appear in the literature (e.g., \cite[Proposition 3.26]{Mathas} or \cite[(4.1)]{Murphy92}); for the reader's convenience, we also prove them in Appendix B to the arXiv version of this paper.

\begin{lem} \label{TiJi}
    Let $n > 1$. Then,
    \[
    T_{n-1}J_{n-1}
	= J_n(T_{n-1} - q + 1) - 1.
    \]
\end{lem}

\begin{lem} \label{lem:TiJk}
    Let $n, i > 0$ with $i \notin \set{n-1, n}$. Then,
    $T_i$ commutes with $J_n$.
\end{lem}

\begin{lem} \label{lem:TiHn-1}
    Assume that $n > 0$.
    Then, $J_n$ commutes with any element of $\HH_{n-1}$.
\end{lem}

\begin{proof}
    The algebra $\HH_{n-1}$ is generated by the elements $T_1, T_2, \ldots, T_{n-2}$.
    But Lemma~\ref{lem:TiJk} shows that all these elements commute with $J_n$.
\end{proof}

Lemma~\ref{lem:TiHn-1} easily yields the well-known fact that the Young--Jucys--Murphy elements $J_1, J_2, J_3, \ldots$ all commute. Finally, we recall a recursion proved in \cite[Equation (5.2) in the proof of Theorem 5.2]{r2r1}:

\begin{lem}\label{k1recursion}
    Let $n > 1$. Then,
    \begin{align*}
        \B_n  \B_n^\ast  = \ive{n}_q + q^n  J_n  + \B_{n-1}^\ast T_{n-1}\B_{n-1}.
    \end{align*}
\end{lem}
One can think of Lemma~\ref{k1recursion} as the $k=1$ precursor to Proposition~\ref{identity-mod-ker} (proved in \S~\ref{ssec.proofcrucialrecursion} below), which will be the crucial final step to prove Theorem~\ref{thm:crucialrecursion}. Specifically, the recursion in Lemma~\ref{k1recursion} is not particularly nice, but multiplying both sides by $\B_n$ on the right (and rewriting using \eqref{eq:uI:TjBnk}) simplifies the messiest term on the right-hand-side, giving the crucial recursion for $k=1$:
\begin{equation}\label{eq:k1crucialrecursion}
    \B_n \R_{n,1} = \big( q \R_{n-1,1} +\, [n]_q + q^{n}J_n\big) \B_n.
\end{equation}
Theorem~\ref{thm:crucialrecursion} is more complicated, and thus proving and applying Proposition~\ref{identity-mod-ker} requires more work than Lemma~\ref{k1recursion}. However, the big picture remains the same.

\subsection{Proof of the crucial recursion}\label{ssec.proofcrucialrecursion}

We now obtain the technical relations needed to prove Lemma~\ref{lem:BnBnk*Bnk}, and then Theorem~\ref{thm:crucialrecursion}. For each $1 \leq i \leq n$, we define the element
\[
\lambda_{n,i}:=  \Big(T_{n-1} \cdots T_{n-i+1} \Big) \bigg(\ive{n-i+1}_q + q^{n-i+1} J_{n-i+1}\bigg)
\]
of $\HH_n$.
These elements clearly satisfy the formula
\begin{equation}
    \lambda_{n,1} = \ive{n}_q + q^n J_n
    \label{eq.lambdan1}
\end{equation}
and the recursion
\begin{equation}
    \lambda_{n,i} = T_{n-1} \lambda_{n-1,i-1} \qquad \text{for all $1 < i \leq n$.}
    \label{eq.lambda-recursion}
\end{equation}

Furthermore, for any $0 \le k \le n$, we define
\[
    \gamma_{n,k}
    := \sum_{i=1}^{k} \lambda_{n, i}
    = \sum_{i=1}^{k} T_{n-1} \cdots T_{n-i+1} \bigg(\ive{n-i+1}_q + q^{n-i+1} J_{n-i+1}\bigg).
\]
Thus, $\gamma_{n,0} = 0$ and $\gamma_{n,1} = \lambda_{n,1} = \ive{n}_q + q^n J_n$. The elements $\gamma_{n,k}$ furthermore satisfy the following recursion:

\begin{lem}
    For any $1 \le k \le n$, we have
    \begin{equation}
    \label{gamma-recursion}
    \gamma_{n,k} = \ive{n}_q + q^n J_n + T_{n-1}\gamma_{n-1,k-1}.
    \end{equation}
\end{lem}

\begin{proof}
The definition of $\gamma_{n,k}$ yields
\begin{align}
\gamma_{n,k}
&= \sum_{i=1}^{k} \lambda_{n, i}
= \lambda_{n,1} + \sum_{i=2}^{k} \lambda_{n, i}
= \lambda_{n,1} + \sum_{i=2}^{k} T_{n-1} \lambda_{n-1,i-1}
\qquad \left(\text{by \eqref{eq.lambda-recursion}}\right) \nonumber\\
&= \lambda_{n,1} + T_{n-1} \sum_{i=2}^{k} \lambda_{n-1,i-1}
= \lambda_{n,1} + T_{n-1} \sum_{i=1}^{k-1} \lambda_{n-1,i}
\label{pf:gamma-recursion:1}
\end{align}
(here, we reindexed the sum by substituting $i$ for $i-1$).
However, by the definition of $\gamma_{n-1,k-1}$, we have
\[ \sum_{i=1}^{k-1} \lambda_{n-1,i} = \gamma_{n-1,k-1} \]
and we also know that $\lambda_{n,1} = \ive{n}_q + q^n J_n$.
Thus, we can rewrite \eqref{pf:gamma-recursion:1} as
$\gamma_{n,k} = \ive{n}_q + q^n J_n + T_{n-1}\gamma_{n-1,k-1}$.
\end{proof}

Finally, we set
    \begin{align}
        \Lambda_{n,k}
        := \B^*_{n-1,k-1} \gamma_{n,k}
        \label{eq:defLambdank}
    \end{align}
for each $1 \le k \le n$.
We extend this to $k = 0$ by setting $\Lambda_{n,0} := 0$.
Note that
\begin{equation}
\Lambda_{n,1} = \B^*_{n-1,0} \gamma_{n,1} = \gamma_{n,1} = \ive{n}_q + q^n J_n.
\label{eq:Lambdan1}
\end{equation}

\begin{prop}\label{prop:rewritingBnkstar}
For any $1 \leq k \leq n$, we have
\begin{equation*}
\B_n \B^*_{n,k}
=
\B^*_{n-1, k} \Big(T_{n-1} T_{n-2} \cdots T_{n-k}\Big) \B_{n-k} + \,\Lambda_{n,k}.
\end{equation*}
Here, we understand $T_0$ to mean $0$.
\end{prop}

\begin{proof}
We proceed by induction on $k$. As a base case, when $k = 1$, we have
\begin{align}
    \B_n \B^\ast_n
    &= \B^\ast_{n-1} T_{n-1}  \B_{n-1} + \ive{n}_q + q^n J_n
	\qquad \left(\text{by Lemma~\ref{k1recursion}}\right) \nonumber\\
	&= \B^\ast_{n-1} T_{n-1}  \B_{n-1} + \,\Lambda_{n, 1}
	\qquad \left(\text{by \eqref{eq:Lambdan1}}\right).
	\label{eqn:k=1}
\end{align}
(This also holds for $n = 1$, since we set $T_0 = 0$.)

Now, let $k > 1$, and assume the proposition holds for $k - 1$.
Then, the induction hypothesis (applied to $n-1$ and $k-1$ instead of $n$ and $k$) yields
\begin{align}
    &\B_{n-1} \B^*_{n-1,k-1}
    \nonumber\\
	&=
    \B^*_{n-2, k-1} \Big(T_{n-2} T_{n-3} \cdots T_{n-k}\Big) \B_{n-k} + \,\Lambda_{n-1,k-1}
	\label{pf:prop:rewritingBnkstar:IH}
\end{align}
since $(n-1)-(k-1)=n-k$.
Furthermore, since $\B^*_{n-1,k-1} \in \HH_{n-1}$,
Lemma~\ref{lem:TiHn-1} shows that $J_n$ commutes with $\B^*_{n-1,k-1}$.
Thus, $\ive{n}_q + q^n J_n$ commutes with $\B^*_{n-1,k-1}$ as well.
In view of \eqref{eq:Lambdan1}, we can restate this as follows:
$\Lambda_{n, 1}$ commutes with $\B^*_{n-1,k-1}$.
But \eqref{eq:trivrec2} yields
\begin{align*}
    \B_n \B^\ast_{n, k} &= \left(\B_n  \B^\ast_n\right) \B^\ast_{n-1, k - 1}\\
    &= \left(\B^\ast_{n-1}T_{n-1} \B_{n-1} + \,\Lambda_{n, 1}\right) \B^\ast_{n-1, k - 1} \qquad \left(\text{by \eqref{eqn:k=1}}\right)\\
    &= \B^\ast_{n-1}T_{n-1} \left(\B_{n-1} \B^\ast_{n-1,k-1}\right) + \,\Lambda_{n, 1}  \B_{n-1, k - 1}^\ast\\
    &= \B^\ast_{n-1}T_{n-1} \left(\B_{n-1} \B^\ast_{n-1,k-1}\right) + \B_{n-1, k - 1}^\ast \,\Lambda_{n, 1}
    \qquad \left(\text{since $\Lambda_{n, 1}$ commutes with $\B^*_{n-1,k-1}$}\right) \\
    &= \begin{multlined}[t]
        \B^\ast_{n-1} T_{n-1}\Big(\B^\ast_{n-2, k - 1}\left(T_{n-2}T_{n-3}\cdots T_{n-k}\right)\B_{n-k}
        + \, \Lambda_{n - 1,k - 1}\Big) + \B_{n-1, k-1}^\ast \,\Lambda_{n, 1}        \qquad \left(\text{by \eqref{pf:prop:rewritingBnkstar:IH}}\right)
	    \end{multlined} \\
    &= \begin{multlined}[t]
        \underbrace{\B^\ast_{n-1} T_{n-1} \B^\ast_{n-2, k - 1}\left(T_{n-2}T_{n-3}\cdots T_{n-k}\right)\B_{n-k}}_{=:\  \alpha} + \underbrace{\B^\ast_{n-1} T_{n-1} \Lambda_{n - 1,k - 1} + \B_{n-1, k-1}^\ast \,\Lambda_{n, 1}}_{=:\  \beta}.
    \end{multlined}
\end{align*}
Note that $T_{n-1}$ commutes with $\B^\ast_{n-2, k - 1}$ since $\B^\ast_{n-2, k - 1} \in \HH_{n-2}$. With $\alpha,\beta$ defined above, we thus have
\begin{align*}
    \alpha
    &= \B^\ast_{n-1} \B^\ast_{n-2, k - 1}\left(T_{n-1}  T_{n-2}T_{n-3}\cdots T_{n-k}\right)\B_{n-k}
    \\
    &= \B^\ast_{n-1, k}\left(T_{n-1}T_{n-2}\cdots T_{n-k}\right)\B_{n-k} \qquad \left(\text{by \eqref{eq:trivrec2}}\right).
\end{align*}
Hence, to complete the proof, it suffices to prove that $\beta = \Lambda_{n,k}$. But the definition of $\Lambda_{n-1,k-1}$ yields
\begin{align*}
    \beta
    &= \B^*_{n-1} T_{n-1} \B^*_{n-2,k-2} \gamma_{n-1,k-1} + \B_{n-1, k-1}^\ast \,\Lambda_{n, 1} \\
    &= \B^*_{n-1} \B^*_{n-2,k-2} T_{n-1} \gamma_{n-1,k-1} + \B_{n-1, k-1}^\ast \,\Lambda_{n, 1}
    \qquad \left(\text{since $T_{n-1}$ commutes with $\B^\ast_{n-2, k-2}$}\right) \\
    &= \B^*_{n-1,k-1} T_{n-1} \gamma_{n-1,k-1} + \B_{n-1, k-1}^\ast \,\Lambda_{n, 1} \qquad \left(\text{by \eqref{eq:trivrec2}}\right) \\
    &= \B^*_{n-1,k-1} \left( \,\Lambda_{n, 1} + T_{n-1} \gamma_{n-1,k-1} \right) \\
	&= \B^*_{n-1,k-1} \left( \ive{n}_q + q^nJ_n + T_{n-1} \gamma_{n-1,k-1} \right)
    \qquad \left(\text{since } \Lambda_{n,1} = \ive{n}_q + q^n J_n \right) \\
    &= \B^*_{n-1,k-1} \gamma_{n,k}
    \qquad \left(\text{by \eqref{gamma-recursion}}\right) \\
    &= \Lambda_{n, k},
\end{align*}
completing the proof.
\end{proof}

\begin{prop}
    \label{identity-mod-ker}
    Let $0 \leq k \leq n$.
    Then,
    \[
    \gamma_{n,k} \B_{n,k} = \ive{k}_q \Big(\ive{n-k+1}_q + q^{n-k+1}J_n \Big) \B_{n,k}.
    \]
\end{prop}

\begin{proof}
We let $\ker(\B_{n,k})$ be the set of all $h \in \HH_n$ that satisfy $h \B_{n,k} = 0$. This is a left ideal of $\HH_n$.

Furthermore, set
\[
\rho_{n,k} := \ive{k}_q \Big(\ive{n-k+1}_q + q^{n-k+1} J_n \Big) .
\]

We must show that $\gamma_{n,k} \B_{n,k} = \rho_{n,k} \B_{n,k}$.
Equivalently, we must show that $\gamma_{n,k} \equiv \rho_{n,k} \pmod{ \ker(\B_{n,k})}$.

We proceed by induction on $k$, where the $k=0$ case boils down to $0 = 0$, and the $k = 1$ case follows from $\gamma_{n,1} = \ive{n}_q + q^n J_n = \rho_{n,1}$.

Now assume $k > 1$ (whence $n > 1$ as well).
By our induction hypothesis, 
\[
\gamma_{n-1, k-1} \equiv \rho_{n-1, k-1} \pmod{\ker(\B_{n-1, k-1})}.
\]
Since $\ker(\B_{n-1, k-1}) \subseteq \ker(\B_{n,k})$ (a consequence of \eqref{eq:trivrec1}), this entails
\[
\gamma_{n-1, k-1} \equiv \rho_{n-1, k-1} \pmod{\ker(\B_{n, k})}.
\]
In view of this, \eqref{gamma-recursion} becomes
\begin{equation}
    \gamma_{n,k} \equiv \ive{n}_q + q^n J_n + T_{n-1}\rho_{n-1,k-1} \pmod{\ker(\B_{n, k})}.
    \label{pf:identity-mod-ker:IHmutated}
\end{equation}
We rewrite the term $T_{n-1} \ \rho_{n-1,k-1}$ as follows:
\begin{align*}
    T_{n-1} \ \rho_{n-1,k-1}
    & = T_{n-1} \ive{k-1}_q \Big(\ive{n-k+1}_q + q^{n-k+1} J_{n-1} \Big)
    \qquad \left(\text{by definition}\right)
    \\
    &= \ive{k-1}_q\ive{n-k+1}_q T_{n-1}
        + q^{n-k+1}\ive{k-1}_q T_{n-1}J_{n-1} \\
    &\qquad\qquad \left(\text{by distributivity}\right)
    \\
    &= \ive{k-1}_q\ive{n-k+1}_q T_{n-1}
      +  q^{n-k+1}\ive{k-1}_q \Big( J_n (T_{n-1} - q + 1) - 1 \Big)
	\\
    &\qquad\qquad \left(\text{by Lemma~\ref{TiJi}}\right) \\
	& \equiv
        \ive{k-1}_q\ive{n-k+1}_q \ q
        + q^{n-k+1}\ive{k-1}_q \Big( J_n (q - q + 1) - 1\Big)
    \\
	& \qquad\qquad \left(\text{since \eqref{eq:uI:TjBnk} shows that $T_{n-1} \equiv q \pmod{\ker(\B_{n, k})}$}
              \right) \\
    &= q \ive{k-1}_q\ive{n-k+1}_q  +  q^{n-k+1}\ive{k-1}_q ( J_n -1) \pmod{\ker(\B_{n, k})}.
\end{align*}
We substitute this into \eqref{pf:identity-mod-ker:IHmutated} to obtain
\begin{align*}
    \gamma_{n,k} 
    &\equiv \ive{n}_q + q^n J_n + q \ive{k-1}_q\ive{n-k+1}_q  +  q^{n-k+1}\ive{k-1}_q ( J_n -1) \\
    &=
        \Big(\ive{n}_q  - q^{n-k+1}\ive{k-1}_q\Big)
        + q \ive{k-1}_q\ive{n-k+1}_q
        + q^{n-k+1}\big(q^{k-1} + \ive{k-1}_q \big) J_n
    \\
    &= \ive{n-k+1}_q + q \ive{k-1}_q\ive{n-k+1}_q + q^{n-k+1}\ive{k}_q J_n
	\qquad \left(\text{by \eqref{eq:qint-id:1} and \eqref{eq:qint-id:2}}\right) \\
    &= (1 + q \ive{k-1}_q) \ive{n-k+1}_q + q^{n-k+1}\ive{k}_q J_n \\
    &= \ive{k}_q \ive{n-k+1}_q + q^{n-k+1}\ive{k}_q J_n
	\qquad \left(\text{by \eqref{eq:qint-id:3}}\right) \\
    &= \ive{k}_q \big( \ive{n-k+1}_q +  q^{n-k+1} J_n \big) \pmod{\ker(\B_{n, k})}, 
\end{align*}
where we used the $q$-integer identities
\begin{align}
        \ive{n}_q - q^{n-k+1}\ive{k-1}_q & = \ive{n-k+1}_q, \label{eq:qint-id:1}\\
        q^{k-1} + \ive{k-1}_q & = \ive{k}_q, \label{eq:qint-id:2}\\
        1 + q \ive{k-1}_q & = \ive{k}_q. \label{eq:qint-id:3}
\end{align}
In other words, $\gamma_{n,k} \equiv \rho_{n,k} \pmod{\ker(\B_{n, k})}$. This completes the induction step, and with it the proof of the proposition.
\end{proof}

We obtain the following natural corollary.
\begin{cor}\label{cor:lambda}
    Let $1 \leq k \leq n$. Then,
    \[
	\Lambda_{n,k} \B_{n,k} = \ive{k}_q \bigg( \ive{n+1-k}_q + q^{n+1-k} J_n \bigg) \B^*_{n-1,k-1} \B_{n,k}.
	\]
\end{cor}

\begin{proof}
	Lemma~\ref{lem:TiHn-1} shows that $J_n$ commutes with $\B^*_{n-1, k-1}$ since $\B^*_{n-1, k-1} \in \HH_{n-1}$.
	Hence, 
    \[ \ive{k}_q \bigg( \ive{n+1-k}_q + q^{n+1-k} J_n \bigg) \] commutes with $\B^*_{n-1, k-1}$ as well.
	
    By Proposition~\ref{identity-mod-ker},
    \begin{align}
     \gamma_{n,k} \B_{n,k}
     &= \ive{k}_q \bigg( \ive{n+1-k}_q + q^{n+1-k} J_n \bigg) \B_{n,k}.
     \label{pf.cor:lambda.1}
    \end{align}
    Recalling the definition of $\Lambda_{n,k}$, we have
    \begin{align*}
        \Lambda_{n,k} \B_{n,k}
        &= \B^*_{n-1, k-1} \gamma_{n,k} \B_{n,k} \\
		&= \B^*_{n-1, k-1} \ive{k}_q \bigg( \ive{n+1-k}_q + q^{n+1-k} J_n \bigg) \B_{n,k}
		\qquad \left(\text{by \eqref{pf.cor:lambda.1}}\right) \\
        &= \ive{k}_q \bigg( \ive{n+1-k}_q + q^{n+1-k} J_n \bigg) \B^*_{n-1,k-1} \B_{n,k}
    \end{align*} 
    since $\ive{k}_q \big( \ive{n+1-k}_q + q^{n+1-k} J_n \big)$ commutes with $\B^*_{n-1, k-1}$.
\end{proof}
Our work thus far comes together in the following lemma, from which Theorem~\ref{thm:crucialrecursion} will follow easily.
\begin{lem}
\label{lem:BnBnk*Bnk}
For any $1 \le k \le n$, we have
\[
    \B_n \B_{n,k}^* \B_{n,k} = q^k \B_{n-1,k}^* \B_{n,k+1} + \ive{k}_q \Big(  \ive{n+1-k}_q + q^{n + 1 - k}J_n \Big)\B_{n-1,k-1}^* \B_{n,k}.
\]
\end{lem}

\begin{proof}
By Proposition~\ref{prop:rewritingBnkstar}, we have that 
\[
\B_n \B_{n,k}^* = \B^*_{n-1, k} \Big(T_{n-1} T_{n-2} \cdots T_{n-k}\Big) \B_{n-k} + \,\Lambda_{n,k}.
\]
Multiply both sides by $\B_{n,k}$ from the right. The first term on the right-hand side becomes 
\begin{align*}
    \B^*_{n-1, k} \Big(T_{n-1} T_{n-2} \cdots T_{n-k}\Big) \B_{n-k} \B_{n,k}
    &= \B^*_{n-1, k} \Big(T_{n-1} T_{n-2} \cdots T_{n-k}\Big)  \B_{n,k+1}
    \qquad \left(\text{by \eqref{eq:trivrecb1}}\right) \\
    &= q^{k} \B^*_{n-1, k} \B_{n,k+1},
\end{align*} 
since \eqref{eq:uI:TjBnk} yields
$\Big(T_{n-1} T_{n-2} \cdots T_{n-k}\Big)  \B_{n,k+1} = q^k \B_{n,k+1}$
(this holds in the border case $k = n$ as well, since \eqref{eq:Bnk:zero} entails $\B_{n,k+1}=0$ in this case).
On the other hand, Corollary~\ref{cor:lambda} implies that 
    \[
	\Lambda_{n,k} \B_{n,k} = \ive{k}_q \Big( \ive{n+1-k}_q + q^{n+1-k} J_n \Big) \B^*_{n-1,k-1} \B_{n,k}.
	\]
Combining these terms completes the proof.
\end{proof}

We are finally ready to prove Theorem~\ref{thm:crucialrecursion}.

\begin{proof}[Proof of Theorem~\ref{thm:crucialrecursion}]
By Lemma~\ref{lem:BnBnk*Bnk}, we have
\begin{align*}
	\B_n \B_{n,k}^* \B_{n,k}
	&= q^k \B_{n-1,k}^* \B_{n,k+1} + \ive{k}_q \Big(  \ive{n+1-k}_q + q^{n + 1 - k}J_n \Big)\B_{n-1,k-1}^* \B_{n,k} \\
	&= q^k \B_{n-1,k}^* \B_{n-1,k} \B_n + \ive{k}_q \Big(  \ive{n+1-k}_q + q^{n + 1 - k}J_n \Big)\B_{n-1,k-1}^* \B_{n-1,k-1} \B_n \\
	& \qquad \qquad \left(\text{since $\B_{n,k+1} = \B_{n-1,k} \B_n$ and $\B_{n,k} = \B_{n-1,k-1} \B_n$}\right) \\
	&= \left(q^k \B_{n-1,k}^* \B_{n-1,k} + \ive{k}_q \Big(  \ive{n+1-k}_q + q^{n + 1 - k}J_n \Big)\B_{n-1,k-1}^* \B_{n-1,k-1} \right) \B_n .
\end{align*}
Now divide both sides by $\ive{k}!_q$.
In view of the fact that $\R_{n,k} = \frac{1}{\ive{k}!_q} \B^*_{n,k} \B_{n,k}$ and the analogous identities for smaller subscripts, this becomes precisely the claim of the theorem.
\end{proof}

\begin{remark}
Proposition~\ref{identity-mod-ker} and Corollary~\ref{cor:lambda}
can be made stronger by replacing each appearance of $\B_{n,k}$
by $\M_{n,k}$ (see Proposition~\ref{prop:uI} and recall
\eqref{eq:uI:Bnk=}).
However, we will have no need for these stronger variants,
so we leave the proofs to the interested reader
(the main new ingredient is the easy fact that
$\M_{n,k} \in \M_{n-1,k-1} \HH_n$).
\end{remark}

\section{Commutativity of the $k$-random-to-random shuffles}\label{section.commutativity}

We now show how the crucial recursion in Theorem~\ref{thm:crucialrecursion} allows us to prove that the $\R_{n,k}$ pairwise commute.
\begin{theorem}
\label{thm:commut}
Let $n\geq0$. The elements $\R_{n,k}$ for all $k \geq0$ commute.
\end{theorem}

\begin{proof}[Proof of Theorem \protect\ref{thm:commut}.]
We induct on $n$. The base case ($n=0$) is obvious, so we proceed to the induction step. Let $n>0$ and $i,j\geq 0$. We must show that $\R_{n,i} \R_{n,j}
= \R_{n,j} \R_{n,i}$. This is obvious if $i=0$
(since $\R_{n,i}=1$ in this case) and also obvious if $i>n$ (since
\eqref{eq:Rnk:zero} yields that $\R_{n,i}=0$ in this case). Hence, we assume without loss of generality that $1\leq i\leq n$.
For similar reasons, we assume that $1\leq j\leq n$.

The element $\R_{n-1,j-1}\in \HH_{n-1}$ commutes with $\R_{n-1,i}$ and $\R_{n-1,i-1}$ by the induction
hypothesis, and also commutes with $J_n$ since any element of
$\HH_{n-1}$ commutes with $J_n$ by Lemma~\ref{lem:TiHn-1}.
The element $\R_{n-1,i-1}\in \HH_{n-1}$ also commutes with $J_n$ for
a similar reason.

Theorem~\ref{thm:crucialrecursion} yields
\begin{equation*}
\B_n \R_{n,i} = \left( q^i\R_{n-1,i}+\left( \ive{n+1-i}_q +q^{n+1-i}J_n\right) \R_{n-1,i-1}\right) \B_n.
\end{equation*}
Setting
\begin{equation}
\mathcal{W}_{n,i} := q^{i}\R_{n-1,i}+\left( \ive{n+1-i}_q +q^{n+1-i}J_n\right) \R_{n-1,i-1} ,
\label{pf:thm:commut:W}
\end{equation}
we can rewrite this as
\begin{equation}
\B_n \R_{n,i} = \mathcal{W}_{n,i} \B_n.  \label{pf:thm:commut:1}
\end{equation}

Recall that $\R_{n-1,j-1}$ commutes with $\R_{n-1,i}$ and $\R_{n-1,i-1}$ and with $J_n$.
Thus, $\R_{n-1,j-1}$ commutes with $\mathcal{W}_{n,i}$ as well (since
$\mathcal{W}_{n,i}$ is a polynomial in $\R_{n-1,i}$, $\R_{n-1,i-1}$ and $J_n$).

Applying the anti-involution $a \mapsto a^*$ to the equality \eqref{pf:thm:commut:W}, we obtain
\begin{align}
    \left(\mathcal{W}_{n,i}\right)^*
    &= q^{i} \left(\R_{n-1,i}\right)^* + \left( \R_{n-1,i-1} \right)^* \left( \ive{n+1-i}_q + q^{n+1-i} \left(J_n\right)^* \right)
    \nonumber \\
    &\qquad\qquad \left(\text{since an anti-involution reverses all products}\right) \nonumber \\
    &= q^i \R_{n-1,i} + \R_{n-1,i-1} \left( \ive{n+1-i}_q + q^{n+1-i} J_n \right)
    \qquad \left(\text{by \eqref{eq:star:R} and \eqref{eq:star:J}}\right) \nonumber  \\
    &= q^i \R_{n-1,i} + \left( \ive{n+1-i}_q + q^{n+1-i} J_n \right) \R_{n-1,i-1} \nonumber \\
    &
    \qquad\qquad \left(\text{since $\R_{n-1,i-1} $ commutes with $J_n$}\right) \nonumber \\
    &= \mathcal{W}_{n,i}
    \qquad \left(\text{by \eqref{pf:thm:commut:W}}\right).
	\label{eq:star:W}
\end{align}

But \eqref{eq:trivrec3} yields $\ive{j}_q \R_{n,j}=
\B_n^\ast \R_{n-1,j-1} \B_n$, so that 
\begin{equation*}
\R_{n,j}
= \dfrac{1}{\ive{j}_q} \B_n^\ast \R_{n-1,j-1} \B_n.
\end{equation*}
Thus,
\begin{align}
\R_{n,j} \R_{n,i} &= \dfrac{1}{\ive{j}_q}
\B_n^\ast \R_{n-1,j-1} \B_n \R_{n,i}
\nonumber \\
&= \dfrac{1}{\ive{j}_q} \B_n^\ast \R_{n-1,j-1}
 \mathcal{W}_{n,i} \B_n  \label{pf:thm:commut:2}
\end{align}
(by \eqref{pf:thm:commut:1}).
Again applying the anti-involution $a\mapsto a^\ast$ to
both sides of this equality, we find
\begin{align*}
 \left( \R_{n,i}\right) ^\ast \left( \R_{n,j}\right)
^\ast
&= \dfrac{1}{\ive{j}_q} \left( \B_n\right) ^\ast
\left( \mathcal{W}_{n,i} \right)^*
\left( \R_{n-1,j-1}\right) ^\ast 
\left( \B_n^\ast\right) ^\ast \\
&= \dfrac{1}{\ive{j}_q} \B_n^\ast
\mathcal{W}_{n,i} \R_{n-1,j-1}  \B_n
\qquad
\left(\text{by  \eqref{eq:star:W} and \eqref{eq:star:R}}\right) \\
&= \dfrac{1}{\ive{j}_q} \B_n^\ast
\R_{n-1,j-1} \mathcal{W}_{n,i} \B_n
\end{align*}
(since $\R_{n-1,j-1}$ commutes with $\mathcal{W}_{n,i}$).
Comparing this with \eqref{pf:thm:commut:2}, we obtain
\begin{equation*}
\R_{n,j} \R_{n,i}=\left( \R_{n,i}\right) ^{\ast
} \left( \R_{n,j}\right) ^\ast
= \R_{n,i} \R_{n,j}
\qquad \left(\text{by \eqref{eq:star:R}}\right).
\end{equation*}
In other words, $\R_{n,i} \R_{n,j}
= \R_{n,j} \R_{n,i}$. Thus, the induction is
complete.
\end{proof}

\section{Proof of the crucial eigenvalue recursion}\label{section.eigenvectorsandkernels}
Theorem~\ref{thm:mainthm3}(1) is now proved.
We next take aim at the Crucial Eigenvalue Recursion (Theorem~\ref{thm:mainthm3}(2) from the introduction). To do so, we will need to understand the $\HH_n$-module structure on the $\R_{n,k}$-eigenspaces. We give a very short overview of the necessary notation and vocabulary in \S \ref{ssec.reptheorybackgroun}; in \S~\ref{ssec.kernelofrnk} we study the kernels of the $\R_{n,k}$ and their relationships as $k$ varies. We then prove the crucial eigenvalue recursion in \S~\ref{ssec.crucialevalrecursion} as Theorem~\ref{prop:eval-recursion-1-step}, and obtain a simultaneous eigenbasis for the $\R_{n,k}$ in the process. We note that the proof of Theorem~\ref{prop:eval-recursion-1-step} references many ideas in \cite{r2r1}, and the reader interested in a slower  and more intuitive explanation of these concepts should first read Section 5 in \cite{r2r1}.

\subsection{Brief representation theory background}\label{ssec.reptheorybackgroun}
Our goal is to determine the eigenvalues from the right action of $\R_{n,k}$ on $\HH_n$. To do so, we will need to decompose $\HH_n$ into irreducible representations and understand the action of $\R_{n,k}$ on each irreducible.

In this section, we assume that $q$ is neither $0$ nor a primitive $k$-th root of unity for any $k \in \set{2,3,\ldots,n}$.
Moreover, we assume that if $q = 1$, then $n!$ is invertible in $\kk$ (that is, we don't have $1 \leq \charr \kk \leq n$).

With these assumptions in place, the Hecke algebra $\HH_n$ will be semisimple, and its representation theory well-understood (cf. \cite[Chapter 3]{Mathas}). In particular, for each partition $\lambda$ of $n$, there is a simple right $\HH_n$-module $S^\lambda$ called the \emph{Specht module}, and these Specht modules constitute all the simple right $\HH_n$-modules up to isomorphism.
The regular right $\HH_n$-module $\HH_n$ can be decomposed as
\begin{align}
\HH_n \cong \bigoplus_{\lambda \vdash n} \tup{S^\lambda}^{\oplus f^\lambda},
\label{eq:Hn=sumSlam}
\end{align}
where $f^\lambda$ denotes the number of standard Young tableaux of shape $\lambda$.
See also \cite[Theorem 5.2]{DipperJames86} for a proof in the case when $\kk$ is a field of characteristic $0$ and $q$ is either an indeterminate or a positive real.
(This case is sufficient for our needs, since our results --- that is, Corollaries~\ref{thm:charpoly-lam} and \ref{cor:charpoly} --- can be automatically extended from it to the general case.)

The semisimplicity of $\HH_n$ and the classification of its irreducible representations as $S^\lambda$ yields the $\kk$-algebra isomorphism
\begin{equation}\label{eq:algdecomp}
\HH_n \cong \prod_{\lambda \vdash n} \End_\kk \left(S^\lambda\right)
\end{equation}
by the Artin--Wedderburn theorem.
We can also view $\HH_n$ as an $\HH_n$-$\HH_n$-bimodule via left and right multiplication. In this case, we can use \eqref{eq:algdecomp} to decompose $\HH_n$ as
\begin{equation}\label{eq:bimoduledecomp}
    \HH_n \cong \bigoplus_{\lambda \vdash n} (S^\lambda)^{*} \otimes S^\lambda,
\end{equation}
where $S^\lambda$ is a \emph{right} Specht module and $(S^\lambda)^{*}$ is its \emph{left} dual. The key idea here is that $\R_{n,k}$ acts on the right via multiplication, which means:
\begin{itemize}
    \item its eigenvectors can be understood as living inside the right $\HH_n$-module $S^\lambda$, and 
    \item the left multiplication action commutes with right multiplication by $\R_{n,k}$, and so its \emph{right} eigenspaces are \emph{left} $\HH_n$-modules.
\end{itemize}
In this paper, we will first identify an $\R_{n,k}$-eigenbasis for the right module $S^\lambda$.
We will then be able to use the decomposition in \eqref{eq:bimoduledecomp} to deduce information about the action of $\R_{n,k}$ on all of $\HH_n$. 

We follow the definitions and notations in \cite{r2r1}, which we briefly recall here. We refer the reader to \cite[\S 2]{r2r1} for a detailed description of these concepts.

\subsubsection{(Skew) partitions and tableaux}

\begin{itemize}[itemsep=0.5pc]
    \item A \emph{partition} is a finite weakly decreasing tuple $\lambda = \tup{\lambda_1, \lambda_2, \ldots, \lambda_k}$ of positive integers. We identify it with the infinite sequence $\tup{\lambda_1, \lambda_2, \ldots, \lambda_k, 0, 0, 0, \ldots}$.
	The notation ``\emph{$\lambda \vdash n$}'' is shorthand for ``$\lambda$ is a partition of $n$'', which means that $\lambda$ is a partition and $\lambda_1 + \lambda_2 + \ldots + \lambda_k = n$.
	\item The \emph{Young diagram} (or, for short, \emph{diagram}) of a partition $\lambda = \tup{\lambda_1, \lambda_2, \ldots, \lambda_k}$ is a table with rows of lengths $\lambda_1, \lambda_2, \ldots, \lambda_k$. These rows are listed from the top (English notation) and left-aligned. Formally, this diagram is the set of all pairs $\tup{i, j}$ of positive integers satisfying $i \leq k$ and $j \leq \lambda_i$. These pairs are called the \emph{cells} (or \emph{boxes}) of the diagram.
    \item If $\lambda$ is a partition, then a \emph{(Young) tableau} of shape $\lambda$ means a filling of the diagram of $\lambda$ with positive integers (one positive integer per cell). Formally, this means a function from the diagram of $\lambda$ to $\set{1,2,3,\ldots}$.
    \item A tableau $\frakt$ with $n$ cells is said to be \emph{standard} if its entries are $1,2,\ldots,n$ and increase left-to-right along rows and top-to-bottom down columns. The set of all standard tableaux of shape $\lambda$ is denoted \emph{$\syt\tup{\lambda}$}.

    \item Two partitions $\mu$ and $\lambda$ are said to satisfy $\mu \subseteq \lambda$ if and only if the diagram of $\mu$ is contained in that of $\lambda$.

	\item A \emph{skew partition} is a pair $\tup{\lambda, \mu}$ of two partitions satisfying $\mu \subseteq \lambda$. It is stylized $\lambda \sm \mu$. Its \emph{(skew Young) diagram} is the set of all cells of the diagram of $\lambda$ that do not belong to the diagram of $\mu$.
    \item A \emph{skew tableau} of shape $\lambda \sm \mu$ (where $\lambda \sm \mu$ is a skew partition) is a filling of the skew diagram of $\lambda \sm \mu$. Standard skew tableaux are defined in the obvious way, and are filled with the integers $j + 1, j + 2, \ldots, n$ if $\mu \vdash j$ and $\lambda \vdash n$.
	\item A skew partition $\lambda \sm \mu$ is said to be a \emph{horizontal strip} if its skew diagram has no two cells in the same column.
	\item Given two partitions $\lambda$ and $\mu$, we write ``$\mu \lessdot \lambda$'' if the diagram of $\mu$ is the diagram of $\lambda$ with one cell removed. In other words, it means that $\mu \subseteq \lambda$ and the skew diagram of $\lambda \sm \mu$ consists of a single box.
    \item The \emph{content} (or \emph{diagonal index}) of a cell $\tup{i, j}$ is defined to be the integer $\content(i,j):= j-i$. The \emph{content set} of a skew partition $\lambda \sm \mu$ is the (multi)set $\content(\lambda \sm \mu)$ consisting of the content of each box in $\lambda \sm \mu$.
    (We shall only use it when $\lambda \sm \mu$ is a horizontal strip; in this case, the multiset is a genuine set.)

\end{itemize}

\subsubsection{Hecke algebra representations}

\begin{itemize}[itemsep=0.5pc]
	\item Each partition $\lambda$ of $n$ determines a central idempotent $p_\lambda \in \HH_n$, called the \emph{isotypic projector}; this is the element that acts as $1$ on $S^\lambda$ and as $0$ on all other $S^\mu$. It projects $\HH_n$ onto the $\lambda$-addend in \eqref{eq:algdecomp} and in \eqref{eq:bimoduledecomp}.
    These idempotents are orthogonal (i.e., satisfy $p_\lambda p_\mu = \delta_{\mu, \lambda} p_\lambda$).
    \item Every partition $\lambda$ of $n$ corresponds to a right $\HH_n$-module $W^\lambda$, called the \emph{(word) Young module}. As a vector space, it is the span of all words of length $n$ and content $\lambda$ (that is, with $\lambda_i$ letters $i$ for each positive integer $i$). The $\HH_n$-action on $W^\lambda$ is specified in \cite[(2.9)]{r2r1}. For readers familiar with the symmetric group, the $q=1$ specialization of $W^\lambda$ is the permutation module with Frobenius image $h_\lambda$. The Hecke analog plays an analogous role, see \cite[(2.9)]{r2r1}.
    
    \item To every partition $\lambda$ of $n$ corresponds a simple right $\HH_n$-module, called the \emph{Specht module} and denoted by $S^\lambda$.
	It is defined as a certain submodule of the word module $W^\lambda$;
	see \cite[Section 2.2.2]{r2r1}.
    
	\item Let $\lambda \sm \mu$ be a skew partition with $|\lambda| = n$ and $|\mu|=j$. Let $c_1,c_2,\ldots,c_{n-j}$ be the cells of its Young diagram, listed row by row from top to bottom, where the cells in each given row are listed from left to right.
    For example, if $\lambda = (3,2)$ and $\mu = (1,1)$, then $c_1 = (1,2)$, $c_2 = (1,3)$ and $c_3 = (2,2)$. Let $r_1,r_2,\ldots,r_{n-j}$ be the numbers of the rows in which the respective cells $c_1,c_2,\ldots,c_{n-j}$ lie. Then in our example, $r_1 = r_2 = 1$ and $r_3 = 2$.
    
	The linear map $\Phi_{\lambda \sm \mu} : W^\mu \to W^\lambda$ sends each word $w \in W^\mu$ to the concatenation $wr_1r_2\ldots r_{n-j}$.
    We write this map $\Phi_{\lambda \sm \mu}$ on the right of its argument -- i.e., we write $a \ \Phi_{\lambda \sm \mu}$ rather than $\Phi_{\lambda \sm \mu}\tup{a}$.
    This map $\Phi_{\lambda \sm \mu}$ is precisely the $\Phi_{\frakt^{\lambda \sm \mu}}$ from \cite[Definition 3.4]{r2r1} (since the tableau $\frakt^{\lambda\sm\mu}$ defined in \cite[\S 2.2.2]{r2r1} has its entries $j+1,j+2,\ldots,n$ in the cells $c_1,c_2,\ldots,c_{n-j}$, hence in rows $r_1,r_2,\ldots,r_{n-j}$, respectively).
    Note that \cite[Lemma 3.6]{r2r1} says that
    \begin{align}
    \Phi_{\lambda \sm \mu}\text{ commutes with right multiplication by any } a \in \HH_j.
    \label{eq:Phi-commut}
    \end{align}

    If $\lambda \sm \mu$ consists of a single box (that is, if $\mu \lessdot \lambda$), and if this box lies in row $r$, then we also denote the map $\Phi_{\lambda \sm \mu}$ by $\Phi_r$.

    \end{itemize}

\subsection{The kernel of $\R_{n,k}$}\label{ssec.kernelofrnk}

Recall that elements of $\HH_n$ act from the right on $\HH_n$ by multiplication. Thus, for any $x \in \HH_n$,
\begin{itemize}
    \item we let $\ker x$ denote the kernel of the right action of $x$ on $\HH_n$, that is, the set 
    $\left\{ y \in \HH_n \mid yx = 0 \right\}$;
    \item we let $\im x$ denote the image of this right action, i.e., the left ideal $\HH_n x$ of $\HH_n$;
    \item we define the \emph{characteristic polynomial}, the \emph{eigenvectors} and the \emph{eigenvalues} of $x$ to be those of the $\kk$-vector space endomorphism of $\HH_n$ given by $y \mapsto yx$;
    \item we say that $x$ is \emph{diagonalizable} if this endomorphism is diagonalizable (over $\kk$, not just over a field extension).
\end{itemize}
Our goal in this subsection is to understand the kernel of $\R_{n,k}$.

\subsubsection{The kernel of $k$-random-to-bottom}

We begin with a brief aside about the kernels of $\B_n^*$ and $\B_{n,k}^\ast$, as these are interesting in their own right (see also \cite[Proposition 2]{Lusztig03}).
The reader only interested in $\R_{n,k}$ can skip to \S~\ref{subsubsec.kerRnk}.
Recall that $\B_n^\ast = \B_{n,1}^\ast$.

In \cite[\S 6]{Lusztig03}\footnote{In \cite[\S 6]{Lusztig03}, Lusztig works with $q$ being a formal indeterminate over $\QQ$, but the claim then automatically extends to arbitrary $q$ because it is a polynomial identity in $q$.}
, \cite[Theorem 4.11 (2)]{r2r1} and \cite{BCR-flags} (see \cite[\S 4.3-4]{r2r1} for a translation of \cite{BCR-flags} to the present setting), it is shown that $\B_n^\ast$ has characteristic polynomial
\begin{equation}\label{eq:charpolyBn}
    \prod_{j=0}^{n} \big( x - [j]_q \big)^{\binom{n}{j}d_{n-j}}
\end{equation}
where 
\[ d_m = m! \sum_{r=0}^m \frac{(-1)^r}{r!} \]
is the $m$-th derangement number counting permutations in $\symm_m$ with no fixed points.

\begin{lem}\label{lem:B*nk-as-prod-of-B*n}
For any $k \geq 0$, we have
    \begin{align}
    q^{k \choose 2}\B_{n, k}^\ast = \prod_{j = 0}^{k - 1} \left(\B_n^\ast - [j]_q\right).
    \label{eq:lem:B*nk-as-prod-of-B*n:main}
    \end{align}
\end{lem}

\begin{proof}
First, we will prove the identity
\begin{align}\label{eqn:ind-step}
    \B^\ast_n \B^\ast_{n, k} = q^k \B^\ast_{n, k + 1} +\, [k]_q \B^\ast_{n, k}
\end{align}
by induction on $k$.
If $k = 0$, then both sides are equal to $\B_n^\ast$.
Now, towards the inductive step, let $k>0$.
The cases $n\leq 1$ are easy, so we assume furthermore that $n>1$.
Then, \eqref{eq:uI:TjBnk} yields $T_{n-1} \B_{n,2} = q \B_{n,2}$, that is, $T_{n-1} \B_{n-1} \B_n = q \B_{n-1} \B_n$ (since $\B_{n,2} = \B_{n-1} \B_n$). Applying $a \mapsto a^*$ to this, we find $\B_n^* \B_{n-1}^* T_{n-1} = q \B_n^* \B_{n-1}^*$.
On the other hand, we have $\B_{n}^\ast = 1 + \B^\ast_{n - 1}T_{n - 1} $ (by applying \eqref{eq.B*n=} to $n$ and to $n-1$ and comparing).
Hence,
\[
\B^\ast_n \B^\ast_n = \B_n^\ast(1 + \B^\ast_{n - 1}T_{n - 1}) = \B_n^\ast + \B_n^* \B_{n-1}^* T_{n-1}
= \B_n^\ast + \,q \B_n^* \B_{n-1}^* = \B^\ast_n(1 + q \B^\ast_{n - 1}).
\]
However, multiplying \eqref{eq:trivrec2} by $\B_n^*$ from the left, we find
\begin{align*}
    \B^\ast_n \B^\ast_{n, k} &=  \B^\ast_n \B^*_n \B^\ast_{n - 1, k - 1}
    =\B_n^\ast(1 + q\B^\ast_{n - 1})\B^\ast_{n - 1, k - 1}
    \qquad \text{(by the preceding equality)}\\
    & = \B^\ast_{n, k} + \,q\B^\ast_n (\B^\ast_{n - 1}\B^\ast_{n - 1, k - 1})
    \qquad \text{(by \eqref{eq:trivrec2})}     \\
    &= \B^\ast_{n, k} + \,q\B^\ast_n (q^{k - 1}\B^\ast_{n - 1, k} + [k - 1]_q \B^\ast_{n - 1, k - 1})\qquad \text{(by induction)}\\
    &= \B_{n,k}^* + \,q^k \B_{n,k+1}^* + \,q[k-1]_q\B_{n,k}^*
    \qquad \text{(by \eqref{eq:trivrec2} again)} \\
    &= q^{k}\B^\ast_{n, k + 1} + \,(1 + q[k - 1]_q)\B^\ast_{n, k}\\
    &= q^{k}\B^\ast_{n, k + 1} + \,[k]_q\B^\ast_{n, k}.
\end{align*}
Hence, we have proved Equation \eqref{eqn:ind-step}. 

We are now ready to prove the Lemma by induction on $k$.
As a base case, when  $k = 0$, both sides of the claimed equation are $1$.
Now, to step from $k$ to $k+1$,
we use ${k + 1 \choose 2} = k + {k \choose 2}$;
this gives us
\begin{align*}
    q^{k + 1 \choose 2} \B^\ast_{n, k + 1}
    &= q^{k \choose 2} \left(q^k\B^\ast_{n, k + 1}\right)
    = q^{k \choose 2} \left(\B_n^\ast \B_{n, k}^\ast - [k]_q \B^\ast_{n, k}\right)
    \qquad \text{(by \eqref{eqn:ind-step})} \\
    &=  \left(\B_n^\ast - [k]_q\right) \left( q^{k \choose 2}\B^\ast_{n, k}\right)
    =  \left(\B_n^\ast - [k]_q\right) \prod_{j = 0}^{k - 1}\left(\B^\ast_n - [j]_q\right)
    \qquad \text{(by the induction hypothesis)}
    \\
    &= \prod_{j = 0}^{k}\left(\B^\ast_n - [j]_q\right).
    \qedhere
\end{align*}
\end{proof}

To move on, we recall some linear algebra:

If $f$ is an endomorphism of a $\kk$-vector space $V$, and if $\eta \in \kk$ is a scalar, then $\Eig_\eta(f)$ shall mean the $\eta$-eigenspace of $f$, that is, the subspace of $V$ consisting of all $v \in V$ satisfying $f(v) = \eta v$.
(This is the zero space if $\eta$ is not an eigenvalue of $f$.)
From linear algebra, we need two basic properties of eigenspaces
(their proofs can be found in Appendix B of the arXiv version):

\begin{prop} \label{prop:eigenspace-dirsum}
Let $f$ be an endomorphism of a $\kk$-vector space $V$.
Let $S$ be a finite subset of $\kk$.
\begin{enumerate}
\item[(a)] Then,
\[
\ker \left(\prod_{\eta \in S} \left(f - \eta \id\right)\right)
= \bigoplus_{\eta \in S} \Eig_\eta(f).
\]
\item[(b)]
Assume that $V$ is finite-dimensional and
\begin{align}
\prod_{\eta \in S} \left(f - \eta \id\right) = 0.
\end{align}
Then, the endomorphism $f$ is diagonalizable, and its eigenvalues all belong to $S$.
\end{enumerate}
\end{prop}

We thus obtain the following.
\begin{cor}\label{cor:Bnk*-diagable}
Assume that the $q$-integers $[0]_q,[1]_q,\ldots,[n]_q$ are distinct.
Then:
\begin{enumerate}
\item[(a)] The element $\B_n^* \in \HH_n$ (that is, its right action on $\HH_n$) is diagonalizable, with eigenvalues
\[
[0]_q, [1]_q, \ldots, [n-2]_q, [n]_q \qquad \left(\text{skipping $[n-1]_q$}\right).
\]
Each eigenspace $\Eig_{[j]_q}(\B^\ast_n)$ (with $0 \leq j \leq n$) has dimension $\binom{n}{j} d_{n-j}$.
\item[(b)] The element $\B_{n,k}^\ast$ is diagonalizable for each $k \geq 0$.
\item[(c)] For $0 \leq k \leq n$, we have
\begin{equation}
\label{eqn:decomposing-B*nk-kernel}
    \ker \B_{n, k}^\ast = \bigoplus_{j = 0}^{k - 1}\Eig_{[j]_q}(\B_n^\ast).
\end{equation}
\item[(d)] For $0 \leq k \leq n$, we have
\[
\dim(\ker \B_{n,k}^\ast)
= \sum_{j=0}^{k-1} \binom{n}{j} d_{n-j}
=
\left|\{ \sigma \in \symm_n: \sigma \text{ has at most } k-1 \text{ fixed points} \}\right|.
\]
\end{enumerate}
\end{cor}

\begin{proof}
(a) Applying \eqref{eq:lem:B*nk-as-prod-of-B*n:main} to $k=n+1$, we find
$0 = \prod_{j = 0}^{n} \left(\B_n^\ast - [j]_q\right)$, since $\B^\ast_{n,n+1}=0$.
By Proposition~\ref{prop:eigenspace-dirsum}(b) (applied to $V = \HH_n$ and $f = \B_n^*$ and $S = \left\{[0]_q,[1]_q,\ldots,[n]_q\right\}$), this entails that $\B_n^*$ is diagonalizable (here we are using the distinctness of $[0]_q,[1]_q,\ldots,[n]_q$), and that all its eigenvalues belong to $\left\{[0]_q,[1]_q,\ldots,[n]_q\right\}$.
Hence, \eqref{eq:charpolyBn} shows that each eigenspace $\Eig_{[j]_q}(\B^\ast_n)$ (with $0 \leq j \leq n$) has dimension $\binom{n}{j} d_{n-j}$.
In particular, the eigenspace $\Eig_{[n-1]_q}(\B^\ast_n)$ has dimension $\binom{n}{n-1} d_1 = 0$ (since $d_1 = 0$), so that $[n-1]_q$ is not actually an eigenvalue.
This leaves only $[0]_q, [1]_q, \ldots, [n-2]_q, [n]_q$ as possible eigenvalues; and these indeed exist because the respective dimensions are positive (we have $d_k > 0$ for all $k \neq 1$).

(b) The equality \eqref{eq:lem:B*nk-as-prod-of-B*n:main} shows that $\B_{n,k}^\ast$ is a polynomial in $\B_n^\ast$.
Since $\B_n^*$ is diagonalizable (by part (a)), its polynomial $\B_{n,k}^*$ must be diagonalizable as well (with the same eigenbasis).

(c) Applying Proposition~\ref{prop:eigenspace-dirsum}(a) to $V = \HH_n$ and $f = \B_n^*$ and $S = \left\{[0]_q,[1]_q,\ldots,[k-1]_q\right\}$, we obtain
\[
\ker \left(\prod_{j=0}^{k-1} \left(\B_n^* - [j]_q\right)\right)
= \bigoplus_{j = 0}^{k - 1}\Eig_{[j]_q}(\B_n^\ast)
\]
(again using the distinctness of the $q$-integers).
In view of \eqref{eq:lem:B*nk-as-prod-of-B*n:main}, this is precisely the claim of part (c).

(d) This follows from (c) using the dimensions of the eigenspaces.
\end{proof}

\subsubsection{\label{subsubsec.kerRnk}The kernel of $\R_{n,k}$ on $\HH_n$}

Recall that 
\[ \R_{n,k} = \frac{1}{[k]!_q} \B_{n,k}^\ast \B_{n,k}. \]
The connection between the kernels of $\R_{n,k}$ and $\B_{n,k}^\ast$ lies in the case when $q \in \RR_{>0}$:

\begin{lem}
\label{lem:ker-im-pos}
Assume that $\kk \subseteq \RR$ and
$q\in\RR_{>0}$. Let $x\in\HH_n$ be arbitrary. Then,
\[
\ker\left(  x^{\ast}x\right)  =\ker\left(  x^{\ast}\right)
\qquad \text{ and } \qquad
\im\left(  x^{\ast}x\right)  =\im\left(  x\right)  .
\]
\end{lem}

\begin{proof}
From \cite[Lemma 6.1 (1)]{r2r1}, we know that
\begin{equation}
\text{if }a\in\HH_{n}\text{ satisfies }a^{\ast}a=0\text{, then }
a=0.
\label{pf.lem:ker-im-pos.1}
\end{equation}
This innocent-looking fact will be crucial to the present proof.

Clearly, $\ker\left(  x^{\ast}\right)  \subseteq\ker\left(  x^{\ast}x\right)
$. Let us now prove the reverse inclusion. Indeed, let $v\in\ker\left(
x^{\ast}x\right)  $. Then, $vx^{\ast}x=0$, so that $\left(  xv^{\ast}\right)
^{\ast}\left(  xv^{\ast}\right)  =\underbrace{v^{\ast\ast}}_{=v}x^{\ast
}xv^{\ast}=\underbrace{vx^{\ast}x}_{=0}v^\ast=0$. Hence,
\eqref{pf.lem:ker-im-pos.1} (applied to $a=xv^{\ast}$) yields $xv^{\ast}=0$.
Hence, $\left(  xv^{\ast}\right)  ^{\ast}=0^{\ast}=0$. In view of $\left(
xv^{\ast}\right)  ^{\ast}=v^{\ast\ast}x^{\ast}=vx^{\ast}$, this rewrites as
$vx^{\ast}=0$. That is, $v\in\ker\left(  x^{\ast}\right)  $. Since we have
proved this for each $v\in\ker\left(  x^{\ast}x\right)  $, we thus obtain
$\ker\left(  x^{\ast}x\right)  \subseteq\ker\left(  x^{\ast}\right)  $.
Combined with $\ker\left(  x^{\ast}\right)  \subseteq\ker\left(  x^{\ast
}x\right)  $, this entails $\ker\left(  x^{\ast}x\right)  =\ker\left(
x^{\ast}\right)  $. Thus, the first half of Lemma~\ref{lem:ker-im-pos} is proved.

It remains to show that $\im\left(  x^{\ast}x\right)  =\im\left(  x\right)  $.
Since $\im\left(  x^{\ast}x\right)  \subseteq\im\left(  x\right)  $ is
obvious, it will suffice to show that $\dim\left(  \im\left(  x^{\ast
}x\right)  \right)  \geq\dim\left(  \im\left(  x\right)  \right)  $.

Assume the contrary. Thus, $\dim\left(  \im\left(  x^{\ast}x\right)  \right)
<\dim\left(  \im\left(  x\right)  \right)  $, so that $\dim\left(  \im\left(
x\right)  \right)  >\dim\left(  \im\left(  x^{\ast}x\right)  \right)
=\dim\left(  \HH_{n}\right)  -\dim\left(  \ker\left(  x^{\ast}x\right)
\right)  $ by the rank-nullity theorem. Hence, $\dim\left(  \im\left(
x\right)  \right)  +\dim\left(  \ker\left(  x^{\ast}x\right)  \right)
>\dim\left(  \HH_{n}\right)  $. Therefore, the subspaces $\im\left(  x\right)
$ and $\ker\left(  x^{\ast}x\right)  $ of $\HH_{n}$ have a nonzero
vector $u$ in common. Consider this $u$. Thus, $u=wx$ for some $w\in
\HH_{n}$ (since $u\in\im\left(  x\right)  $) and $ux^{\ast}=0$ (since
$u\in\ker\left(  x^{\ast}x\right)  =\ker\left(  x^{\ast}\right)  $). From
$u=wx$, we obtain $u^{\ast}=\left(  wx\right)  ^{\ast}=x^{\ast}w^{\ast}$, so
that $uu^{\ast}=\underbrace{ux^{\ast}}_{=0}w^{\ast}=0$.

Set $a:=u^{\ast}$. Then, $a^{\ast}=u^{\ast\ast}=u$ and thus $a^{\ast
}a=uu^{\ast}=0$. Hence, \eqref{pf.lem:ker-im-pos.1} yields $a=0$.
Consequently, $a^{\ast}=0$, so that $u=a^{\ast}=0$, which contradicts $u$
being nonzero. This contradiction shows that our assumption was false, and the
proof of the lemma is complete.
\end{proof}

\begin{prop} \label{prop:imR=imB}
    Assume that $\kk \subseteq \RR$ and $q \in \RR_{>0}$.
    Then, each $k \geq 0$ satisfies
	\begin{align}
	\ker \R_{n,k} = \ker \B^*_{n,k}
	\qquad \text{ and } \qquad
	\im \R_{n, k} = \im \B_{n, k}.
    \label{eq:imR=imB}
	\end{align}
Moreover, in this case we have 
\begin{align}
      \label{eq:Rnimagecontainment}  \im \R_{n, 0} &\supseteq \im \R_{n, 1} \supseteq \im \R_{n, 2} \supseteq \cdots, \\ 
   \label{eq:Rnkernelcontainment} \ker \R_{n, 0} &\subseteq \ker \R_{n, 1} \subseteq \ker \R_{n, 2}\subseteq \cdots. 
\end{align}
\end{prop}

\begin{proof}
Let $k \geq 0$.
Recall that $\R_{n,k}$ equals $\B^*_{n,k} \B_{n,k}$ up to
a nonzero scalar factor (by \eqref{eq:Rnk:def}, since $q > 0$
entails $[k]!_q \neq 0$).
Hence, $\ker \R_{n,k} = \ker\tup{\B^*_{n,k} \B_{n,k}}
= \ker \B^*_{n,k}$ by the first equality in
Lemma~\ref{lem:ker-im-pos}, applied to
$x = \B_{n,k}$. Similarly,
$\im \R_{n,k} = \im\tup{\B^*_{n,k} \B_{n,k}}
= \im \B_{n,k}$ by the second equality in
Lemma~\ref{lem:ker-im-pos}, applied to
$x = \B_{n,k}$.
Thus \eqref{eq:imR=imB} is proved.

For each $k>0$, we have
$\im \B_{n,k-1} \supseteq \im \B_{n,k}$
by \eqref{eq:trivrecb1}. In view of
\eqref{eq:imR=imB}, this rewrites as
$\im \R_{n,k-1} \supseteq \im \R_{n,k}$.
This proves \eqref{eq:Rnimagecontainment}.
Similarly, we can prove \eqref{eq:Rnkernelcontainment}
(since \eqref{eq:trivrecb2} yields
$\ker \B_{n,k-1}^* \subseteq \ker \B_{n,k}^*$).
\end{proof}

\subsubsection{The kernel of $\R_{n,1}$ on Specht modules}
Finally, in order to construct a joint eigenbasis for the $\R_{n,k}$, we will need to understand the right action of the $\R_{n,k}$ on individual Specht modules. In this case, the kernels are more subtle. We describe  below the nullity of $\R_{n,1}$ acting on $S^\lambda$ proved in \cite{r2r1}; we will be able to generalize this to a statement for all $k$ in Corollary~\ref{cor:dimkernelRk} once we have computed the spectra of $\R_{n,k}$.

If $\frakt$ is a standard tableau with $n$ entries, then the \emph{descents} of $\frakt$ are the numbers $i \in \ive{n-1}$ that appear further north (and thus weakly east) than $i+1$ in $\frakt$.
The set of all descents of $\frakt$ is denoted by $\Des\tup{\frakt}$.
A standard tableau $\frakt$ with $n$ entries is said to be a \emph{desarrangement tableau} if the smallest element of $\ive{n} \setminus \Des\tup{\frakt}$ is even. By convention, the empty tableau is a desarrangement tableau.

For instance, here are the five standard tableaux $\frakt$ of shape $\tup{2,2,1}$ with their descent sets $\Des\tup{\frakt}$:
\begin{center}
\begin{tabular}{r||c|c|c|c|c}
& & & & & \\
$\frakt$\phantom{u} & $\ytableaushort{14,25,3}$ & $\ytableaushort{13,25,4}$ & $\ytableaushort{13,24,5}$ & $\ytableaushort{12,35,4}$ & $\ytableaushort{12,34,5}$
\\
& & & & & \\
$\Des\tup{\frakt}$\phantom{u} & $\set{1,2}$ & $\set{1,3}$ & $\set{1,3,4}$ & $\set{2,3}$ & $\set{2,4}$ \\
& & & & & \\
desarrangement tableau? & no & yes & yes & no & no \\
& & & & &
\end{tabular}
\end{center}
\begin{definition}
    For any partition $\mu$, we define its \emph{desarrangement number} to be $d^\mu$, the number of desarrangement tableaux of shape $\mu$.
\end{definition}

Note that $d^\emptyset =1$.
We note the following relationship between the desarrangement numbers $d_n$ and $d^\mu$ due to \cite[Proposition 2.3]{ReinerWebb}:
\[ \sum_{\mu \vdash n} f^\mu d^\mu = d_n. \]

The relevance of desarrangement numbers is that they describe the nullity of $\R_{n,1}$, which we will show in Corollary~\ref{thm:eigenbasis} is closely tied to the simultaneous eigenbasis for $\R_{n,k}$. The following is proved in \cite{r2r1}:
\begin{theorem}[Proposition 7.1 of \cite{r2r1}]\label{thm:dimkernelR1}
Whenever $q \in \RR_{>0}$, we have that 
\begin{align*}
    \dim \bigg( \ker \left(\R_{\abs{\mu}, 1} \vert_{S^\mu} \right) \bigg) &= d^\mu.
\end{align*}
\end{theorem}

\subsection{The crucial eigenvalue recursion} \label{ssec.crucialevalrecursion}
Finally, we can use the crucial \emph{algebraic} recursion in Theorem~\ref{thm:crucialrecursion} to derive our \emph{eigenvalue} recursion. To do so, we first collect a few facts. Recall that we have now established:
\begin{itemize}
    \item By Theorem~\ref{thm:crucialrecursion}, we know that   $\B_n  \R_{n, k}  = \Big(q^{k}\R_{n-1, k}  + \Big( \ive{n+1-k}_q + q^{n + 1 - k}J_n \Big) \R_{n - 1, k - 1} \Big)\B_n$;

    \item By \eqref{eq:Rnkernelcontainment}, when $q \in \RR_{>0}$, we know that if $v \in \ker \R_{n,k-1}$, then $v \in \ker \R_{n,k}$. 
\end{itemize}

We shall also need the following fact from \cite{r2r1} (see the expansion of $v_3$ within the proof of \cite[Theorem 5.4]{r2r1}; a detailed proof is found in Appendix B to the arXiv version of this paper):

\begin{lem}
\label{lem:Slam-glue}
Let $\lambda\vdash n$ and $\lambda^{\prime}\lessdot\lambda$. Let $v\in S^{\lambda^{\prime}}$,
and let $r$ be the row in which the box of $\lambda \sm \lambda'$ lies.
Then,
\[
v\,\Phi_{r}\,p_{\lambda}\,J_{n}=\left[  \content\left(  \lambda
\sm\lambda^{\prime}\right)  \right]  _{q}\,v\,\Phi_{r}\,p_{\lambda}.
\]
\end{lem}

Using these facts, we will prove that the eigenvalues of $\R_{n,k}$ on a given Specht module $S^\lambda$ can be obtained recursively from the eigenvalues of $\R_{n-1,k}$ and $\R_{n-1,k-1}$ on $S^{\lambda'}$, for $\lambda' \lessdot \lambda$.

Specifically, we use Theorem~\ref{thm:crucialrecursion} to show how one obtains an eigenvector for $\R_{n,k}$ from a joint eigenvector of $\R_{n-1,k}$ and $\R_{n-1,k-1}$. Most importantly, the corresponding eigenvalue changes in a predictable way. 

\begin{theorem}\label{prop:eval-recursion-1-step}
Assume $k \geq 1$. Let $\lambda \vdash n$ and ${\lambda}' \lessdot \lambda$ be partitions, and let $r$ be the row in which the box $\lambda \sm \lambda'$ lies.  Suppose $v \in S^{\lambda'}$ is an eigenvector\footnote{We do not require eigenvectors to be nonzero. Thus, the zero vector is an eigenvector of any linear map with any eigenvalue. As a consequence, when we say that some vector is an eigenvector of some operator with eigenvalue $z$, we are not actually saying that $z$ is an eigenvalue of the operator; this conclusion can only be drawn if our vector is nonzero.}
	of $\R_{n-1, k}$ with eigenvalue $\varepsilon_k$ and of $\R_{n-1, k - 1}$ with eigenvalue $\varepsilon_{k-1}$.    Then $v \, \Phi_r   \B_n   p_\lambda$ is in $S^\lambda$ and is an eigenvector for $\R_{n, k}$ with eigenvalue
	\begin{equation}\label{eq:evalrecursion}
	q^k \varepsilon_k + [n + 1 - k + \content(\lambda \sm \lambda')]_q  \ \varepsilon_{k - 1}.
	\end{equation}
\end{theorem}

\begin{proof}
To see that $v \ \Phi_r \B_n \ p_\lambda \in S^\lambda$, see \cite[Theorem 5.4]{r2r1}\footnote{Specifically, \cite[Theorem 5.4]{r2r1}, which writes $\Phi_{\lambda \sm \lambda'}$ for our $\Phi_r$, $\B_n(q)$ for our $\B_n$, and $u'$ for our $v$. That work makes an additional assumption that $v$ be an eigenvector of $\R_{n-1,1}$, but this assumption is not used in the proof of $v  \Phi_r   \B_n   p_\lambda \in S^\lambda$.}.
It remains to prove the eigenvector claim.

Recall that $p_\lambda$ is central in $\HH_n$, hence commutes with both $\B_n$ and $\R_{n,k}$.
Thus,
    \begin{align*}
        \left({v \ \Phi_r \B_n \ p_\lambda}\right)  \ \R_{n, k}
		&= v \ \Phi_r \ \left(\B_n \R_{n, k}\right)  \ p_\lambda
		\\
        &= v  \ \Phi_r \ \Bigg(q^k \R_{n-1, k} + \left([n + 1 - k]_q + q^{n + 1 - k}J_n\right)\R_{n-1, k - 1}\Bigg) \B_n \ p_\lambda
    \end{align*}
	by Theorem~\ref{thm:crucialrecursion}\footnote{Theorem~\ref{thm:crucialrecursion} was stated with the requirement $k \leq n$. However, this requirement is unnecessary (with our convention for $q$-integers $\ive{i}_q$ with $i < 0$), because both sides of Theorem~\ref{thm:crucialrecursion} vanish whenever $k > n$ (by \eqref{eq:Rnk:zero}).}.
We need to show that this equals
\[
	\left(q^k \varepsilon_k + [n + 1 - k + \content(\lambda \sm \lambda')]_q  \ \varepsilon_{k - 1}\right) {v \ \Phi_r \ \B_n  \ p_\lambda}.
\]
Since $p_\lambda$ commutes with $\B_n$, this is tantamount to showing that
\begin{align}
&v  \ \Phi_r \ \Bigg(q^k \R_{n-1, k} + \left([n + 1 - k]_q + q^{n + 1 - k}J_n\right)\R_{n-1, k - 1}\Bigg) p_\lambda \ \B_n \nonumber\\
&= \left(q^k \varepsilon_k + [n + 1 - k + \content(\lambda \sm \lambda')]_q  \ \varepsilon_{k - 1}\right) v \ \Phi_r \ p_\lambda \ \B_n.
\label{pf:prop:eval-recursion-1-step:goal0}
\end{align}


To show this, we note that the
eigenvector condition on $v$ yields $v\R_{n-1,k}=\varepsilon_{k}v$ and
$v\R_{n-1,k-1}=\varepsilon_{k-1}v$. We also recall that
$\R_{n-1,k-1}$ and $\R_{n-1,k}$ commute with the operator $\Phi_{r}$ (by
\eqref{eq:Phi-commut}). Hence,
\[
v\,\Phi_{r}\,\R_{n-1,k}\,p_{\lambda}=\underbrace{v\,\R_{n-1,k}}_{=\varepsilon
_{k}v}\,\Phi_{r}\,p_{\lambda}=\varepsilon_{k}v\,\Phi_{r}\,p_{\lambda}%
\]
and%
\[
v\,\Phi_{r}\,\R_{n-1,k-1}\,p_{\lambda}=\underbrace{v\,\R_{n-1,k-1}%
}_{=\varepsilon_{k-1}v}\,\Phi_{r}\,p_{\lambda}=\varepsilon_{k-1}v\,\Phi
_{r}\,p_{\lambda}%
\]
and (since Lemma~\ref{lem:TiHn-1} shows that
$\R_{n-1,k-1}$ commutes with $J_{n}$) also
\begin{align*}
v\,\Phi_{r}\,J_{n}\,\R_{n-1,k-1}\,p_{\lambda}  & =v\,\Phi_{r}\,\R_{n-1,k-1}%
\,J_{n}\,p_{\lambda}
=\underbrace{v\,\Phi_{r}\,\R_{n-1,k-1}\,p_{\lambda}}_{=\varepsilon
_{k-1}v\,\Phi_{r}\,p_{\lambda}}\,J_{n}\qquad\left(  \text{since }p_{\lambda
}\text{ is central}\right)  \\
& =\varepsilon_{k-1}v\,\Phi_{r}\,p_{\lambda}\,J_{n}
=\varepsilon_{k-1}\,[\content(\lambda\sm\lambda^{\prime}%
)]_{q}\,v\,\Phi_{r}\,p_{\lambda}\qquad\left(  \text{by
Lemma~\ref{lem:Slam-glue}}\right)  .
\end{align*}
Thus, distributivity yields%
\begin{align*}
&  v\,\Phi_{r}\,\Bigg(q^{k}\R_{n-1,k}+\left(  [n+1-k]_{q}+q^{n+1-k}%
J_{n}\right)  \R_{n-1,k-1}\Bigg)p_{\lambda}\\
&  =q^{k}\,\underbrace{v\,\Phi_{r}\,\R_{n-1,k}\,p_{\lambda}}_{=\varepsilon
_{k}v\,\Phi_{r}\,p_{\lambda}}+\,[n+1-k]_{q}\,\underbrace{v\,\Phi
_{r}\,\R_{n-1,k-1}\,p_{\lambda}}_{=\varepsilon_{k-1}v\,\Phi_{r}\,p_{\lambda}%
}+\,q^{n+1-k}\underbrace{v\,\Phi_{r}\,J_{n}\,\R_{n-1,k-1}\,p_{\lambda}%
}_{=\varepsilon_{k-1}\,[\content(\lambda\sm\lambda^{\prime}%
)]_{q}\,v\,\Phi_{r}\,p_{\lambda}} \\
&  =q^{k}\,\varepsilon_{k}v\,\Phi_{r}\,p_{\lambda}+\,[n+1-k]_{q}%
\,\varepsilon_{k-1}v\,\Phi_{r}\,p_{\lambda}+\,q^{n+1-k}\,\varepsilon
_{k-1}\,[\content(\lambda\sm\lambda^{\prime})]_{q}\,v\,\Phi
_{r}\,p_{\lambda}\\
&  =\left(  q^{k}\varepsilon_{k}+
\underbrace{\big([n+1-k]_{q}+q^{n+1-k}\,[\content%
(\lambda\sm\lambda^{\prime})]_{q}\big)}_{
=[n+1-k+\content(\lambda\sm\lambda^{\prime})]_{q}}
\,\varepsilon_{k-1}\right)
{v\,\Phi_{r}\,\,p_{\lambda}}\\
&  =\left(  q^{k}\varepsilon_{k}+[n+1-k+\content(\lambda\sm
\lambda^{\prime})]_{q}\,\varepsilon_{k-1}\right)  {v\,\Phi_{r}\,p_{\lambda}}.
\end{align*}
Upon right multiplication with $\B_n$, this equality yields
\eqref{pf:prop:eval-recursion-1-step:goal0}.
Thus, our proof is complete.
\end{proof}

In \cite[Theorem 7.2]{r2r1}, the authors construct an eigenbasis for the right action of $\R_{n, 1}$ on $S^\lambda$. 



\begin{theorem}[\cite{r2r1}, Theorem 7.2]\label{def:basis}
    Assume that $\kk \subseteq \RR$ and $q \in \RR_{>0}$. Let $\lambda \vdash n$.
Fix a basis $\kappa_\mu$ of the vector space $ \ker \left(\R_{\abs{\mu}, 1} \vert_{S^\mu} \right)$ for each partition $\mu \subseteq \lambda$.
For any partition $\mu \subseteq \lambda$ and any $u \in S^\mu$, we set
\[
y_{\mu(u)}^{\lambda} := u\  \Phi_{\lambda \sm \mu}  \  \B_{n, n-\abs{\mu}} \ p_{\lambda} 
\]
and define
\[
\basis_\lambda := \big \{ y_{\mu(u)}^{\lambda}  : u \in \kappa_\mu \text{ and } \lambda \sm \mu \textrm{ is a horizontal strip}  \big \}.
\]
Then $\basis_\lambda$ is an $\R_{n,1}$-eigenbasis for $S^\lambda$.
\end{theorem}

Proving Theorem~\ref{def:basis} took significant work; we refer the reader to \cite{r2r1} for details. However, one important aspect to highlight is that 
$\basis_\lambda$ is obtained by iteratively applying the lifting process in Theorem~\ref{prop:eval-recursion-1-step}. Specifically, fix a horizontal strip $\lambda \sm \mu$ and a kernel vector $u \in S^\mu$ (a 0-eigenvector for $\R_{j,1}$ where $j=\abs{\mu}$). Consider the chain of partitions \[ \mu  = \lambda^{(0)} \lessdot \lambda^{(1)} \lessdot \lambda^{(2)} \lessdot \cdots \lessdot \lambda^{(n-j)} = \lambda, \] obtained by adding cells row by row from top to bottom, where the cells in each row are added left to right. If we repeatedly iterate the lifting process in Theorem~\ref{prop:eval-recursion-1-step} along this path, then we obtain the vector \begin{equation}\label{eqn:v-first}
    v := u \cdot \Phi_{\lambda^{(1)} \sm \lambda^{(0)}} \B_{j+1} p_{\lambda^{(1)}} \cdots \Phi_{\lambda^{(n-j)} \sm \lambda^{(n-j-1)}} \B_{n} p_{\lambda^{(n-j)}},
\end{equation} which is necessarily a $\R_{n, 1}$-eigenvector in $S^\lambda$. Using the centrality of the $p_{\lambda^{(j)}}$ in $\HH_{|\lambda^{(j)}|}$ and \cite[Lemma 3.6]{r2r1}, one can rearrange the terms to rewrite $v$ as $v =  u\  \Phi_{\lambda \sm \mu}  \ p_{\lambda^{(1)}}\cdots p_{\lambda^{(n - j)}}  \B_{n, n-\abs{\mu}}$. Lemma 6.9 in \cite{r2r1} then shows that 
\begin{equation}\label{eqn:v=y}
    v =   u\  \Phi_{\lambda \sm \mu}  \ p_{\lambda^{(1)}}\cdots p_{\lambda^{(n - j)}}  \B_{n, n-\abs{\mu}}= y_{\mu(u)}^\lambda.
\end{equation} 
It is also shown that the choice of path does not matter, and any sequence of partitions from $\mu$ to $\lambda$ results in a scalar multiple of $y_{\mu(u)}^\lambda$ \cite[Lemma 6.7]{r2r1}.

The pathwise construction \eqref{eqn:v-first} of $y_{\mu(u)}^{\lambda}$ can be conveniently rephrased as a recursion:

\begin{lem}
\label{lem:y-rec}
Assume that $\kk \subseteq \RR$ and $q\in 
\RR_{>0}$. Let $\lambda \vdash n$ and two further partitions $\mu
\subseteq \lambda ^{\prime }\lessdot \lambda $ be such that $\lambda
\sm \mu$ is a horizontal strip and that the box in $\lambda \sm
\lambda^{\prime }$ is the rightmost box of the last nonempty row of
$\lambda \sm \mu$. Let $u\in S^{\mu}$. Then,
\begin{equation}
y_{\mu \left( u\right) }^{\lambda }
= y_{\mu \left( u\right) }^{\lambda^{\prime }}
\,\Phi _{\lambda \sm \lambda ^{\prime }}
\, \mathcal{B}_{n}\,p_{\lambda}.
\label{eq:lem:y-rec:eq}
\end{equation}
\end{lem}

This is just a recursive form of 
\eqref{eqn:v-first} (note that $\lambda ^{\prime }=\lambda ^{(n-j-1)}$); but
a more detailed, standalone proof is given in Appendix B of the arXiv
version.

We illustrate this recursive realization of $y_{\mu(u)}^{\lambda}$ with the following small example. 
\begin{example}  Consider the horizontal strip $\lambda \sm \mu$ when $\mu = (1,1)$ and $\lambda = (2,1,1)$. The Specht module $S^{(1,1)}$ belongs to the kernel of $\R_{2,1}$ and is one-dimensional, spanned by a vector $u$.

By definition, \begin{align*}
y_{\mu(u)}^{\lambda} &:= u \ \underbrace{\Phi_{(2,1, 1) \sm (1,1)}}_{\text{append }13} \ \B_{4, 2} \ p_{(2,1,1)}.
\end{align*} 

However, as explained above, $y_{\mu(u)}^{\lambda}$ can be obtained by iterating the lifting process in Theorem~\ref{prop:eval-recursion-1-step} along the path $(1,1)\to (2,1) \to (2,1,1)$. Specifically,
\begin{equation}\label{eqn:vec1}
  \underbrace{u}_{\text{eigenvector for $\R_{2, 1}$}} \longrightarrow  \underbrace{u  \ \Phi_1 \ \B_3 \ p_{(2,1)}}_{\text{eigenvector for $\R_{3, 1}$}} \longrightarrow \underbrace{u \ \Phi_1 \ \B_3 \ p_{(2,1)} \ \Phi_3 \ \B_4 \ p_{(2,1,1)}}_{\text{eigenvector for $\R_{4, 1}$}} = y^{(2,1,1)}_{(1,1)(u)}. 
\end{equation}



\end{example}

Importantly, Theorem~\ref{prop:eval-recursion-1-step} shows the one-step lifting process produces eigenvectors for $\R_{n, k}$ for any $k$, not just $k = 1$.
Applying Theorem~\ref{prop:eval-recursion-1-step} to Lemma~\ref{lem:y-rec},
we see that each  $y_{\mu(u)}^\lambda \in \basis_\lambda$ is in fact also an eigenvector for $\R_{n,k}$. We thus obtain the following corollary
(see \eqref{eq:thm:big.eigenvalues:eq} below for more details).

\begin{cor}
\label{thm:eigenbasis}
Assume that $\kk \subseteq \RR$ and $q \in \RR_{>0}$.
For any $\lambda \vdash n$, the elements in $\basis_\lambda$ form a simultaneous $\R_{n, k}$-eigenbasis for $S^\lambda$ and $0 \leq k \leq n$.
\end{cor}


\begin{remark}
    Even though $\kappa_\mu$ is a basis for the kernel of $\R_{n,1}$ acting on $S^\mu$, it is not necessarily a basis for the kernel of $\R_{n,k}$ acting on $S^\mu$, because the latter kernel is larger. Thus $\basis_\lambda$ contains some kernel eigenvectors for $\R_{n,k}$ that are not in the kernel of $\R_{n,1}$. In Section~\ref{sec.eigenvalues}, we will prove a formula for the eigenvalues of $\R_{n,k}$ that makes it clear when they are zero.
\end{remark}

\section{Formulas for the spectra of $\R_{n,k}$}\label{sec.eigenvalues}
In this section, we use the crucial eigenvalue recursion in Theorem~\ref{prop:eval-recursion-1-step} to obtain a closed form expression for all eigenvalues of $\R_{n,k}$ (\S \ref{ssec.eigenvalueexpression}), to prove the eigenvalues of $\R_{n,k}$ lie in $\ZZ_{\geq 0}[q]$ (\S \ref{ssec.positivity}), to compute the characteristic polynomial of the action of $\R_{n,k}$ on $\HH_n$ (\S \ref{ssec.charpoly}), and to prove multiplicative formulas for eigenvalues in the special case they come from hooks, including the second-largest eigenvalue (\S \ref{ssec.hooks}).

\subsection{Closed expression for eigenvalues}\label{ssec.eigenvalueexpression}
Theorem~\ref{prop:eval-recursion-1-step} gives a method of ``lifting'' a simultaneous eigenvector of $\R_{n-1,k-1}$ and $\R_{n-1,k}$ to $\R_{n,k}$ in a way that predictably changes the eigenvalue. Our next goal is to show that the recursion in \eqref{eq:evalrecursion} allows us to recover a closed expression for the eigenvalues of $\R_{n,k}$. To do so, we will introduce a family of symmetric functions that satisfy the same recursion as \eqref{eq:evalrecursion}. 

Recall that $e_k$ denotes the $k$-th elementary symmetric polynomial, while $h_k$ denotes the $k$-th complete homogeneous symmetric polynomial. (This includes $h_0 = e_0 = 1$ and $h_k = e_k = 0$ for all negative $k$.)

\begin{definition}\label{def:supersym}
Let $A$ be a commutative ring.
For $a\geq 0$ and finite lists $\sfX = (x_1, \ldots, x_c)$ and $\sfY = (y_1, \ldots, y_d)$ of elements of $A$, define the \emph{$a$-th supersymmetric polynomial}
   \[  e_a(\sfX / \sfY):= \sum_{j=0}^a e_j(\sfX) \cdot h_{a-j}(\sfY) =  \sum_{j=0}^a e_j(x_1, \ldots, x_{c}) \cdot h_{a-j}(y_1, \ldots, y_d).\]

(The lists $\sfX$ and $\sfY$ can be unordered, since the polynomials $e_j$ and $h_{a-j}$ are symmetric.)
\end{definition}

Note that this generalizes both the $e_k$ and the $h_k$; indeed, if $\varnothing$ denotes the empty list, then $e_k(\sfX / \varnothing) = e_k(\sfX)$ and $e_k(\varnothing / \sfY) = h_k(\sfY)$.


We give two more formulas for these polynomials, for which
we claim no novelty, but provide self-contained proofs in Appendix B in the
arXiv version of this paper. The first formula is a recurrence:

\begin{prop}\label{prop:supersymrecursion}
Let $a,c,d$ be positive integers, and let
$X_c=(x_1,\ldots,x_c)$ and $Y_d=(y_1,\ldots,y_d)$ be lists of elements
of a commutative ring $A$. Let $X_{c-1} =(x_1,\ldots,x_{c-1})$
and $Y_{d-1}=(y_1,\ldots,y_{d-1})$. Then,
\[
    e_a(X_c/Y_d)
    = e_a(X_{c-1}/Y_{d-1})
      + (x_c+y_d)e_{a-1}(X_{c-1}/Y_d).
\]
\end{prop}

The second formula (a specialization of \cite[(6.20)]{MacdonaldSTV}
for a one-row partition $\lambda = (a)$)
is explicit, albeit conditional on $d=c-a+1$ (this
condition can always be achieved by padding the lists $\sfX$
and $\sfY$ with zeroes, which has no effect on
$e_a(\sfX/\sfY)$).

\begin{prop}\label{prop:superproduct}
Let $A$ be a commutative ring.
Let $0 \leq a \leq c$ and $d=c-a+1$.
Let $\sfX = (x_1, \ldots, x_c)$ and $\sfY = (y_1, \ldots, y_d)$ be lists of elements of $A$.
Then,
\begin{equation} \label{eqn:rewrite-skew-es}
    e_a(\sfX/\sfY) = \sum_{1 \leq j_1 < j_2 < \cdots < j_a \leq c}
       \ \ \prod_{m=1}^{a} \tup{x_{j_m} + y_{j_m-m+1}}.
\end{equation}
(For $a=0$, this is just saying $1 = 1$.)
\end{prop}

Our key insight is that the recursion satisfied by the eigenvalues of $\R_{n,k}$ in \eqref{eq:evalrecursion} can be realized as a specialization of Proposition~\ref{prop:supersymrecursion}. In particular, we define the following lists, which will be functions of $k$ and the horizontal strip $\lambda \sm \mu$ for $|\lambda|=n$.

\begin{definition}
    For any horizontal strip $\lambda \sm \mu$ with $|\lambda| =n$ and any $k \geq 0$, we define the unordered lists
    \begin{align}
        \sfX_{\lambda \sm \mu}^{(k)}:=&\ ( q^{n-k+1}[\content(i,j)]_q : (i,j) \in \lambda \sm \mu ) \label{eq:specialX}, \\
        \sfY_{\lambda \sm \mu}^{(k)}:=&\ ( {{q^{n - k + 1 - j}[j]_q}}: |\mu| < j \leq n - k + 1), \label{eq:specialY}
    \end{align}
    of respective lengths $|\sfX_{\lambda \sm \mu}^{(k)}| = n-|\mu|$ and $|\sfY_{\lambda \sm \mu}^{(k)}| = \max \{0, n-k+1-|\mu|\}$.
    We further define the polynomial
    \begin{align}
        \eigen_{\lambda\sm \mu}^{(k)} :=
            e_k(  \sfX_{\lambda \sm \mu}^{(k)}/  \sfY_{\lambda \sm \mu}^{(k)}).
            \label{eq:eigenformula}
    \end{align}
\end{definition}

Note that negative powers of $q$ may appear in the entries of $\sfX_{\lambda \sm \mu}^{(k)}$ and thus also in \eqref{eq:eigenformula}.
(For example, for $n=k=3$, $\lambda=(1,1,1)$ and $\mu=(1,1)$, we have
$\sfX_{\lambda\sm\mu}^{(3)}=(-1-q^{-1})$.)
Hence, \eqref{eq:eigenformula} is not literally well-defined for $q = 0$.
However, Corollary~\ref{cor:eigen0forbigk} and Theorem~\ref{lem:positive} show that the resulting expression for $\eigen_{\lambda\sm \mu}^{(k)}$ lies in $\ZZ[q]$.
Thus, all negative powers of $q$ cancel, and \eqref{eq:eigenformula} becomes universally
true if its right-hand side is first expanded into a polynomial in $q$ and only then evaluated.

\begin{cor}\label{cor:eigen0forbigk}
Suppose $|\lambda\sm\mu|<k$. Then $\eigen_{\lambda\sm\mu}^{(k)}=0$.
\end{cor}

\begin{proof}
Let $n = \abs{\lambda}$.
Then, $n - \abs{\mu} = \abs{\lambda \sm \mu} < k$, so $\abs{\mu} \geq n-k+1$ and thus
$\sfY_{\lambda \sm \mu}^{(k)} = \emptyset$.
Hence, \eqref{eq:eigenformula} rewrites as
\[ \eigen_{\lambda\sm \mu}^{(k)} = e_k(\sfX_{\lambda \sm \mu}^{(k)}/ \emptyset) = e_k(\sfX_{\lambda \sm \mu}^{(k)}) = 0,
\qquad \text{ since } |\sfX_{\lambda \sm \mu}^{(k)}| = |\lambda \sm \mu| < k.
\qedhere
\]
\end{proof}


We will prove that all eigenvalues of $\R_{n,k}$ are precisely of the form $\eigen_{\lambda \sm \mu}^{(k)}$.

\begin{theorem}\label{thm:big.eigenvalues}
Let $k \geq 0$. All eigenvalues of $\R_{n,k}$ are of the form
$\eigen_{\lambda\sm \mu}^{(k)}$ for
horizontal strips $\lambda \sm \mu$ with $|\lambda| = n$.
Specifically, for each $y_{\mu(u)}^{\lambda}  \in \basis_\lambda$, we have
\begin{align}
y_{\mu(u)}^{\lambda} \cdot \R_{n,k} = \eigen_{\lambda\sm \mu}^{(k)} \cdot y_{\mu(u)}^{\lambda}.
\label{eq:thm:big.eigenvalues:eq}
\end{align}
\end{theorem}

\begin{proof}
We first focus on proving \eqref{eq:thm:big.eigenvalues:eq}. Note that both
$\R_{n,0}$ and $\eigen_{\lambda\sm\mu}^{(0)}$ equal $1$ (indeed,
$\eigen_{\lambda\sm\mu}^{(0)}=1$ follows from the fact that
$e_{0}(\mathsf{X}/\mathsf{Y})=1$ for any lists $\mathsf{X}$ and $\mathsf{Y}$).
Hence, \eqref{eq:thm:big.eigenvalues:eq} holds for $k=0$ (since both sides
simplify to $y_{\mu(u)}^{\lambda}$ in this case).
Thus, we assume without loss of generality that $k \geq 1$.

We prove \eqref{eq:thm:big.eigenvalues:eq} by induction on $\left\vert
\lambda\sm\mu\right\vert =n-\left\vert \mu\right\vert $, beginning with the \textit{base case} $\left\vert \lambda\sm
\mu\right\vert =0$. In this case, we have $\mu=\lambda$
and $\left\vert \mu\right\vert = n$, so that
\[
y_{\mu(u)}^{\lambda}=y_{\mu\left(  u\right)  }^{\mu}=u\ \underbrace{\Phi
_{\mu\sm\mu}}_{=\id}%
\ \underbrace{\B_{n,n-\left\vert \mu\right\vert }}_{=\mathcal{B}_{n,n-n}%
=1}\ p_{\lambda}=u\ p_{\lambda}=u
\]
(since $u\in S^{\mu}=S^{\lambda}$, but $p_{\lambda}$ is the central projector
on the $S^{\lambda}$-isotypic component of $\mathcal{H}_{n}$). Thus,
\[
y_{\mu(u)}^{\lambda} = u \in \kappa_{\mu}
\subseteq \ker\R_{\left\vert \mu\right\vert
,1} = \ker\R_{n,1}
\subseteq\ker\R_{n,k}\ \ \ \ \ \ \ \ \ \ \left(
\text{by \eqref{eq:Rnkernelcontainment}, since }1\leq k\right)  ,
\]
so that $y_{\mu(u)}^{\lambda}\cdot\R_{n,k}=0$. On the other hand, 
$\sfX_{\lambda\sm\lambda}^{(k)}=\varnothing$ and $\mathsf{Y}_{\lambda
\sm\lambda}^{(k)}=\varnothing$, so we have
$\eigen_{\lambda\sm\lambda}^{(k)}
=e_{k}(\mathsf{X}_{\lambda\sm\lambda}^{(k)}
/\mathsf{Y}_{\lambda\sm\lambda}^{(k)})=e_{k}\left(
\varnothing/\varnothing\right)  =0$ (since $k\geq1$). Hence,
\eqref{eq:thm:big.eigenvalues:eq} again boils down to $0=0$, and the base case
is proved.

We come to the \textit{induction step}, where we let $\left\vert
\lambda\sm\mu\right\vert $ be positive. We let $\lambda^{\prime}%
\lessdot\lambda$ be the partition obtained from $\lambda$ by removing the
rightmost box of the last nonempty row of $\lambda\sm\mu$. Then,
$\mu\subseteq\lambda^{\prime}\lessdot\lambda$, and the induction hypothesis
shows that%
\begin{align}
y_{\mu(u)}^{\lambda^{\prime}}\cdot\R_{n-1,k}  & =\eigen_{\lambda^{\prime
}\sm\mu}^{(k)}\cdot y_{\mu(u)}^{\lambda^{\prime}}\qquad\text{and}%
\label{pf:thm:big.eigenvalues:ih}\\
y_{\mu(u)}^{\lambda^{\prime}}\cdot\R_{n-1,k-1}  & =\eigen_{\lambda
^{\prime}\sm\mu}^{(k-1)}\cdot y_{\mu(u)}^{\lambda^{\prime}%
}.\label{pf:thm:big.eigenvalues:ih-1}%
\end{align}

Let $r$ be the row in which the box $\lambda\sm\lambda^{\prime}$ lies.
The equalities
\eqref{pf:thm:big.eigenvalues:ih} and \eqref{pf:thm:big.eigenvalues:ih-1} show
that $y_{\mu(u)}^{\lambda^{\prime}}\in S^{\lambda^{\prime}}$ is an eigenvector
of $\R_{n-1,k}$ with eigenvalue $\eigen_{\lambda^{\prime}\sm\mu
}^{(k)}$ and of $\R_{n-1,k-1}$ with eigenvalue $\eigen_{\lambda^{\prime
}\sm\mu}^{(k-1)}$. Hence, Proposition~\ref{prop:eval-recursion-1-step}
(applied to $v=y_{\mu(u)}^{\lambda^{\prime}}$ and $\varepsilon_{k}%
=\eigen_{\lambda^{\prime}\sm\mu}^{(k)}$ and $\varepsilon
_{k-1}=\eigen_{\lambda^{\prime}\sm\mu}^{(k-1)}$) yields that
$y_{\mu(u)}^{\lambda^{\prime}}\,\Phi_{r}\,\B_{n}\,p_{\lambda}$ is in
$S^{\lambda}$ and is an eigenvector for $\R_{n,k}$ with eigenvalue
\[
q^{k}\eigen_{\lambda^{\prime}\sm\mu}^{(k)}+[n+1-k+\content%
(\lambda\sm\lambda^{\prime})]_{q}\ \eigen_{\lambda^{\prime
}\sm\mu}^{(k-1)}.
\]
Since
\begin{align*}
y_{\mu(u)}^{\lambda^{\prime}}\,\Phi_{r}\,\B_{n}\,p_{\lambda}  & =y_{\mu\left(
u\right)  }^{\lambda^{\prime}}\,\Phi_{\lambda\sm\lambda^{\prime}%
}\,\mathcal{B}_{n}\,p_{\lambda}\ \ \ \ \ \ \ \ \ \ \left(  \text{because
}r\text{ is the row of }\lambda\sm\lambda^{\prime}\text{, so }
\Phi_r = \Phi_{\lambda \sm \lambda'}\right) \\
& =y_{\mu\left(  u\right)  }^{\lambda}\ \ \ \ \ \ \ \ \ \ \left(  \text{by
\eqref{eq:lem:y-rec:eq}}\right)  ,
\end{align*}
this shows that $y_{\mu\left(  u\right)  }^{\lambda}$ is an eigenvector for
$\R_{n,k}$ with eigenvalue
\[
q^{k}\eigen_{\lambda^{\prime}\sm\mu}^{(k)}+[n+1-k+\content%
(\lambda\sm\lambda^{\prime})]_{q}\ \eigen_{\lambda^{\prime
}\sm\mu}^{(k-1)}.
\]
If we can show that this eigenvalue equals $\eigen_{\lambda\sm\mu
}^{\left(  k\right)  }$, then \eqref{eq:thm:big.eigenvalues:eq} will
follow, and the induction will be complete.

So we must prove that
\begin{equation}
q^{k}\eigen_{\lambda^{\prime}\sm\mu}^{(k)}+[n+1-k+\content%
(\lambda\sm\lambda^{\prime})]_{q}\ \eigen_{\lambda^{\prime
}\sm\mu}^{(k-1)}=\eigen_{\lambda\sm\mu}^{\left(  k\right)
}.\label{pf:thm:big.eigenvalues:ev}%
\end{equation}

If $\left\vert \lambda\sm\mu\right\vert <k$, then $\left\vert
\lambda^{\prime}\sm\mu\right\vert =\left\vert \lambda\sm
\mu\right\vert -1<k-1$, and thus Corollary~\ref{cor:eigen0forbigk} shows that
all of $\eigen_{\lambda^{\prime}\sm\mu}^{(k)}$,
$\eigen_{\lambda^{\prime}\sm\mu}^{(k-1)}$ and
$\eigen_{\lambda\sm\mu}^{\left(  k\right)  }$ equal $0$; hence,
\eqref{pf:thm:big.eigenvalues:ev} is true in this case. So we assume that
$\left\vert \lambda\sm\mu\right\vert \geq k$. Then,
\begin{equation}
\eigen_{\lambda\sm\mu}^{\left(  k\right)  }=e_{k}(\mathsf{X}%
_{\lambda\sm\mu}^{(k)}/\mathsf{Y}_{\lambda\sm\mu}^{(k)}%
),\qquad\eigen_{\lambda^{\prime}\sm\mu}^{\left(  k\right)  }%
=e_{k}(\mathsf{X}_{\lambda^{\prime}\sm\mu}^{(k)}/\mathsf{Y}%
_{\lambda^{\prime}\sm\mu}^{(k)}),\qquad\eigen_{\lambda^{\prime
}\sm\mu}^{\left(  k-1\right)  }=e_{k-1}(\mathsf{X}_{\lambda^{\prime
}\sm\mu}^{(k-1)}/\mathsf{Y}_{\lambda^{\prime}\sm\mu}%
^{(k-1)}).\qquad
\label{pf:thm:big.eigenvalues:E=E}
\end{equation}

For any list $S$, we let $qS$ denote the list $\left(  qs:s\in S\right)  $.
Observe that%
\begin{align*}
\mathsf{X}_{\lambda\sm\mu}^{(k)}  & =q\mathsf{X}_{\lambda^{\prime
}\sm\mu}^{(k)}\cup\underbrace{\left(  q^{n-k+1}\left[  \content%
(\lambda\sm\lambda^{\prime})\right]  _{q}\right)  }_{\text{one-element
list}}\qquad \quad \text{and} \qquad \quad
\mathsf{Y}_{\lambda\sm\mu}^{(k)}  =q\mathsf{Y}_{\lambda^{\prime
}\sm\mu}^{(k)}\cup\underbrace{\left(  \left[  n-k+1\right]  _{q}\right)
}_{\text{one-element list}}
\end{align*}
(since $\left\vert \lambda\sm\mu\right\vert \geq k$ ensures that the list
$\mathsf{Y}_{\lambda\sm\mu}^{(k)}$ is nonempty).
Thus, Proposition~\ref{prop:supersymrecursion} (applied to the lists
$\mathsf{X}_{\lambda\sm\mu}^{(k)}$ and $\mathsf{Y}_{\lambda\sm\mu
}^{(k)}$ in the roles of $X_{c}$ and $Y_{d}$, and their sublists
$q\mathsf{X}_{\lambda^{\prime}\sm\mu}^{(k)}$ and $q\mathsf{Y}%
_{\lambda^{\prime}\sm\mu}^{(k)}$ in the roles of $X_{c-1}$ and
$Y_{d-1}$, and the extra elements $q^{n-k+1}\left[  \content%
(\lambda\sm\lambda^{\prime})\right]  _{q}$ and $\left[  n-k+1\right]  _{q}$
in the roles of $x_c$ and $y_d$) yields%
\begin{align*}
e_{k}(\mathsf{X}_{\lambda\sm\mu}^{(k)}/\mathsf{Y}_{\lambda\sm\mu
}^{(k)})  & =e_{k}(q\mathsf{X}_{\lambda^{\prime}\sm\mu}^{(k)}%
/q\mathsf{Y}_{\lambda^{\prime}\sm\mu}^{(k)})+(q^{n-k+1}[\content%
(\lambda\sm\lambda^{\prime})]_{q}+[n-k+1]_{q})\,e_{k-1}(q\mathsf{X}%
_{\lambda^{\prime}\sm\mu}^{(k)}/\mathsf{Y}_{\lambda\sm\mu}%
^{(k)})\\
& =\underbrace{e_{k}(q\mathsf{X}_{\lambda^{\prime}\sm\mu}^{(k)}%
/q\mathsf{Y}_{\lambda^{\prime}\sm\mu}^{(k)})}_{\substack{=q^{k}%
e_{k}(\mathsf{X}_{\lambda^{\prime}\sm\mu}^{(k)}/\mathsf{Y}%
_{\lambda^{\prime}\sm\mu}^{(k)}) \\\text{(since }e_{k}(\mathsf{X}%
/\mathsf{Y})\text{ is homogeneous} \\\text{of degree }k\text{)}}%
}+\underbrace{(q^{n-k+1}[\content(\lambda\sm\lambda^{\prime}%
)]_{q}+[n-k+1]_{q})}_{\substack{=[\content(\lambda\sm\lambda^{\prime
})+n-k+1]_{q} \\=[n+1-k+\content(\lambda\sm\lambda^{\prime})]_{q}%
}}\,e_{k-1}(\mathsf{X}_{\lambda^{\prime}\sm\mu}^{(k-1)}/\mathsf{Y}%
_{\lambda^{\prime}\sm\mu}^{(k-1)})\\
& \qquad\qquad\left(  \text{since }q\mathsf{X}_{\lambda^{\prime}\sm\mu
}^{(k)}=\mathsf{X}_{\lambda^{\prime}\sm\mu}^{(k-1)}\text{ and
}\mathsf{Y}_{\lambda\sm\mu}^{(k)}=\mathsf{Y}_{\lambda^{\prime}%
\sm\mu}^{(k-1)}\right) \\
& =q^{k}e_{k}(\mathsf{X}_{\lambda^{\prime}\sm\mu}^{(k)}/\mathsf{Y}%
_{\lambda^{\prime}\sm\mu}^{(k)})+[n+1-k+\content(\lambda
\sm\lambda^{\prime})]_{q}\,e_{k-1}(\mathsf{X}_{\lambda^{\prime}%
\sm\mu}^{(k-1)}/\mathsf{Y}_{\lambda^{\prime}\sm\mu}^{(k-1)}).
\end{align*}
In view of \eqref{pf:thm:big.eigenvalues:E=E}, this rewrites as
\[
\eigen_{\lambda\sm\mu}^{\left(  k\right)  }
=q^{k}\eigen_{\lambda^{\prime}\sm\mu}^{(k)}
+[n+1-k+\content(\lambda\sm\lambda^{\prime})]_{q}
\ \eigen_{\lambda^{\prime}\sm\mu}^{(k-1)}.
\]
This proves \eqref{pf:thm:big.eigenvalues:ev}, and thus completes our induction.

Thus we have proved \eqref{eq:thm:big.eigenvalues:eq}.

If $\kk \subseteq \RR$ and $q\in\RR_{>0}$, then Theorem
\ref{def:basis} shows that $\mathfrak{B}_{\lambda}$ is a basis. Thus, in this
case, \eqref{eq:thm:big.eigenvalues:eq} shows that it is an eigenbasis of all
the $\R_{n,k}$, and therefore $\R_{n,k}$ is diagonalizable with all its
eigenvalues being of the form $\eigen_{\lambda\sm\mu}^{(k)}$ for
horizontal strips $\lambda\sm\mu$ with $|\lambda|=n$. As a consequence,
the equality%
\[
\prod_{\substack{\lambda\sm\mu\text{ is a horizontal strip;}
\\\left\vert \lambda\right\vert =n}}
\left(  \R_{n,k}-\eigen_{\lambda\sm\mu}^{(k)}\right)  =0
\]
holds in $\HH_{n}$. Since this equality is a polynomial
identity\footnote{A-priori, it is just a Laurent polynomial identity, since negative
powers of $q$ can appear. But we will see below (in
Corollary~\ref{cor:eigen0forbigk} and Theorem~\ref{lem:positive}) that
negative powers do not actually appear, so it is a polynomial identity
indeed.} in $q$, it must therefore hold not only when $\kk%
\subseteq \RR$ and $q\in\RR_{>0}$, but also for all $q$ in all
$\kk$. Thus, all eigenvalues of $\R_{n,k}$ are of the form
$\eigen_{\lambda\sm\mu}^{(k)}$ for horizontal strips $\lambda
\sm\mu$ with $|\lambda|=n$, even if $\R_{n,k}$ is not always
diagonalizable. This completes the proof of the theorem.
\end{proof}

One immediate consequence of Theorem~\ref{thm:big.eigenvalues} is that we can make the claim of Proposition~\ref{prop:eval-recursion-1-step} explicit:

\begin{cor}\label{cor:big.eigenvalues} 
Let $\lambda \sm \mu$ be a nonempty horizontal strip with $\abs{\lambda} = n$.
Let $\lambda'$ be a partition such that $\mu \subseteq \lambda' \lessdot \lambda$.
Let $k \geq 1$.
Then the eigenvalue $\eigen_{\lambda \sm \mu}^{(k)}$ of $\R_{n,k}$ satisfies
    \[
    \eigen_{\lambda \sm \mu}^{(k)} = q^{k} \eigen_{\lambda' \sm \mu}^{(k)} + [n+1-k + \content(\lambda \sm \lambda ')]_q \,  \, \eigen_{\lambda' \sm \mu}^{(k-1)}.
    \]

\end{cor}

\begin{proof}
See the proof of \eqref{pf:thm:big.eigenvalues:ev} above.
(This proof did not rely on the special choice of $\lambda'$
that was used in the surrounding proof of
Theorem~\ref{thm:big.eigenvalues}.)
\end{proof}

Figure~\ref{fig:pathsinyoungslattice} shows an example of the use
of Corollary~\ref{cor:big.eigenvalues}.

\subsection{Proof of positivity}\label{ssec.positivity}

We shall now prove that all the eigenvalues
$\eigen_{\lambda \sm\mu}^{(k)}$ of $\R_{n,k}$ belong to $\ZZ_{\geq 0}[q]$,
as claimed in Theorem~\ref{thm:main-thm}(2).
For $\abs{\lambda \sm \mu} < k$, this follows from Corollary~\ref{cor:eigen0forbigk},
so we focus on the remaining case.



\begin{theorem}\label{lem:positive}
Let $n,k\geq 0$ and let $\lambda \sm \mu$ be a horizontal strip with $|\lambda \sm \mu| \geq k$. Then, $\eigen_{\lambda \sm \mu}^{(k)} \in \ZZ_{\geq 0}[q] \setminus \{ 0 \}$, where $q$ is regarded as an indeterminate.
\end{theorem}

\begin{proof}
We proceed by induction on $|\lambda \sm \mu|$. In the base case, when  $|\lambda \sm \mu| = 0$, we must have $k = 0$. Since $\R_{n, 0}$ is the identity operator, each of its eigenvalues is $1$, and so  $\eigen_{\lambda \sm \mu}^{(0)} = 1$.

Now, assume that  $|\lambda \sm \mu| \geq k >0$.  By Corollary~\ref{cor:big.eigenvalues}, we have that $\eigen_{\lambda \sm \mu}^{(k)}$ satisfies the recursion
\[ \eigen_{\lambda \sm \mu}^{(k)} = q^{k} \eigen_{\lambda' \sm \mu}^{(k)} + [n+1-k + \content(\lambda \sm \lambda')]_q \,  \, \eigen_{\lambda' \sm \mu}^{(k-1)},\]
where $\lambda' \lessdot \lambda$.
Our strategy will be to prove that the first summand is in $\mathbb{Z}_{\geq 0}[q]$ and the second summand is in $\mathbb{Z}_{\geq 0}[q] \setminus \{ 0 \}$, from which the statement will follow.

For the first summand, if $|\lambda \sm \mu| = k$, then $|\lambda' \sm \mu| = k-1 < k$, and so $\eigen_{\lambda' \sm \mu}^{(k)}  = 0$ by Corollary~\ref{cor:eigen0forbigk}.
Otherwise, $|\lambda \sm \mu| > k$, in which case $|\lambda' \sm \mu| \geq k$, and so by induction $\eigen_{\lambda' \sm \mu}^{(k)} \in \mathbb{Z}_{\geq 0}[q]$.

In the second summand, since $k \leq |\lambda \sm \mu|$, we have that 
\[ k-1 \leq   |\lambda \sm \mu| - 1 =  |\lambda' \sm \mu|,\] and so by induction $\eigen_{\lambda' \sm \mu}^{(k-1)} \in \mathbb{Z}_{\geq 0}[q] \setminus \{ 0 \}$. Thus, it suffices to prove that $[n + 1 - k + \content(\lambda \sm \lambda')]_q \in \mathbb{Z}_{\geq 0}[q]$, or equivalently that $n + 1 - k + \content(\lambda \sm \lambda') > 0$.
 
Since $\lambda \sm \mu$ is a horizontal strip and $\lambda_1$ is the number of columns in $\lambda$, it follows that $|\lambda \sm \mu| \leq \lambda_1$. Thus,
\begin{equation}\label{eq:boundingkbylambda1}
    k \leq |\lambda \sm \mu| \leq \lambda_1.
\end{equation}

Suppose the box $\lambda \sm \lambda'$ lies in position $(i,j)$, so that $\content(\lambda \sm \lambda') = j-i$.
Then, $j \geq 1$ and $i \leq \ell(\lambda)$, so that $j - i \geq 1 - \ell(\lambda)$.
On the other hand, $n+1 \geq \lambda_1 + \ell(\lambda)$, since $\lambda_1 + \ell(\lambda)$ counts the cells in the first row and in the first column of $\lambda$ (overcounting the cell $(1,1)$).
Using this inequality along with
$\content(\lambda \sm \lambda') = j-i \geq 1-\ell(\lambda) > -\ell(\lambda)$,
we obtain
\[
    \underbrace{n+1}_{\geq \lambda_1 + \ell(\lambda)}
    -\, k + \underbrace{\content(\lambda \sm \lambda')}_{> -\ell(\lambda)}
    > (\lambda_1 + \ell(\lambda)) - k -\ell(\lambda)
    \geq \lambda_1 - k \geq 0
    \qquad \left(\text{by \eqref{eq:boundingkbylambda1}}\right),
\]
as required.
\end{proof}

\subsection{Characteristic polynomial}\label{ssec.charpoly}

We now have all the information necessary to write down the characteristic polynomial of the action of $\R_{n,k}$ on a Specht module $S^\lambda$ and on the full algebra $\HH_n$.

Having written down an eigenvalue  $\eigen_{\lambda \sm \mu}^{(k)}$ for every horizontal strip $\lambda \sm \mu$, we would like to understand its multiplicity. To do so, we introduce the following definition.
\begin{definition}\label{def:desstrip}
    Let $\lambda \sm \mu$ be a horizontal strip. Then:
\begin{itemize}
    \item We say $\lambda \sm \mu$ is a \emph{desarrangement strip} if $d^\mu \neq 0$.
    \item We set $\mathcal{D}_{\lambda}$ to be the collection of desarrangement strips $\lambda \sm \mu$ for a fixed $\lambda$. 
    \item We set $\mathcal{D}_{n}$ to be the collection of desarrangement strips $\lambda \sm \mu$ for all $\lambda$ with $|\lambda| = n$. 
\end{itemize}
\end{definition}

We note the following easy fact (see Appendix B in the arXiv version for a proof), which describes what the ``desarrangement strip'' condition actually rules out:

\begin{lem} \label{lem.d-ne0}
Let $\mu$ be a partition.
Then, $d^\mu \neq 0$ if and only if
$\mu \neq (i)$ for all $i > 0$
and $\mu \neq (1^i)$ for all odd $i$.
\end{lem}

First, we generalize Theorem~\ref{thm:dimkernelR1} from $k=1$ to arbitrary $k$ using Corollary~\ref{cor:eigen0forbigk} and Theorem~\ref{lem:positive}.

\begin{cor}\label{cor:dimkernelRk}
    Suppose $\kk \subseteq \RR$ and $q \in \RR_{>0}$. Then for $\lambda \vdash n$, we have
    \begin{align*}
    \dim \bigg( \ker \left(\R_{n, k} \vert_{S^\lambda} \right) \bigg)
    &= \sum_{\substack{\lambda \sm \mu \, \in \, \mathcal{D}_\lambda;\\ |\lambda \sm \mu| < k}} d^\mu.
    \end{align*}
\end{cor}

\begin{proof}
Theorem~\ref{thm:big.eigenvalues} shows that $\R_{n,k}$ (acting on $S^\lambda$)
is diagonalizable, with eigenbasis $\basis_\lambda$ and eigenvalues
$\eigen_{\lambda\sm \mu}^{(k)}$.
Thus, $\dim ( \ker \left(\R_{n, k} \vert_{S^\lambda} \right) )$
is the number of vectors $y^{\lambda}_{\mu(u)}$ in $\basis_\lambda$ whose
corresponding eigenvalue $\eigen_{\lambda\sm \mu}^{(k)}$ is $0$.
By Corollary~\ref{cor:eigen0forbigk} and Theorem~\ref{lem:positive},
this happens precisely when $|\lambda \sm \mu| < k$ (since
an element of $\ZZ_{\geq 0}[q] \setminus \{ 0 \}$ is nonzero when
$q \in \RR_{>0}$).
\end{proof}

When $k=1$, the only case where $|\lambda \sm \mu| < 1$ is when $\lambda = \mu$, which recovers the dimension count $d^\mu$. 

Now, if $\kk \subseteq \RR$ and $q \in \RR_{>0}$, we use Corollary~\ref{cor:dimkernelRk} and our eigenbasis $\basis_\lambda$ from Corollary~\ref{thm:eigenbasis} to deduce the characteristic polynomial of $\R_{n,k}$ acting on $S^\lambda$.
Specifically, each horizontal strip $\lambda \sm \mu$ contributes the eigenvalue $\eigen_{\lambda \sm \mu}^{(k)}$ with multiplicity $\abs{\kappa_\mu} = \dim \left( \ker \left(\R_{\abs{\mu}, 1} \vert_{S^\mu} \right) \right) = d^\mu$ (by Theorem~\ref{thm:dimkernelR1}), and this eigenvalue will be zero if and only if $|\lambda \sm \mu| < k$ (for generic $q$).
(Note that different horizontal strips may generate the same eigenvalue; see Remark~\ref{rmk:contentvshorizontalstrip}.)
Thus, the characteristic polynomial of $\R_{n,k}$ has the following explicit form (which still holds even if we don't have $\kk \subseteq \RR$ and $q \in \RR_{>0}$, because it is a polynomial identity in $q$):

\begin{cor}
\label{thm:charpoly-lam}
Let $n,k\geq 0$ and $\lambda \vdash n$.
Consider the $\kk$-linear endomorphism of the Specht module $S^\lambda$ given by the action of the element $\R_{n,k}$. The characteristic polynomial of this endomorphism is
\[
\prod_{\lambda \sm \mu \in \mathcal{D}_\lambda} \tup{x - \eigen_{\lambda \sm \mu}^{(k)}}^{d^\mu}
= \left( \prod_{\substack{\lambda \sm \mu \in \mathcal{D}_\lambda:\\ |\lambda \sm \mu|<k}} x^{d^\mu} \right) \left(\prod_{\substack{\lambda \sm \mu \in \mathcal{D}_\lambda:\\ |\lambda \sm \mu|\geq k}} \tup{x - \eigen_{\lambda \sm \mu}^{(k)}}^{d^\mu} \right).
\]
\end{cor}

To understand the action of $\R_{n,k}$ on the full Hecke algebra $\HH_n$, we recall the decomposition \eqref{eq:Hn=sumSlam} of $\HH_n$ as a right $\HH_n$-module.
Hence, the eigenvalues of $\R_{n,k}$ on $\HH_n$ are obtained by combining the eigenvalues of $\R_{n,k}$ on each $S^\lambda$, each time scaling the multiplicities by $f^\lambda$. Thus the following holds:
%

\begin{cor}\label{cor:charpoly}
Let $n,k\geq 0$.
Consider the $\kk$-linear endomorphism of $\HH_n$ given by the action of the element $\R_{n,k}$. The characteristic polynomial of this endomorphism is
\[
\prod_{\lambda \sm \mu \in \mathcal{D}_n} \tup{x - \eigen_{\lambda \sm \mu}^{(k)}}^{d^\mu \cdot f^\lambda}
= \left( \prod_{\substack{\lambda \sm \mu \in \mathcal{D}_n:\\ |\lambda \sm \mu|<k}} x^{d^\mu \cdot f^\lambda} \right) \left(\prod_{\substack{\lambda \sm \mu \in \mathcal{D}_n:\\ |\lambda \sm \mu|\geq k}} \tup{x - \eigen_{\lambda \sm \mu}^{(k)}}^{d^\mu \cdot f^\lambda} \right).
\]
\end{cor}

\subsection{Eigenvalues with product formulas}\label{ssec.hooks}

The formula \eqref{eq:eigenformula} for $\eigen_{\lambda \sm \mu}^{(k)}$ does not appear to further simplify in general.
However, in certain cases where $\lambda$ is a hook of the form $(n-\ell,1^\ell)$ and $\mu$ is a hook $(j,1^\ell)$,  the eigenvalues can be computed explicitly in terms of $q$-binomial coefficients.

We recall that $q$-binomial coefficients are defined by
\[
\qbinom{n}{k}_q := \dfrac{[n]_q [n-1]_q \cdots [n-k+1]_q}{[k]!_q}
= \dfrac{[n]!_q}{[k]!_q [n-k]!_q}
\qquad \text{ for all } n, k \in \mathbb{N} ,
\]
where we understand $[m]!_q$ to be $\infty$ for all negative $m$ (and any fraction with $\infty$ in its denominator is $0$).

\begin{theorem}
    \label{prop:evals-n-l-1l}
    Let $k\geq 0$ and $\ell \in \left\{0,1,\ldots,n\right\}$.
    Then,
    \begin{equation}\label{eq:genhookexpression}
	\eigen_{(n-\ell,1^\ell) \sm (j,1^\ell)}^{(k)} =
    \ive{ k}!_q \qbinom{n-j-\ell}{k}_q \qbinom{n+j}{k}_q
    \qquad \text{ for all } 1 \leq j  \leq n-\ell \text{ and for } j=0=\ell .
    \end{equation}
\end{theorem}

\begin{proof}
When $k = 0$, both sides of the equation are $1$.
So, we assume $k\geq 1$.

We fix $j$ and induct on $n \geq j+\ell$. In the base case $n= j+\ell$, both sides are $0$, since $(n-\ell,1^\ell)=(j,1^\ell)$.

In the induction step, we apply the recursion in Corollary~\ref{cor:big.eigenvalues} to $\lambda = (n-\ell,1^\ell)$ and $\mu = (j,1^\ell)$ and $\lambda' = (n-\ell-1,1^\ell)$ to obtain
\[
\eigen_{\left(  n-\ell,1^{\ell} \right) \sm \left(  j,1^{\ell} \right)  }^{(k)}
= q^k \,\eigen_{\left(  n-\ell-1,1^\ell \right) \sm \left(  j,1^\ell \right) }^{(k)}
+ [n+1-k+\left(  n-\ell-1\right)  ]_q\,\eigen_{\left(  n-\ell-1,1^\ell \right) \sm \left(  j,1^\ell \right) }^{(k-1)}.
\]
It remains to show that the right-hand-side of \eqref{eq:genhookexpression} satisfies the same recursion. Routine calculations show that this boils down to checking that
    \[
    \left[  n-j-\ell\right]  _q\left[  n+j\right]  _q
    =q^{k}\left[
    n-j-\ell-k\right]  _q\left[  n+j-k\right]  _q+\left[  2n-\ell-k\right] _q\left[  k\right]  _q.
    \qedhere
    \]
\end{proof}

Theorem~\ref{prop:evals-n-l-1l} describes only some of the eigenvalues obtained from $\R_{n,k}$ acting on the Specht module $S^{(n-\ell,1^{\ell})}$ (in particular, there are other horizontal strips $(n-\ell,1^{\ell}) \sm \mu$ for $\mu \neq (j,1^{\ell})$).  
However, we can specialize Theorem~\ref{prop:evals-n-l-1l} to obtain a complete set of eigenvalues on the Specht modules $S^{(n)}$ and $S^{(n-1,1)}$.

We begin with describing the action of $\R_{n,k}$ on the Specht module $S^{(n)}$.
We recall that this Specht module is $1$-dimensional, and is the $q$-analogue of the trivial representation of $\symm_n$.

\begin{cor}
    \label{prop:evals-n}
    Let $k\geq 0$.
    The unique eigenvalue of $\R_{n,k}$ acting on the Specht module $S^{(n)}$ is
    \[
    \eigen_{(n) \sm \varnothing}^{(k)} = [k]!_q \qbinom{n}{k}_q^2.
    \]
\end{cor}

The next-simplest shape is the partition $\tup{n-1,1}$ for $n \geq 2$.
Its Specht module $S^{(n-1,1)}$ is the $q$-analogue of the reflection representation of $\symm_n$; in particular, it has dimension $n-1$.
Generalizing Lafreni\`ere's \cite[Th\'eor\`eme 88]{Lafreniere} to arbitrary values of $q$, we describe the eigenvalues of $\R_{n,k}$ on it as follows:\footnote{Our $j$ corresponds to Lafreni\`ere's $n+1-i$.}

\begin{cor}
    \label{prop:evals-n-1-1}
    Let $k\geq 0$.
    The eigenvalues of $\R_{n,k}$ acting on the Specht module $S^{(n-1,1)}$, listed with multiplicity, are
    the numbers
    \[
	\eigen_{(n-1,1) \sm (j,1)}^{(k)} =
    \ive{ k}!_q \qbinom{n-j-1}{k}_q \qbinom{n+j}{k}_q,
    \qquad \text{ one for each } j \in \ive{n-1}.
    \]
\end{cor}
\begin{proof}
Corollary~\ref{thm:charpoly-lam} shows that
the eigenvalues of $\R_{n,k}$ acting on $S^{(n-1,1)}$ are the $\eigen_{(n-1,1) \sm \mu}^{(k)}$ for all horizontal strips $(n-1,1) \sm \mu \in \mathcal{D}_{(n-1,1)}$, with respective multiplicities $d^\mu$. It is easily seen that these horizontal strips are precisely the $(n-1,1) \sm (j,1)$ for $j \in [n-1]$, and that the multiplicities $d^{(j,1)}$ are $1$.
(For a shortcut, argue that $d^{(j,1)} \geq 1$ because the standard tableau with first column $(1,2)$ is a desarrangement tableau; this accounts for $n-1$ eigenvectors, and $\dim(S^{(n-1,1)}) = n-1$ shows that we cannot have any more.)
Finally, the formula for $\eigen_{(n-1,1) \sm (j,1)}^{(k)}$ follows from Theorem~\ref{prop:evals-n-l-1l}.
%
\end{proof}

Theorem~\ref{prop:evals-n-l-1l} might suggest that the eigenvalues of $\R_{n,k}$ should be describable as products and fractions of $q$-integers for other shapes $\lambda$ as well. But this is too much to hope for, even for all hook-shaped partitions.
\begin{example}
\label{exa:non-q-ints}
For $n = 5$ and $\lambda \sm \mu = \tup{2,2,1} \sm \tup{2,1}$, we have that $\R_{5,1}$ has an eigenvalue
\begin{align*}
\eigen_{\tup{2,2,1} \sm \tup{2,1}}^{(1)}= q [4]_q + [3]_q
= q^4 + q^3 + 2q^2 + 2q + 1
\end{align*}
with multiplicity $d^{\tup{2,1}} = 1$ on $S^{(2,2,1)}$.
As a polynomial in $q$, this has a non-abelian Galois group over $\mathbb{Q}$ and two roots outside of the unit circle.
Thus, it is not a product of quotients of $q$-integers, showing that the eigenvalues of $\R_{n,k}$ do not always have expressions as products of $q$-integers. The same behavior appears for $\lambda \sm \mu = \tup{3,2} \sm \tup{2,1}$ and for $\lambda \sm \mu = \tup{4, 1, 1} \sm \tup{1, 1}$.
\end{example}

However, there is a simple formula for the $\eigen_{\lambda \sm \mu}^{(k)}$ in the case that $k = |\lambda \sm \mu|$.

\begin{prop}\label{prop:eval.largestk}
    Suppose that $\lambda \sm \mu$ is a horizontal strip with $|\lambda|=n$ and $|\mu| = j$. Then 
    \[ \eigen_{\lambda \sm \mu}^{(n-j)} = \prod_{(r,s) \in \lambda \sm \mu} [j+1 + \content(r,s)]_q.\]
\end{prop}

\begin{proof}
In the case that $j=n$, we interpret the empty product as 1, giving $\eigen_{\lambda \sm \lambda}^{(0)} = 1$. Now assume $j < n$. Since $n-j>0$, we can pick a $\lambda'$ such that $\mu \subseteq \lambda' \lessdot \lambda$. Then, $\eigen_{\lambda'\sm \mu}^{(n-j)} = 0$ by Corollary~\ref{cor:eigen0forbigk}, since $\abs{\lambda' \sm \mu} = n-j-1 < n-j$. Hence Corollary~\ref{cor:big.eigenvalues} for $k=n-j$ simplifies to
\[
\eigen_{\lambda \sm \mu}^{(n-j)}
= [n+1-(n-j) + \content(\lambda \sm \lambda')]_q \, \eigen_{\lambda' \sm \mu}^{(n-j-1)}
= [j + 1 + \content(\lambda \sm \lambda')]_q \, \eigen_{\lambda' \sm \mu}^{(n-j-1)}.
\]
Proceed by recursion (since $\lambda' \sm \mu$ is again a horizontal strip, now with $\abs{\lambda'} = n-1$).
\qedhere

\end{proof}

We note that the following is true for \emph{all} $\eigen_{\lambda \sm \mu}^{(k)}$.

\begin{prop} \label{prop:monic}
Let $\lambda \sm \mu$ be a horizontal strip with $|\lambda| = n$, $|\mu|=j$ and $0 \leq k \leq n-j$.
Then, $\eigen_{\lambda \sm \mu}^{\tup{k}}$ is a monic polynomial in $q$ of degree 
\[ k\tup{n-k} + w_k\tup{\lambda \sm \mu},\] where $w_k\tup{\lambda \sm \mu}$ is the sum of the contents of the rightmost $k$ boxes of $\lambda \sm \mu$.
\end{prop}

\begin{proof}
The case $k=0$ is trivial, since $\eigen_{\lambda\sm\mu}^{\left(
0\right)  }=1$ and $w_{0}\left(  \lambda\sm\mu\right)  =0$.

We induct on $\left\vert \lambda\sm\mu\right\vert =n-j$. The base case
$n-j=0$ is trivial, since the condition $0\leq k\leq n-j$ forces $k=0$.

In the induction step, assume that $\left\vert \lambda\sm\mu\right\vert
=n-j>0$. Let $c_{1},c_{2},\ldots,c_{n-j}$ be the contents of the boxes of
$\lambda\sm\mu$, listed from right to left. Thus, $c_{1}>c_{2}%
>\cdots>c_{n-j}$ (since $\lambda\sm\mu$ is a horizontal strip) and
$w_{k}\left(  \lambda\sm\mu\right)  =c_{1}+c_{2}+\cdots+c_{k}$. We
assume $k\geq1$, since the case $k=0$ has already been handled.

Let $\lambda^{\prime}\lessdot\lambda$ be the partition obtained from $\lambda$
by removing the rightmost box (i.e., the one with content $c_{1}$). Then,
$\content\left(  \lambda\sm\lambda^{\prime}\right)  =c_{1}$ and
$\left\vert \lambda^{\prime}\sm\mu\right\vert =\left\vert \lambda
\sm\mu\right\vert -1$, and the boxes of $\lambda^{\prime}\sm\mu$
(listed from right to left) have contents $c_{2},c_{3},\ldots,c_{n-j}$; thus,
\begin{align}
w_{k-1}\left(  \lambda^{\prime}\sm\mu\right)   &  =c_{2}+c_{3}%
+\cdots+c_{k}=w_{k}\left(  \lambda\sm\mu\right)  -c_{1}
\label{pf:prop:monic:w1}
\end{align}
and (when $k<n-j$) also
\begin{align*}
w_{k}\left(  \lambda^{\prime}\sm\mu\right)   &  =c_{2}+c_{3}%
+\cdots+c_{k+1}=w_{k}\left(  \lambda\sm\mu\right)  -c_{1}+c_{k+1}%
<w_{k}\left(  \lambda\sm\mu\right)
\qquad \left(\text{since $c_{1}>c_{k+1}$}\right).
\end{align*}
By the induction hypothesis, the polynomials
$\eigen_{\lambda^{\prime}\sm\mu}^{(k)}$ and
$\eigen_{\lambda^{\prime}\sm\mu}^{(k-1)}$ are monic of respective
degrees
\begin{align*}
\deg(\eigen_{\lambda^{\prime}\sm\mu}^{(k)}) &  =k(n-1-k)+w_{k}%
\left(  \lambda^{\prime}\sm\mu\right)  \ \ \ \ \ \ \ \ \ \ \text{and}\\
\deg(\eigen_{\lambda^{\prime}\sm\mu}^{(k-1)}) &  =\left(
k-1\right)  \left(  n-k\right)  +w_{k-1}\left(  \lambda^{\prime}\sm
\mu\right)  ,
\end{align*}
with one exception: If $k = n-j$, then
$\eigen_{\lambda^{\prime}\sm\mu}^{(k)}$ is $0$ instead
(because in this case, replacing $\lambda$ by $\lambda'$
breaks the inequality $k \leq n-j$, rendering the induction
hypothesis inapplicable; instead, Corollary~\ref{cor:eigen0forbigk}
yields $\eigen_{\lambda^{\prime}\sm\mu}^{(k)}=0$).

But Corollary~\ref{cor:big.eigenvalues} shows that%
\begin{equation}
\eigen_{\lambda\sm\mu}^{(k)}=q^{k}\eigen_{\lambda^{\prime
}\sm\mu}^{(k)}+[n+1-k+\content(\lambda\sm\lambda^{\prime}%
)]_{q}\,\,\eigen_{\lambda^{\prime}\sm\mu}^{(k-1)}%
.\label{pf:prop:monic:5}%
\end{equation}

However, any $q$-integer $\left[  m\right]  _{q}$ with $m>0$ is a monic
polynomial of degree $m-1$ in $q$. In the proof of Theorem~\ref{lem:positive},
we showed that $n+1-k+\content(\lambda\sm\lambda^{\prime})>0$ (since
$k\leq n-j=\left\vert \lambda\sm\mu\right\vert $). Thus,
$[n+1-k+\content(\lambda\sm\lambda^{\prime})]_{q}$ is a monic
polynomial of degree%
\[
n+1-k+\content(\lambda\sm\lambda^{\prime}%
)-1=n-k+\underbrace{\content(\lambda\sm\lambda^{\prime})}_{=c_{1}%
}=n-k+c_{1}.
\]
Since $\eigen_{\lambda^{\prime}\sm\mu}^{(k-1)}$ is a monic
polynomial of degree $\deg(\eigen_{\lambda^{\prime}\sm\mu}%
^{(k-1)})=\left(  k-1\right)  \left(  n-k\right)  +w_{k-1}\left(
\lambda^{\prime}\sm\mu\right)  $, we thus conclude that the product
$[n+1-k+\content(\lambda\sm\lambda^{\prime})]_{q}\,\,
\eigen_{\lambda^{\prime}\sm\mu}^{(k-1)}$ is monic of degree
\begin{align}
\left(  n-k+c_{1}\right)  +\left(  \left(  k-1\right)  \left(  n-k\right)
+w_{k-1}\left(  \lambda^{\prime}\sm\mu\right)  \right)
& =\underbrace{n-k+\left(  k-1\right)  \left(  n-k\right)  }_{=k\left(
n-k\right)  }+\underbrace{c_{1}+w_{k-1}\left(  \lambda^{\prime}\sm
\mu\right)  }_{\substack{=w_{k}\left(  \lambda\sm\mu\right)
\\\text{(by \eqref{pf:prop:monic:w1})}}}\nonumber\\
& =k\left(  n-k\right)  +w_{k}\left(  \lambda\sm\mu\right)
.\label{pf:prop:monic:deg2}
\end{align}

On the other hand, the polynomial $q^{k}\eigen_{\lambda^{\prime}%
\sm\mu}^{(k)}$ has degree\footnote{As we said, this argument
does not quite work when $k = n-j$; but the final claim
\eqref{pf:prop:monic:deg1} is still correct in this case, for
obvious reasons.}
\begin{align}
\deg\left(  q^{k}\eigen_{\lambda^{\prime}\sm\mu}^{(k)}\right)    &
=k+\underbrace{\deg\left(  \eigen_{\lambda^{\prime}\sm\mu}%
^{(k)}\right)  }_{=k(n-1-k)+w_{k}\left(  \lambda^{\prime}\sm\mu\right)
}
=\underbrace{k+k(n-1-k)}_{=k\left(  n-k\right)  }+\underbrace{w_{k}\left(
\lambda^{\prime}\sm\mu\right)  }_{<w_{k}\left(  \lambda\sm
\mu\right)  }\nonumber\\
& <k\left(  n-k\right)  +w_{k}\left(  \lambda\sm\mu\right)
.\label{pf:prop:monic:deg1}
\end{align}

Now, the equality \eqref{pf:prop:monic:5} becomes
\[
\eigen_{\lambda\sm\mu}^{(k)}=\underbrace{q^{k}
\eigen_{\lambda^{\prime}\sm\mu}^{(k)}}_{\substack{\text{a polynomial
of}\\\text{degree }<k\left(  n-k\right)  +w_{k}\left(  \lambda\sm
\mu\right)  \\\text{(by \eqref{pf:prop:monic:deg1})}}%
}+\underbrace{[n+1-k+\content(\lambda\sm\lambda^{\prime}%
)]_{q}\,\,\eigen_{\lambda^{\prime}\sm\mu}^{(k-1)}}%
_{\substack{\text{a monic polynomial of}\\\text{degree }k\left(  n-k\right)
+w_{k}\left(  \lambda\sm\mu\right)  \\\text{(by
\eqref{pf:prop:monic:deg2})}}},
\]
showing that $\eigen_{\lambda\sm\mu}^{(k)}$ is a monic polynomial
of degree $k\left(  n-k\right)  +w_{k}\left(  \lambda\sm\mu\right)  $
(since the first addend has smaller degree than the second and thus does
not affect the leading term).
This completes the induction.
\end{proof}

\subsection{Inequalities}

In the context of Markov chains, it is of significant interest to understand the \emph{spectral gap} of the process, which depends on the first and second-largest eigenvalues. In the case that $k=1$, it is proved in \cite[Corollaries 7.9 and 7.12]{r2r1} that $\eigen_{(n)\sm \emptyset}^{(1)}$ and $\eigen_{(n-1,1)\sm (1,1)}^{(1)}$ are the first and second-largest eigenvalues of $\R_{n,1}$ when $q \in \RR_{>0}$. We prove that the analog is true for arbitrary $k$, proving a conjecture of Lafreni\`ere in the case that $q=1$.

\begin{theorem}
\label{prop:max2}
Assume that $\kk \subseteq \RR$ and $q\in
\RR_{>0}$. Then, for all $n \geq 2$ and $1 \leq k \leq n$, the largest and second-largest eigenvalues of $\R_{n,k}$ acting on $\HH_{n}$
are $\eigen_{(n) \sm \varnothing}^{(k)}$ and $\eigen_{(n-1,1) \sm (1,1)}^{\left(  k\right)}$, respectively. More concretely:

\begin{enumerate}
\item[(a)] For any horizontal strip $\lambda \sm \mu$ with $\lambda\vdash
n$, we have $\eigen_{\lambda \sm \mu}^{(k)}\leq\eigen_{\left(
n\right)   \sm \varnothing}^{(k)}$.

\item[(b)] For any horizontal strip $\lambda \sm \mu\neq\left(  n\right)
 \sm \varnothing$ with $\lambda\vdash n$ and $d^{\mu}\neq0$, we have
$\eigen_{\lambda \sm \mu}^{(k)}\leq\eigen_{(n-1,1) \sm 
(1,1)}^{(k)}$. 
\end{enumerate}
\end{theorem}

\begin{proof}
We shall only prove part (b), since the proof of (a) is analogous but easier. 

We fix $\mu$ and proceed by induction on
$n = \abs{\lambda} \geq \abs{\mu}$. Note that if $|\lambda \sm \mu| < k$, then (b) holds automatically by Corollary~\ref{cor:eigen0forbigk} and Theorem~\ref{lem:positive}.
So, we assume that $k \leq |\lambda \sm \mu|$.
The base case is $\lambda=\mu$, in which case $\eigen_{\lambda \sm \mu}^{(k)} = 0$
because of $k \geq 1$.

For the induction step, we assume that $k>0$ and
$\lambda\neq\mu$. (Recall that for $k=0$, all eigenvalues are 1.) Fix $\mu \subseteq \lambda^{\prime}\vdash n-1$ where $\lambda' \lessdot \lambda$. From $\lambda \sm \mu\neq\left(  n\right)   \sm \varnothing$, we
know $\lambda^{\prime} \sm \mu\neq\left(  n-1\right)
 \sm \varnothing$. 

Hence, we
can apply the induction hypothesis to $\lambda^{\prime} \sm \mu$ and conclude that 
\[ \eigen_{\lambda^{\prime
} \sm \mu}^{(k)}\leq\eigen_{(n-2,1) \sm (1,1)}^{(k)} \textrm{ and } 
 \eigen_{\lambda^{\prime} \sm \mu
}^{\left(  k-1\right)}  \leq\eigen_{(n-2,1) \sm (1,1)}^{(k-1)}. \] 

Next, we shall show that $\content(\lambda \sm \lambda^{\prime})\leq n-2$. In
fact, it suffices to prove that the content of each cell of $\lambda$ is $\leq
n-2$. To prove
this, assume the contrary. Then, $\lambda$ has a cell with content at least $ n-1$.
Since $\lambda\vdash n$, this cell can only be $\left(  1,n\right)  $, and
thus we conclude that $\lambda=\left(  n\right)  $. Thus, $\mu=\left(
j\right)  $ for some $j\leq n$ (since $\lambda \sm \mu$ is a skew
diagram). Combining this with $d^{\mu}\neq0$, we obtain $\mu=\varnothing$
(since the only desarrangement tableau with only one row is the empty tableau
of shape $\varnothing$). Thus, $\lambda \sm \mu=\left(  n\right)
 \sm \varnothing$, contradicting $\lambda \sm \mu\neq\left(  n\right)
 \sm \varnothing$. Hence, our assumption was wrong, so that $\content(\lambda
 \sm \lambda^{\prime}) \leq n-2$ is proved.

We note that $\content \big(\left(
n-1,1\right)   \sm \left(  n-2,1\right) \big)= n-2$. By the proof of Theorem~\ref{lem:positive} and since $q\in \RR_{>0}$, we conclude that
\[ 0 \leq [n+1-k+\content(\lambda \sm \lambda')]_{q}\leq
[n+1-k+\content(\left(  n-1,1\right)   \sm \left(  n-2,1\right)  )]_{q}
 = [2n-1-k]_q. \]

Thus, Corollary~\ref{cor:big.eigenvalues} yields
\begin{align*}
\eigen_{\lambda \sm \mu}^{(k)}  
&= q^k \eigen_{\lambda' \sm \mu}^{(k)} + [n+1-k + \content(\lambda \sm \lambda')]_q \, \eigen_{\lambda' \sm \mu}^{(k-1)}\\
& \leq q^{k}\,\eigen_{(n-2,1) \sm (1,1)}^{(k)}+[2n-1-k]_{q}\,
\eigen_{(n-2,1) \sm (1,1)}^{(k-1)}\\
& =\eigen_{(n-1,1) \sm (1,1)}^{(k)}
\qquad \left(\text{again by Corollary~\ref{cor:big.eigenvalues}}\right). 
\end{align*}
 
This completes the
induction step, and thus the proof of Theorem~\ref{prop:max2}.
\end{proof}

\begin{question}
It is not hard to see that if the horizontal strip $\lambda \sm \mu$ has all its cells in one row, then each $\eigen_{\lambda \sm \mu}^{(k)}$ is palindromic and unimodal as a polynomial in $q$.
The palindromicity does not generalize to arbitrary $\lambda \sm \mu$ (see Example~\ref{exa:non-q-ints}). Does the unimodality? We have verified it using SageMath for all $n \leq 28$.
Note that these polynomials are not log-concave, as we see for $\eigen_{(2, 1, 1) \sm (1, 1)}^{(1)}$.
\end{question}

\section{Random-to-random algebra}\label{section.ralgebra}
Thus far, we have used the fundamental recursion in Theorem~\ref{thm:crucialrecursion} to prove that the $\R_{n,k}$ pairwise commute, and to produce a complete set of eigenvalues (along with an eigenbasis for those eigenvalues for $q \in \RR_{>0}$). We now turn our attention to two natural and closely related questions:
\begin{enumerate}
    \item The spectral formula $\eigen_{\lambda \sm \mu}^{(k)}$ is completely determined by the set of contents of $\lambda \sm \mu$. Does a converse statement hold? If  $\eigen_{\lambda \sm \mu}^{(k)} = \eigen_{\tau \sm \gamma}^{(k)} $, must the contents of $\lambda \sm \mu$ and $\tau \sm \gamma$ agree?
    \item For a fixed $n$, the collection $\{ \R_{n,k} \}_{0 \leq k \leq n}$ generates a commutative algebra. What can be said about this algebra?
\end{enumerate}

We address question (1) in Section~\ref{subsec:separating-eigenvalues} and question (2) in Section~\ref{subsec:r2r-algebra}.

\subsection{Separating eigenvalues by content set}\label{subsec:separating-eigenvalues}
Given a skew partition $\lambda \sm \mu$, recall that we define the \emph{content (multi)-set} to be 
\[ \content(\lambda \sm \mu):= \{ j-i: (i,j) \in \lambda \sm \mu \}.\]
In fact, since we are interested in horizontal strips, for us $\content(\lambda \sm \mu)$ will be a genuine set.
\begin{example}
    Setting $\lambda = (6, 4, 2, 2)$ and $\mu = (4, 4, 2)$, the corresponding content set is 
    \[ \content((6, 4, 2, 2) \sm (4, 4, 2)) = \{-3, -2, 4, 5\}. \]
\end{example}

One might hope that different content sets separate the eigenvalues of $\R_{n,k}$. However, we can see in the case that $q=1$ and $k=1$ that this is not necessarily the case.

\begin{example}\label{ex.eigenvalues.coincide}
Suppose $q = 1$ and $n = 4$. We have 
\[ \eigen_{(3,1) \sm (2,1)}^{(1)} = 6 = \eigen_{(2,1,1) \sm (1,1)}^{(1)}, \] but the horizontal strips $(3,1) \sm (2,1)$ and $(2,1,1) \sm (1,1)$ do not have the same content set. 
\end{example}

We show in the next two subsections that this can be remedied in two ways. First, we show in Subsection~\ref{ssec.nonintegralq} that if we restrict ourselves to the case where $q$ is not \emph{integral over $\ZZ$} (excluding, for example, $q=1$; see Definition~\ref{def:integral}), then content sets separate the eigenvalues of $\R_{n,1}$ (and thus also the $\R_{n,k}$). In other words, in this case $\content(\lambda\sm \mu) = \content(\alpha \sm \beta)$ if and only if $\eigen_{\lambda \sm \mu}^{(1)} = \eigen_{\alpha \sm \beta}^{(1)}$.

Second, we show in Subsection~\ref{ssec.content.heckesemisimple} that if $q$ is not zero or a root of unity (actually, an even weaker assumption is needed), then the eigenvalues $\eigen_{\lambda \sm \mu}^{(k)}$ are separated by content sets if one considers the full collection of $\R_{n,k}$ for $0 \leq k \leq n$. In other words, in this setting, $\content(\lambda\sm \mu) = \content(\alpha \sm \beta)$ if and only if $\eigen_{\lambda \sm\mu}^{(k)} = \eigen_{\alpha \sm \beta}^{(k)}$ for all $k$.

\begin{remark}\label{rmk:contentvshorizontalstrip}
    Note that even when $\content(\lambda \sm \mu) = \content(\alpha \sm \beta)$, it is not necessarily the case that $\lambda \sm \mu = \alpha \sm \beta$.  For example, if $\lambda \sm \mu = (3, 3, 1) \sm (3, 2, 1)$ and $\alpha \sm \beta = (2, 1^5) \sm (1^6)$ then the content sets agree even though the horizontal strips do not.
\end{remark}

\subsubsection{When content sets separate the $\R_{n,1}$-eigenvalues}\label{ssec.nonintegralq}

From now on, \textbf{we assume that $\kk$ is a field}.

\begin{definition}\label{def:integral}
 An element $a\in \kk$ is \emph{integral over $\ZZ$} if there exists
a monic polynomial $f \in \ZZ[x]$ satisfying $f(a) = 0$.
\end{definition}
For example, $q=1$ is integral over $\ZZ$ with $f(x) = x-1$.

In this subsection, we shall consider the case when $q$ is \emph{not} integral over $\ZZ$.
This case includes the case when $q$ is transcendental over the prime field of $\kk$.
It also includes the case when $\QQ \subseteq \kk$ and $q \in \QQ \setminus \ZZ$. 

\begin{theorem}\label{thm.A1}
Assume that $q$ is not integral over $\ZZ$.
Let $\lambda \sm \mu$ and $\alpha \sm \beta$ be two horizontal strips with $\lambda,\alpha \vdash n$. Then,
\[ \eigen_{\lambda \sm \mu}^{(1)} = \eigen_{\alpha \sm \beta}^{(1)}\] if and only if 
\[ \content(\lambda \sm \mu) = \content(\alpha \sm \beta).\]
In this case, $\eigen_{\lambda \sm \mu}^{(k)} = \eigen_{\alpha \sm \beta}^{(k)}$ for all $k \geq 0$.
\end{theorem}

\begin{proof}
If $\content(\lambda \sm \mu) = \content(\alpha \sm \beta)$,
then $\abs{\mu} = \abs{\beta}$ (since $\abs{\content(\lambda \sm \mu)} = \abs{\lambda \sm \mu} = \abs{\lambda} - \abs{\mu} = n - \abs{\mu}$ and similarly $\abs{\content(\alpha \sm \beta)} = n - \abs{\beta}$),
and thus the equality \eqref{eq:eigenformula} shows that $\eigen_{\lambda \sm \mu}^{(k)} = \eigen_{\alpha \sm \beta}^{(k)}$ for all $k \geq 0$
(since $\sfX^{(k)}_{\lambda \sm \mu} = \sfX^{(k)}_{\alpha \sm \beta}$ and $\sfY^{(k)}_{\lambda \sm \mu} = \sfY^{(k)}_{\alpha \sm \beta}$);
in particular, we then obtain $\eigen_{\lambda \sm \mu}^{(1)} = \eigen_{\alpha \sm \beta}^{(1)}$.

To prove the other direction, we induct on $n$. Let $\content(p)$ denote the content of a box $p$.

Assume that $\lambda \sm \mu$ and $\alpha \sm \beta$ are two horizontal strips with $\lambda, \alpha \vdash n$ satisfying $\eigen_{\lambda \sm \mu}^{(1)} = \eigen_{\alpha \sm \beta}^{(1)}$.
We must prove that $\lambda \sm \mu$ and $\alpha \sm \beta$ have the same content set.

Note that $\eigen_{\lambda \sm \mu}^{(1)}$ is $0$ if $\mu = \lambda$,
and otherwise is a monic polynomial in $q$ of degree $n-1 + \content(\text{rightmost box of }\lambda \sm \mu)$ with coefficients in $\ZZ$ by Proposition~\ref{prop:monic}.
An analogous statement holds for $\alpha \sm \beta$.
Since $\eigen_{\lambda \sm \mu}^{(1)} = \eigen_{\alpha \sm \beta}^{(1)}$,
we thus conclude the following:
\begin{enumerate}
\item If $\eigen_{\lambda \sm \mu}^{(1)} = \eigen_{\alpha \sm \beta}^{(1)} = 0$, then we have $\mu = \lambda$ and $\beta = \alpha$ (because $q$ is not integral over $\ZZ$, so that a monic polynomial in $q$ with coefficients in $\ZZ$ cannot equal $0$).
\item Otherwise, we have $\mu \neq \lambda$ and $\beta \neq \alpha$, and furthermore 
\[ n-1 + \content(\text{rightmost box of }\lambda \sm \mu) = n-1 + \content(\text{rightmost box of }\alpha \sm \beta) \] since the non-integrality of $q$ guarantees that two monic polynomials in $q$ with coefficients in $\ZZ$ cannot be equal unless they have the same degree\footnote{Otherwise, their difference would be a monic polynomial in $q$ vanishing at $q$; but this contradicts the non-integrality of $q$.}, so that
\begin{align}
\content(\text{rightmost box of }\lambda \sm \mu)
= \content(\text{rightmost box of }\alpha \sm \beta).
\label{pf.thm.A1.c=c}
\end{align}
\end{enumerate}
In the first of these two cases, our proof is complete.
In the second case, we conclude that the rightmost box of $\lambda \sm \mu$ and the rightmost box of $\alpha \sm \beta$ have the same content.
Peeling off the former box from $\lambda \sm \mu$, we obtain a smaller horizontal strip $\lambda' \sm \mu$ with $\lambda' \vdash n-1$.
Peeling off the latter box from $\alpha \sm \beta$, we similarly obtain a smaller horizontal strip $\alpha' \sm \beta$ with $\alpha' \vdash n-1$.
Now, Corollary~\ref{cor:big.eigenvalues} (for $k=1$) yields
\[
\eigen_{\lambda \sm \mu}^{(1)} = q \eigen_{\lambda' \sm \mu}^{(1)} + [n + \content(\text{rightmost box of }\lambda \sm \mu)]_q
\]
and likewise
\[
\eigen_{\alpha \sm \beta}^{(1)} = q \eigen_{\alpha' \sm \beta}^{(1)} + [n + \content(\text{rightmost box of }\alpha \sm \beta)]_q.
\]
Since $\eigen_{\lambda \sm \mu}^{(1)} = \eigen_{\alpha \sm \beta}^{(1)}$ and $\content(\text{rightmost box of }\lambda \sm \mu) = \content(\text{rightmost box of }\alpha \sm \beta)$, we thus obtain
\[
q \eigen_{\lambda' \sm \mu}^{(1)} = q \eigen_{\alpha' \sm \beta}^{(1)}.
\]
Since $q\neq 0$ (after all, $q$ is not integral over $\ZZ$), we can divide this equality by $q$, thus obtaining $\eigen_{\lambda' \sm \mu}^{(1)} = \eigen_{\alpha' \sm \beta}^{(1)}$.
By the induction hypothesis, we conclude that $\lambda' \sm \mu$ and $\alpha' \sm \beta$ have the same content set.
Combining this with \eqref{pf.thm.A1.c=c}, we conclude that $\lambda \sm \mu$ and $\alpha \sm \beta$ have the same content set.
This completes the induction.
%
\end{proof}

We can use Theorem~\ref{thm.A1} to separate the eigenvalues of $\R_{n,k}$ for any $k$ by content set.

\begin{cor}\label{cor.A1k}
Assume that $q$ is not integral over $\ZZ$.
If two horizontal strips $\lambda \sm \mu$ and $\alpha \sm \beta$ with $\lambda,\alpha \vdash n$ satisfy $\eigen_{\lambda \sm \mu}^{(k)} \neq \eigen_{\alpha \sm \beta}^{(k)}$ for some $k \geq 0$,
then $\eigen_{\lambda \sm \mu}^{(1)} \neq \eigen_{\alpha \sm \beta}^{(1)}$ and thus $\lambda \sm \mu \neq \alpha \sm \beta$.
\end{cor}

\begin{proof}
Assume the contrary.
Then, $\eigen_{\lambda \sm \mu}^{(1)} = \eigen_{\alpha \sm \beta}^{(1)}$.
Hence, Theorem~\ref{thm.A1} shows that the strips $\lambda \sm \mu$ and $\alpha \sm \beta$ have the same content set.
But then $\eigen_{\lambda \sm \mu}^{(k)} = \eigen_{\alpha \sm \beta}^{(k)}$ for all $k \geq 0$.
This contradicts our assumption.
\end{proof}

\subsubsection{When content sets separate the $\R_{n,k}$-eigenvalues}\label{ssec.content.heckesemisimple}

We now relax our conditions on $q$ somewhat.
We still assume that $\kk$ is a field.

\begin{definition}
    Let $m \geq 0$. We say that $q$ is \emph{$m$-safe} if $q$ and the $q$-integers $[1]_q,[2]_q,\ldots,[m]_q$ are nonzero.
\end{definition}

We note a few simple properties of safety:
\begin{itemize}
\item When $\kk$ has characteristic $0$, the element $q$ is $m$-safe if and only if $q$ is nonzero and not a primitive $j$-th root of unity for $1 < j \leq m$.
\item The element $1$ is $m$-safe if and only if the characteristic of $\kk$ is $0$ or larger than $m$.
\item For any given $m \geq 0$, all but finitely many $q \in \kk$ are $m$-safe.
\item If $q$ is $m$-safe, then $q$ is $\ell$-safe for each $\ell \leq m$.
\item If $q$ is $n$-safe, then $\HH_n$ is semisimple (by \cite[Theorem 1.18 and Exercise 1.8]{Mathas}).
\end{itemize}

A safety condition does not let us use $\R_{n,1}$ to separate eigenvalues by content set (as seen in the case $q=1$ in Example \ref{ex.eigenvalues.coincide}), but we can use the collection $\{\R_{n,k}:0 \leq k \leq n\}$ to do so. We will prove the following:

\begin{theorem}\label{thm.BJS-is-e}
Assume that $q$ is $2n$-safe. Let
$\lambda\sm\mu$ and $\alpha\sm\beta$ be two horizontal strips with $\lambda, \alpha \vdash n$. Then,
\[
\eigen_{\lambda\sm\mu}^{(k)}=\eigen_{\alpha\sm\beta}^{(k)}
\qquad \text{ for all } 0 \leq k \leq n
\]
if and only if
\[
\content(\lambda \sm \mu) = \content(\alpha \sm \beta).
\]
\end{theorem}

To prove Theorem~\ref{thm.BJS-is-e}, we will need a few lemmas.
The first one rewrites safety in a more convenient way:
\begin{lem}\label{lem.qint.differ}
    Assume that $q$ is $m$-safe, where $m \geq 0$.
    Then, any $m$ consecutive $q$-integers $[i+1]_q,\ [i+2]_q,\ \ldots,\ [i+m]_q$ (where $i \in \ZZ$ is arbitrary) are distinct.
\end{lem}
\begin{proof}
    If $a<b$, then $[b]_q - [a]_q = q^a [b-a]_q$, so that $[a]_q$ can only equal $[b]_q$ if $q = 0$ or $[b-a]_q = 0$.
    But none of this can happen when $q$ is $m$-safe and $i+1 \leq a < b \leq i+m$.
\end{proof}


Our next lemma is a bound on the contents of cells in a horizontal strip.

\begin{lem}\label{lem.horstrip-bound}
Let $\lambda \sm \mu$ be a horizontal strip with $\lambda \vdash n$ and $\mu \vdash j$.
Let $(r,s)$ be any cell of $\lambda \sm \mu$.
Then, $-j \leq \content(r,s) < n$.
\end{lem}
\begin{proof}
    A bit of easy combinatorics. See Appendix B in the arXiv version for the details.
\end{proof}

\begin{lem}
\label{lem.Enzero}
Assume that $q$ is $(n+j)$-safe.
Let $\lambda\sm\mu$ be a horizontal strip with
$\lambda\vdash n$ and $\mu\vdash j$.
Then, $\eigen_{\lambda\sm\mu}^{(n-j)}\neq0$.
\end{lem}

\begin{proof}
Proposition~\ref{prop:eval.largestk} shows that
\begin{equation}
\eigen_{\lambda\sm\mu}^{(n-j)}=\prod_{(r,s)\in\lambda\sm\mu
}[j+1+\content(r,s)]_{q}.
\label{pf.lem.Enzero.1}
\end{equation}
However, Lemma~\ref{lem.horstrip-bound} shows that
each $\left(  r,s\right)  \in\lambda\sm\mu$ satisfies
$-j \leq \content(r,s) < n$,
therefore $\content(r,s) \in \set{-j,-j+1,\ldots,n-1}$,
and thus $j+1+\content(r,s) \in \set{1,2,\ldots,n+j}$.
Since $q$ is $(n+j)$-safe, this entails that
$[j+1+\content(r,s)]_{q}\neq 0$.
So all the factors in the product in
\eqref{pf.lem.Enzero.1} are $\neq0$. Thus, \eqref{pf.lem.Enzero.1} becomes
$\eigen_{\lambda\sm\mu}^{(n-j)}\neq0$.
\end{proof}

Next, we will need the following application of Vieta's formulas. 

\begin{lem}\label{lem.e=e}
Let $\sfA := ( a_{1},a_{2},\ldots,a_{p} )$ and
$\sfB := ( b_{1},b_{2},\ldots,b_{p} )$ be two unordered lists of scalars
in $\kk$, each containing $p$ entries.
Assume that
\[
e_i \left(  a_{1},a_{2},\ldots,a_{p}\right)
= e_i \left(  b_{1},b_{2},\ldots,b_{p}\right)
\qquad \text{for all }i\in\left\{ 0,1,\ldots,p\right\}  ,
\]
where $e_{i}$ denotes the $i$-th elementary symmetric polynomial.
Then $\sfA = \sfB$ (as unordered lists).
\end{lem}

\begin{proof}
Using Vieta's formulas, conclude
$\prod_{i=1}^{p}\left(  x-a_{i}\right)  =\prod_{i=1}^{p}\left(  x-b_{i}\right)$;
then use unique factorization in $\kk[x]$.
See Appendix B in the arXiv version for details.
\end{proof}



We can now prove Theorem~\ref{thm.BJS-is-e}.

\begin{proof}[Proof of Theorem~\ref{thm.BJS-is-e}]
If $\content(\lambda\sm\mu)=\content(\alpha\sm\beta)$, then we have $\eigen_{\lambda\sm\mu}^{(k)}=\eigen_{\alpha\sm\beta}^{(k)}$ for all $k \geq 0$, as we saw in the first paragraph of the proof of Theorem~\ref{thm.A1}.

Now assume that $\eigen_{\lambda\sm\mu}^{(k)}=\eigen_{\alpha\sm\beta}^{(k)}$ for all $k\in\left\{  0,1,\ldots,n\right\}  $.
We must prove that $\content(\lambda\sm\mu)=\content(\alpha
\sm\beta)$.

We define the further unordered tuple
\[
\content_{q}(\lambda\sm\mu)
:= \left(  \left[  \content\left(  r,s\right)  \right]  _{q}
\ :\ \left(  r,s\right)  \in\lambda\sm \mu\right)
= \left(  \left[  m\right]  _{q}\ :\ m\in\content\left(
\lambda\sm\mu\right)  \right)
\]
(since $\content(\lambda \sm \mu) = \left( \content(r, s) : (r,s) \in \lambda \sm \mu \right)$).
Likewise, define $\content_{q}(\alpha\sm\beta)$.

Set $j = \abs{\mu}$, and recall that $\abs{\lambda} = n$.
Corollary~\ref{cor:eigen0forbigk} shows that
$\eigen_{\lambda\sm\mu}^{(k)}=0$ for all $k>\abs{\lambda\sm\mu} = n-j$,
whereas Lemma~\ref{lem.Enzero} shows that $\eigen_{\lambda\sm\mu}^{(n-j)}\neq0$
(since $q$ is $2n$-safe and therefore $(n+j)$-safe\footnote{because $j \leq n$ entails $n+j\leq 2n$}).
Hence, the largest $k\geq0$ for which
$\eigen_{\lambda\sm\mu}^{(k)}\neq0$ is $n-j=
\abs{\lambda\sm\mu}$.
Similarly, the largest $k\geq0$ for which
$\eigen_{\alpha\sm\beta}^{(k)}\neq0$ is
$\abs{\alpha\sm\beta}$.
Since $\eigen_{\lambda\sm\mu
}^{(k)}=\eigen_{\alpha\sm\beta}^{(k)}$ for all $k\in\left\{
0,1,\ldots,n\right\}  $, it thus follows that $\left\vert \lambda\sm
\mu\right\vert =\left\vert \alpha\sm\beta\right\vert =n-j$.

So each of the skew shapes $\lambda\sm\mu$ and $\alpha\sm\beta$
has $n-j$ boxes. Hence, $\mathsf{Y}_{\lambda\sm\mu}^{(k)}%
=\mathsf{Y}_{\alpha\sm\beta}^{(k)}$ for each $k\geq0$, since the
definition in \eqref{eq:specialY} depends only on $k$ and $|\lambda
\sm\mu|$. Let us denote this unordered tuple simply by $\mathsf{Y}%
^{\left(  k\right)  }$. Moreover, the definition \eqref{eq:specialX} shows
that the entries $\mathsf{X}_{\lambda\sm\mu}^{(k)}$ are precisely the
entries of $\content_{q}(\lambda\sm\mu)$ times $q^{n-k+1}$. Hence, for
each $\ell\geq0$, we have%
\[
e_{\ell}\left(  \mathsf{X}_{\lambda\sm\mu}^{(k)}\right)  =q^{\left(
n-k+1\right)  \ell}e_{\ell}\left(  \content_{q}(\lambda\sm\mu)\right)
\]
(since $e_{\ell}$ is homogeneous of degree $\ell$). Thus,
\eqref{eq:eigenformula} rewrites as%
\begin{align*}
\eigen_{\lambda\sm\mu}^{\left(  k\right)  } &  =e_{k}\left(
\mathsf{X}_{\lambda\sm\mu}^{(k)}/\mathsf{Y}^{\left(  k\right)  }\right)
=\sum_{\ell=0}^{k}\underbrace{e_{\ell}\left(  \mathsf{X}_{\lambda\sm\mu
}^{(k)}\right)  }_{=q^{\left(  n-k+1\right)  \ell}e_{\ell}\left(
\content_{q}(\lambda\sm\mu)\right)  }h_{k-\ell}\left(  \mathsf{Y}%
^{\left(  k\right)  }\right)  \\
&  =\sum_{\ell=0}^{k}q^{\left(  n-k+1\right)  \ell}e_{\ell}\left(
\content_{q}(\lambda\sm\mu)\right)  h_{k-\ell}\left(  \mathsf{Y}%
^{\left(  k\right)  }\right)  \\
&  =\sum_{\ell=0}^{k-1}q^{\left(  n-k+1\right)  \ell}e_{\ell}\left(
\content_{q}(\lambda\sm\mu)\right)  h_{k-\ell}\left(  \mathsf{Y}%
^{\left(  k\right)  }\right)  +q^{\left(  n-k+1\right)  k}e_{k}\left(
\content_{q}(\lambda\sm\mu)\right)  \underbrace{h_{0}\left(
\mathsf{Y}^{\left(  k\right)  }\right)  }_{=1}\\
&  =\sum_{\ell=0}^{k-1}q^{\left(  n-k+1\right)  \ell}e_{\ell}\left(
\content_{q}(\lambda\sm\mu)\right)  h_{k-\ell}\left(  \mathsf{Y}%
^{\left(  k\right)  }\right)  +q^{\left(  n-k+1\right)  k}e_{k}\left(
\content_{q}(\lambda\sm\mu)\right)  ,
\end{align*}
and likewise we find
\[
\eigen_{\alpha\sm\beta}^{\left(  k\right)  }=\sum_{\ell=0}%
^{k-1}q^{\left(  n-k+1\right)  \ell}e_{\ell}\left(  \content_{q}%
(\alpha\sm\beta)\right)  h_{k-\ell}\left(  \mathsf{Y}^{\left(  k\right)
}\right)  +q^{\left(  n-k+1\right)  k}e_{k}\left(  \content_{q}%
(\alpha\sm\beta)\right)  .
\]
If $k\in\left\{  0,1,\ldots,n\right\}  $, then the left-hand sides of these
two equalities are equal (by our $\eigen_{\lambda\sm\mu}%
^{(k)}=\eigen_{\alpha\sm\beta}^{(k)}$ assumption), and so the
right-hand sides must be equal as well. Hence, by strong induction on $k$, we
conclude that
\[
e_{k}\left(  \content_{q}(\lambda\sm\mu)\right)  =e_{k}\left(
\content_{q}(\alpha\sm\beta)\right)  \ \ \ \ \ \ \ \ \ \ \text{for all
}k\in\left\{  0,1,\ldots,n\right\}
\]
(since the induction hypothesis shows that $e_{\ell}\left(  \content%
_{q}(\lambda\sm\mu)\right)  =e_{\ell}\left(  \content_{q}%
(\alpha\sm\beta)\right)  $ for all $\ell\leq k-1$, and the assumption
$q\neq0$ allows us to divide by $q^{\left(  n-k+1\right)  k}$). In particular,
this equality holds for all $k\in\left\{  0,1,\ldots,n-j\right\}  $.
Therefore, Lemma~\ref{lem.e=e} (applied to $p=n-j$ and $\sfA =
\content_{q}(\lambda\sm\mu)$ and $\sfB = \content_{q}(\alpha\sm\beta)$)
yields
$\content_{q}(\lambda\sm\mu)=\content_{q}(\alpha\sm\beta)$.

However, Lemma~\ref{lem.horstrip-bound} shows that all cells $(r,s)\in \lambda\sm\mu$ satisfy $\content(r,s) \in \set{-j,-j+1,\ldots,n-1}$.
The same holds for all cells $(r,s)\in \alpha\sm\beta$ (for the same reason).
So the contents of all these cells are among the $n+j$ consecutive integers $-j,-j+1,\ldots,n-1$.
Thus, if two of these cells have contents $a$ and $b$ satisfying $[a]_q = [b]_q$, then $a = b$
(since $q$ is $(n+j)$-safe, so that Lemma~\ref{lem.qint.differ} yields that the $q$-integers $[-j]_q,\ [-j+1]_q,\ \ldots,\ [n-1]_q$ are distinct).
Therefore, from $\content_{q}(\lambda\sm\mu)=\content_{q}(\alpha
\sm\beta)$,
we obtain $\content(\lambda \sm \mu) = \content(\alpha \sm \beta)$, which completes the proof.
\end{proof}

One consequence of Theorem~\ref{thm.BJS-is-e} is that the $\R_{n,k}$ are diagonalizable when $q$ is $2n$-safe; this was heretofore only known when $q \in \RR_{>0}$. The key idea is that the separation of eigenvalues in Theorem~\ref{thm.BJS-is-e} allows us to conclude that the eigenspaces of $\R_{n,k}$ are of a sufficiently large total dimension.

Recall from Definition~\ref{def:desstrip} that $\mathcal{D}_n$ is the collection of horizontal strips $\lambda \sm \mu$ such that $|\lambda| = n$ and $d^\mu \neq 0$.

\begin{cor}\label{cor:eigenbasis.semisimple}
Suppose $q$ is $2n$-safe. Then $\R_{n,k}$ is diagonalizable.
\end{cor}

\begin{proof}
Given $\lambda \sm \mu \in \mathcal{D}_n$, let $\gamma:= \content(\lambda \sm \mu)$, and let $\eigen_\gamma^{(k)}$ denote the eigenvalue $\eigen^{(k)}_{\lambda \sm \mu}$ corresponding to this content set (this depends only on $\gamma$, by Theorem~\ref{thm.BJS-is-e}). In other words, $\eigen_{\gamma}^{(k)} = \eigen_{\alpha \sm \beta}^{(k)}$ for all $\alpha \sm \beta \in \mathcal{D}_n$ with $\content(\alpha \sm \beta) = \gamma$ (see e.g. Remark~\ref{rmk:contentvshorizontalstrip}). 
Consider the subspace
\[
V_\gamma:=  \bigcap_{k=0}^n \ker \big( \R_{n,k} - \eigen_{\gamma}^{(k)}\big) \subseteq \HH_n .
\]
Using the eigenbasis $\basis_\lambda$ for $q \in \RR_{>0}$, one can use a standard polynomial argument%
\footnote{In more detail:
By its definition, $V_\gamma$ is the space of all joint $\R_{n,k}$-eigenvectors with respective eigenvalues $\eigen^{(k)}_{\gamma}$.
If $q \in \RR_{>0}$, then all the basis vectors $y^{\lambda}_{\mu(u)} \in \basis_\lambda$ (in the respective copies of the Specht modules $S^\lambda$ inside $\HH_n$) with $\lambda \sm \mu \in \mathcal{D}_n$ satisfying $\content(\lambda \sm \mu) = \gamma$ are such joint eigenvectors; hence, they belong to $V_\gamma$.
Since they are linearly independent, this gives a lower bound on $\dim(V_\gamma)$: namely, it shows that, if $q \in \RR_{>0}$, then $\dim(V_\gamma)$ is at least the number of these basis vectors, i.e., the sum of $d^\mu f^\lambda$ over all $\lambda \sm \mu \in \mathcal{D}_n$ satisfying $\content(\lambda \sm \mu) = \gamma$
(by the argument proving Corollary~\ref{cor:charpoly}).
Let us denote the latter sum by $s_\gamma$. (Note that it does not depend on $q$.)
So $\dim(V_\gamma) \geq s_\gamma$ when $q \in \RR_{>0}$.
\par
But the statement ``$\dim(V_\gamma) \geq s_\gamma$'' can be restated in the language of explicit matrix algebra: The space $V_\gamma$ is the kernel of an $(n+1)!\times n!$-matrix whose entries are integer polynomials in $q$ (namely, the matrix obtained by representing the actions of $\R_{n,k} - \eigen_{\gamma}^{(k)}$ on $\HH_n$ for all $k \in \set{0,1,\ldots,n}$ as matrices in the standard basis of $\HH_n$ and stacking the resulting $n+1$ matrices atop each other); thus, the statement ``$\dim(V_\gamma) \geq s_\gamma$'' is just saying that the nullity of this matrix is at least $s_\gamma$, that is, all minors of size larger than $n!-s_\gamma$ of this matrix are $0$.
This is a collection of polynomial identities in $q$.
Hence, if it holds for all $q \in \RR_{>0}$, then it must hold for all $\kk$ and all $q$.
So we have proved that $\dim(V_\gamma) \geq s_\gamma$ for all $\kk$ and all $q$.}
to show that for arbitrary $q$, one has 
\begin{equation}\label{eq:dimVgamma}
    \dim\big( V_\gamma \big) \geq \sum_{ \substack{ \lambda \sm \mu \in \mathcal{D}_n; \\ \content(\lambda \sm \mu) = \gamma}} d^\mu f^\lambda.
\end{equation}
But Theorem~\ref{thm.BJS-is-e} shows that distinct content sets $\gamma$ have different tuples of eigenvalues $(\eigen^{(0)}_\gamma, \eigen^{(1)}_\gamma, \ldots, \eigen^{(n)}_\gamma)$. Hence,
the subspaces $V_\gamma$ are linearly disjoint for distinct content sets $\gamma$ (by a multi-operator version of the classical fact that distinct eigenspaces of an operator are linearly disjoint\footnote{The proper formulation is representation-theoretical: Distinct isotypic components of a representation of an algebra are linearly disjoint.}).
But the sum of their dimensions is at least $\dim(\HH_n)$: Indeed, summing all the inequalities \eqref{eq:dimVgamma} yields $\sum_\gamma \dim\big( V_\gamma \big) \geq \sum_{ \lambda \sm \mu \in \mathcal{D}_n } d^\mu f^\lambda$, and the right hand side is easily seen to be $n!$ because $\basis_\lambda$ is a basis for $q \in \RR_{>0}$ (and this is a purely combinatorial identity, so it suffices to prove it for one $q$).
Hence, the $V_\gamma$ form a direct sum decomposition of $\HH_n$:
\begin{align}\label{eq:safe-eigendecomp}
 \HH_n = \bigoplus_{\gamma} V_\gamma.
\end{align}
This shows that all the $\R_{n,k}$ are diagonalizable.
\end{proof}

\begin{remark}
Corollary~\ref{cor:eigenbasis.semisimple} fails if we only require $q$ to be $n$-safe. For example, $\R_{5,1}$ fails to be diagonalizable if $q$ is a primitive $7$-th root of unity $\omega_7$, even though $\HH_5$ is semisimple in this case. 
\end{remark}

\subsection{Consequences for the random-to-random algebra}\label{subsec:r2r-algebra}

We can now draw consequences about the algebra generated by the $k$-random-to-random shuffles.

\begin{definition}
    For a fixed $n$, define $\ralgebra_n$ to be the $\kk$-subalgebra of $\HH_n$ generated by the $\R_{n,k}$ shuffles:
    \[ \ralgebra_n:= \langle \R_{n,k}: 0 \leq k \leq n \rangle.\]
\end{definition}
Note that this algebra $\ralgebra_n$ is commutative, since its generators $\R_{n,k}$ pairwise commute by Theorem~\ref{thm:commut}.
Of course, the generator $\R_{n,0} = 1$ can be omitted.

We first note the following folklore Lemma. Fix $R$ to be an arbitrary $\kk$-algebra, and let $S$ be a finite set of $\kk$-algebra morphisms $\chi : R\to \kk$ (we call them ``characters'').  We say that $S$ \emph{separates} $R$ if each element of $R$ is uniquely determined by the family of its images under the $\chi \in S$, that is, if $\bigcap_{\chi \in S} \ker\chi = 0$.

\begin{lem}\label{lem.A2}
Let $R$ be a $\kk$-algebra, where $\kk$ is a field. Suppose $S$ separates $R$. Then:
\begin{enumerate}
\item[(a)] We have $\dim R = \abs{S}$. More concretely: $S$ is a basis of the dual vector space $R^*$.
\item[(b)] Let $r \in R$ be an element such that any two distinct characters $\phi,\psi \in S$ satisfy $\phi(r) \neq \psi(r)$. Then $R = \kk[r]$.
\end{enumerate}
\end{lem}

\begin{proof}
See Appendix B to the arXiv version of this paper.
\end{proof}

We will now set $R$ to be $\ralgebra_n$, and prove that $\ralgebra_n$ can in fact be generated by just $\R_{n,1}$.

\begin{theorem}\label{thm.A3}
Assume that $q$ is not integral over $\ZZ$. Then, $\ralgebra_n = \kk[\R_{n,1}]$.
\end{theorem}

\begin{proof}
By assumption, $q$ is $2n$-safe.
Hence, \eqref{eq:safe-eigendecomp} shows that $\ralgebra_n$ (or, rather, its action on $\HH_n$ by right multiplication) has a joint eigenbasis $(u_1, u_2, \ldots, u_{n!})$ on $\HH_n$
(obtained by choosing a basis of each $V_\gamma$).

For each $i \in \{ 1, 2, \ldots, n!\}$, let $\chi_i : \ralgebra_n \to \kk$ be the linear map that sends each $r \in \ralgebra_n$ to its eigenvalue on the eigenvector $u_i$. This is a $\kk$-algebra morphism.
Let $S = \set{\chi_1, \chi_2, \ldots, \chi_{n!}}$ be the set of these $n!$ morphisms.
Then, the set $S$ separates $\ralgebra_n$, because two elements of $\ralgebra_n$ that act identically on all the $u_i$ must be equal, since the $u_i$ form a basis of $\HH_n$.

Any character $\chi \in S$ sends each $r \in \ralgebra_n$ to its eigenvalue on one of the eigenvectors $u_1, u_2, \ldots, u_{n!}$ (namely, on the eigenvector $u_i$, if $\chi = \chi_i$).
This eigenvector lies in one of the eigenspaces $V_\gamma$ in \eqref{eq:safe-eigendecomp} for some content set $\gamma$, and thus the eigenvalues of the elements $\R_{n,k}$ on this eigenvector are the $\eigen_{\gamma}^{\tup{k}}$.
Thus, for any character $\chi \in S$, there exists a content set $\gamma = \gamma_\chi$ (depending only on $\chi$) such that all $k \geq 0$ satisfy
\begin{align}
\chi\tup{\R_{n,k}} = \eigen_{\gamma_\chi}^{\tup{k}} .
\label{pf.thm.A3.chi-eigen}
\end{align}

Now, we shall show that any two distinct characters $\phi, \psi \in S$ differ on $\R_{n,1}$.

Indeed, let $\phi, \psi \in S$ be two distinct characters.
Let $\gamma_{\phi}$ and $\gamma_{\psi}$ be the content sets associated to them.
The characters $\phi$ and $\psi$ are distinct, and thus must differ on $\R_{n,k}$ for some $k \geq 0$
(since otherwise, they would agree on all the generators $\R_{n,k}$ of the $\kk$-algebra $\ralgebra_n$,
and thus -- being $\kk$-algebra morphisms -- would have to be equal).
In other words, $\phi\tup{\R_{n,k}} \neq \psi\tup{\R_{n,k}}$ for some $k \geq 0$.

In view of \eqref{pf.thm.A3.chi-eigen}, this rewrites as $\eigen_{\gamma_\phi}^{\tup{k}} \neq \eigen_{\gamma_\psi}^{\tup{k}}$ for some $k \geq 0$.
Hence, Corollary~\ref{cor.A1k} (applied to two horizontal strips $\lambda \sm \mu$ and $\alpha \sm \beta$ in $\mathcal{D}_n$ that satisfy $\content(\lambda \sm \mu) = \gamma_{\phi}$ and $\content(\alpha \sm \beta) = \gamma_{\psi}$)
yields $\eigen_{\gamma_\phi}^{(1)} \neq \eigen_{\gamma_\psi}^{(1)}$.
Using \eqref{pf.thm.A3.chi-eigen} again, we can rewrite this as $\phi\tup{\R_{n,1}} \neq \psi\tup{\R_{n,1}}$.
In other words, $\phi$ and $\psi$ differ on $\R_{n,1}$.

We thus have shown that any two distinct characters $\phi, \psi \in S$ differ on $\R_{n,1}$.
Hence, Lemma~\ref{lem.A2}(b) applied to $r = \R_{n,1}$ yields $\ralgebra_n = \kk[\R_{n,1}]$.
\end{proof}

\begin{remark}
In fact, using an almost identical proof, one can use Theorem~\ref{thm.BJS-is-e} to prove that when $\kk$ is infinite, there is a Zariski-dense set of $r \in \ralgebra_n$ such that $\ralgebra_n = \kk[r]$ whenever $q$ is $2n$-safe.
\end{remark}

\begin{remark}
Theorem~\ref{thm.A3} fails for $q = 1$ and $n = 4$. The non-integrality of $q$ is essential!
\end{remark}

\begin{theorem}\label{thm.Adim}
Assume that $q$ is $2n$-safe. Then, the dimension of the $\kk$-vector space $\ralgebra_n$ is the number of all distinct content sets of all $\lambda \sm \mu \in \mathcal{D}_n$.
\end{theorem}

\begin{proof}
Proceed as in the proof of Theorem~\ref{thm.A3} above until the proof that $S$ separates $\ralgebra_n$.
At that point, apply Lemma~\ref{lem.A2}(a) to obtain $\dim \ralgebra_n = \abs{S}$.
But we also saw in the same proof that each character $\chi \in S$ is associated to a content set $\gamma_\chi$ of some $\lambda \sm \mu \in \mathcal{D}_n$, which is uniquely determined by $\chi$ (since Theorem~\ref{thm.BJS-is-e} shows that distinct content sets create different characters).
This is a one-to-one correspondence (since any content set of any $\lambda \sm \mu \in \mathcal{D}_n$ creates at least one eigenvector $u_i$ and thus at least one character $\chi_i \in S$).
Thus, $\abs{S}$ is the number of these content sets.
\end{proof}

\begin{remark}
These dimensions for $n = 0,1,2,\ldots$ form the \href{https://oeis.org/A393539}{OEIS sequence A393539}.
It would be very interesting to find an elegant expression for this dimension.
\end{remark}

\section*{Acknowledgments}
We thank Pavel Etingof, Nadia Lafreni\`{e}re, John Lentfer, Shifra Reif, and Vic Reiner for interesting and informative conversations.
Work on this paper was done at the Mathematisches Forschungsinstitut Oberwolfach in October 2024, as the four authors were Oberwolfach Research Fellows (2442p),
and at the ICERM program ``Categorification and Computation in Algebraic Combinatorics''\ in Fall 2025
(supported by the National Science Foundation under Grant No. DMS-1929284).
The first author is partially supported by the NSF MSPRF DMS-2303060 and the second author was partially supported by an NSF GRFP fellowship and was in-residence at the Simons Laufer Mathematical Sciences Institute for part of this work. 
The fourth author was supported by NSERC (RGPIN-2023-04476).
The SageMath computer algebra system \cite{sagemath} was used to find several of the results.

\section*{AI use Statement}
GPT-6 Astra was used to proofread this paper, catching various minor mistakes.
No AI was used in generating the mathematical ideas or in writing the text.

\bibliographystyle{abbrv}
\bibliography{bibliography}

\appendix

\newpage

\section{The eigenvalues $\eigen_{\lambda \setminus \mu}^{(k)}$, expanded}\label{sec:Appendix-Raw-Data}
We collect here the eigenvalues of $\R_{n,k}$ acting on $\HH_n$ for $n \leq 5$ and $k \leq 5$.

\subsection*{$n = 3$}\begin{center}
\renewcommand{\arraystretch}{1.5}
\begin{tabular}{c|c|c}
   $\lambda \setminus \mu$  &  $k$ & $\eigen_{\lambda \setminus \mu}^{(k)}$\\ \hline \hline
    $(3) \setminus \emptyset$ & $1$ & $q^{4} + 2q^{3} + 3q^{2} + 2q + 1
$\\ \hline 
 $(3) \setminus \emptyset$ & $2$ & $q^{5} + 3q^{4} + 5q^{3} + 5q^{2} + 3q + 1
$\\ \hline
    $(3) \setminus \emptyset$ & $3$ & $q^{3} + 2q^{2} + 2q + 1
$\\ \hline
     $(2,1) \setminus (1,1)$ & $1$ & $q^3 + q^2 + q + 1$\\ \hline
     $(1,1,1) \setminus (1,1)$ & $1$ & $1$\\ \hline
\end{tabular}
\end{center}

\subsection*{$n = 4$}
\begin{center}
\renewcommand{\arraystretch}{1.5}
\begin{tabular}{c|c|c}
   $\lambda \setminus \mu$  &  $k$ & $\eigen_{\lambda \setminus \mu}^{(k)}$\\ \hline \hline
    $(4) \setminus \emptyset$ & $1$ & $q^{6} + 2q^{5} + 3q^{4} + 4q^{3} + 3q^{2} + 2q + 1
$ \\ \hline 
 $(4) \setminus \emptyset$ & $2$ & $q^{9} + 3q^{8} + 7q^{7} + 11q^{6} + 14q^{5} + 14q^{4} + 11q^{3} + 7q^{2} + 3q + 1
$\\ \hline
    $(4) \setminus \emptyset$ & $3$ & $q^{9} + 4q^{8} + 9q^{7} + 15q^{6} + 19q^{5} + 19q^{4} + 15q^{3} + 9q^{2} + 4q + 1
$\\ \hline
     $(4) \setminus \emptyset$ & $4$ & $q^{6} + 3q^{5} + 5q^{4} + 6q^{3} + 5q^{2} + 3q + 1
$\\ \hline
     $(3,1) \setminus (2,1)$ & $1$ & $q^5 + q^4 + q^3 + q^2 + q + 1$\\ \hline
     $(3,1) \setminus (1,1)$ & $1$ & $q^{5} + 2q^{4} + 2q^{3} + 2q^{2} + 2q + 1
$\\ \hline
     $(3,1) \setminus (1,1)$ & $2$ & $q^{7} + 2q^{6} + 3q^{5} + 4q^{4} + 4q^{3} + 3q^{2} + 2q + 1
$\\ \hline
     $(2,2) \setminus (2,1)$ & $1$ & $q^3 + q^2 + q + 1$\\ \hline
      $(2,1,1) \setminus (2,1)$ & $1$ & $q + 1$\\ \hline
    $(2,1,1) \setminus (1,1)$ & $1$ & $q^{4} + q^{3} + q^{2} + 2q + 1
$\\ 
    \hline
     $(2,1,1) \setminus (1,1)$ & $2$ & $q^3 + q^2 + q + 1$\\ 
\end{tabular}
\end{center}

\newpage
\subsection*{$n = 5$}
\begin{center}
\renewcommand{\arraystretch}{1.5}
\begin{tabular}{c|c|c}
   $\lambda \setminus \mu$  &  $k$ & $\eigen_{\lambda \setminus \mu}^{(k)}$\\ \hline \hline
   $(5) \setminus \emptyset$ & $1$ & $q^{8} + 2q^{7} + 3q^{6} + 4q^{5} + 5q^{4} + 4q^{3} + 3q^{2} + 2q + 1
$\\ \hline
   $(5) \setminus \emptyset$ & $2$ & $q^{13} + 3q^{12} + 7q^{11} + 13q^{10} + 20q^{9} + 26q^{8} + 30q^{7} + $\\
& & $30q^{6} + 26q^{5} + 20q^{4} + 13q^{3} + 7q^{2} + 3q + 1
$ \\ \hline
   $(5) \setminus \emptyset$ & $3$ & $q^{15} + 4q^{14} + 11q^{13} + 23q^{12} + 40q^{11} + 59q^{10} + 76q^{9} + 86q^{8} + 86q^{7} + 76q^{6} +  
$ \\ & & $ 59q^{5} + 40q^{4} + 23q^{3} + 11q^{2} + 4q + 1$\\ \hline
   $(5) \setminus \emptyset$ & $4$ & $q^{14} + 5q^{13} + 14q^{12} + 29q^{11} + 49q^{10} + 70q^{9} + 86q^{8} + 92q^{7} + 86q^{6} + 70q^{5} + 
$\\ & & $49q^{4} + 29q^{3} + 14q^{2} + 5q + 1$ \\ \hline
   $(5) \setminus \emptyset$ & $5$ & $q^{10} + 4q^{9} + 9q^{8} + 15q^{7} + 20q^{6} + 22q^{5} + 20q^{4} + 15q^{3} + 9q^{2} + 4q + 1
$\\ \hline
   $(4, 1) \setminus (3,1)$ & $1$ & $q^{7} + q^{6} + q^{5} + q^{4} + q^{3} + q^{2} + q + 1
$\\ \hline
   $(4,1) \setminus (2,1)$ & $1$ & $q^{7} + 2q^{6} + 2q^{5} + 2q^{4} + 2q^{3} + 2q^{2} + 2q + 1
$\\ \hline
$(4,1) \setminus (2,1)$ & $2$ & $q^{11} + 2q^{10} + 3q^{9} + 4q^{8} + 5q^{7} + 6q^{6} + 6q^{5} + 5q^{4} + 4q^{3} + 3q^{2} + 2q + 1
$\\ \hline
$(4,1) \setminus (1,1)$ & $1$ & $q^{7} + 2q^{6} + 3q^{5} + 3q^{4} + 3q^{3} + 3q^{2} + 2q + 1
$ \\ \hline
$(4,1) \setminus (1,1)$ & $2$ & $q^{11} + 3q^{10} + 6q^{9} + 9q^{8} + 12q^{7} + 14q^{6} + 14q^{5} + 12q^{4} + 9q^{3} + 6q^{2} + 3q + 1
$\\ \hline
$(4,1) \setminus (1,1)$ & $3$ & $q^{12} + 3q^{11} + 6q^{10} + 10q^{9} + 14q^{8} + 17q^{7} + 18q^{6} + 17q^{5} + 14q^{4} + 10q^{3} + 6q^{2} + 3q + 1
$\\ \hline
$(3,2) \setminus (2,2)$ & $1$ & $q^6 + q^5 + q^4 + q^3 + q^2 + q + 1$\\ \hline
$(3,2) \setminus (3,1)$ & $1$ & $q^4 + q^3 + q^2 + q + 1$\\ \hline
$(3,2) \setminus (2,1)$ & $1$ & $q^{6} + q^{5} + 2q^{4} + 2q^{3} + 2q^{2} + 2q + 1
$\\ \hline
$(3,2) \setminus (2,1)$ & $2$ & $q^{8} + 2q^{7} + 3q^{6} + 4q^{5} + 4q^{4} + 4q^{3} + 3q^{2} + 2q + 1
$ \\ \hline
$(3,1,1) \setminus (2,1,1)$ & $1$ & $q^6 +q^5 + q^4 + q^3 + q^2 + q + 1$ \\ \hline
$(3,1,1) \setminus (3,1)$ & $1$ & $q^2 + q + 1$\\ \hline
$(3,1,1) \setminus (2,1)$ & $1$ & $q^{6} + q^{5} + q^{4} + q^{3} + 2q^{2} + 2q + 1
$\\ \hline
$(3,1,1) \setminus (2,1)$ & $2$ & $q^{6} + 2q^{5} + 2q^{4} + 2q^{3} + 2q^{2} + 2q + 1
$ \\ \hline
$(3,1,1) \setminus (1,1)$ & $1$ & $q^{6} + 2q^{5} + 2q^{4} + 2q^{3} + 3q^{2} + 2q + 1
$\\ \hline
$(3,1,1) \setminus (1,1)$ & $2$ & $q^{9} + 2q^{8} + 3q^{7} + 5q^{6} + 7q^{5} + 7q^{4} + 6q^{3} + 5q^{2} + 3q + 1
$ \\ \hline
$(3,1,1) \setminus (1,1)$ & $3$ & $q^{7} + 2q^{6} + 3q^{5} + 4q^{4} + 4q^{3} + 3q^{2} + 2q + 1
$ \\ \hline
$(2,2,1) \setminus (2,2)$ & $1$ & $q^2 + q + 1$\\ \hline
$(2,2,1) \setminus (2,1,1)$ & $1$ & $q^4 + q^3 + q^2 + q + 1$ \\ \hline
$(2,2,1) \setminus (2,1)$ & $1$ & $q^{4} + q^{3} + 2q^{2} + 2q + 1
$\\ \hline
$(2,2,1) \setminus (2,1)$ & $2$ & $q^{4} + 2q^{3} + 2q^{2} + 2q + 1
$ \\ \hline
$(2,1,1,1) \setminus (1,1,1,1)$ & $1$ & $q^5 + q^4 + q^3 + q^2 + q + 1$\\ \hline
$(2, 1, 1, 1) \setminus (2,1, 1)$ & $1$ & $q + 1$\\ \hline
$(1,1, 1, 1, 1) \setminus (1,1,1, 1)$ & $1$ & $1$\\ \hline
\end{tabular}
\end{center}

\begin{verlong}
\section{\label{sec:yjm-proofs}Proofs of some folklore lemmas}

\begin{proof}[Proof of Proposition~\ref{prop:rewr}.]
For each $i \in \left\{1,2,\ldots,n\right\}$, we have $T_{n-1} T_{n-2} \cdots T_i = T_{(n, n-1, \ldots, i)}$, due to the fact that $s_{n-1} s_{n-2} \cdots s_i$ is a reduced expression for the cycle $(n, n-1, \ldots, i)$.
Thus, the first equality sign in \eqref{eq:prop:rewr:B} is just a restatement of the definition of $\B_n$.

The second equality sign in \eqref{eq:prop:rewr:B} follows from the combinatorial observation that the permutations $w \in \symm_n$ satisfying $w^{-1}\left(1\right) < w^{-1}\left(2\right) < \cdots < w^{-1}\left(n-1\right)$ are precisely the cycles $(n, n-1, \ldots, i)$ for $i \in \left\{1,2,\ldots,n\right\}$ (because each of the former permutations $w$ is uniquely determined by the value $w^{-1}\left(n\right)$, and if we denote this value by $i$, then our $w$ is the cycle $(n, n-1, \ldots, i)$).
Thus, \eqref{eq:prop:rewr:B} is proved.

The proof of \eqref{eq:prop:rewr:Bast} is similar (here, the permutations $w \in \symm_n$ satisfying $w\left(1\right) < w\left(2\right) < \cdots < w\left(n-1\right)$ are classified by the value $w\left(n\right)$).
\end{proof}

\begin{proof}[Proof of Proposition~\ref{prop:uI}.]
If $k > n$, then $\B_{n,k}$ and $\X_{n,k}$ and $\M_{n,k}$ all equal $0$
(the former by \eqref{eq:Bnk:zero}, the latter two by definition).
Thus, we assume that $k \leq n$ from here on.

We recall some notations from \cite{r2r1}.
A \emph{composition} of $n$ means a finite tuple of positive integers whose sum is $n$.
We shall use exponential notation, so that $(1^{n-k}, k)$ means the composition of $n$ that consists of $n-k$ many $1$s followed by a $k$.
(If $k=0$, then the $k$ is dropped, so this composition is simply $(1^n)$.)

For a composition $\alpha = \left(\alpha_1, \alpha_2, \ldots, \alpha_p\right)$ of $n$, consider the Young subgroup $\symm_\alpha$ of $\symm_n$ consisting of the permutations that permute the smallest $\alpha_1$ elements among themselves, permute the next-smallest $\alpha_2$ elements among themselves, and so on.
Thus, $\symm_{(1^{n-k}, k)} = \symm_{\ive{n}-\ive{n-k}}$.

Furthermore, for a composition $\alpha = \left(\alpha_1, \alpha_2, \ldots, \alpha_p\right)$ of $n$, we let $X_\alpha$ denote the set of minimal-length right coset representatives of $\symm_\alpha$ in $\symm_n$.
Explicitly, these representatives are the permutations that (when written in one-line notation) contain the smallest $\alpha_1$ elements in increasing order from left to right, contain the next-smallest $\alpha_2$ elements in increasing order from left to right, and so on.
Thus, $X_{(n-k, k)}$ consists of the permutations $w \in \symm_n$ that contain the numbers $1, 2, \ldots, n-k$ in increasing order from left to right, and contain the numbers $n-k+1, n-k+2, \ldots, n$ in increasing order from left to right.
These permutations are just the $u_I$ for the $n-k$-element subsets $I$ of $\ive{n}$.

For any composition $\alpha$ of $n$, we define the elements
\begin{align*}
    x_{\alpha} := \sum_{w \in X_\alpha} T_w
    \qquad \text{ and } \qquad
    m_\alpha := \sum_{w \in \symm_\alpha} T_w
    \qquad \text{ in } \HH_n.
\end{align*}
It is shown in \cite[Proposition 3.10]{r2r1}
(for $j = n-k$) that
\begin{align}
	\B_{n, k} = m_{(1^{n-k}, k)} \ x_{(n-k, k)}.
	\label{eq:B=mx}
\end{align}
However, $\symm_{(1^{n-k}, k)} = \symm_{\ive{n}-\ive{n-k}}$ shows that $m_{(1^{n-k}, k)} = \M_{n,k}$, whereas our above description of $X_{(n-k, k)}$ shows that $x_{(n-k, k)} = \X_{n,k}$.
Hence, \eqref{eq:B=mx} can be rewritten as
\[
	\B_{n, k} = \M_{n,k} \X_{n,k}.
\]
This proves \eqref{eq:uI:Bnk=}.

Now let us show \eqref{eq:uI:Mnk*}.
The set $\symm_{\ive{n}-\ive{n-k}}$ is a subgroup of $\symm_n$, and thus is closed under taking inverses (i.e., we have $w^{-1} \in \symm_{\ive{n}-\ive{n-k}}$ for each $w \in \symm_{\ive{n}-\ive{n-k}}$).
Thus, the map $a \mapsto a^*$ permutes the addends of the sum $\M_{n,k} = \sum_{w \in \symm_{\ive{n}-\ive{n-k}}} T_w$.
Hence, $\M_{n,k}^* = \M_{n,k}$.
This proves \eqref{eq:uI:Mnk*}.

We let $\ell\tup{w}$ denote the (Coxeter) length of any permutation $w \in \symm_n$, that is, the number of inversions of $w$.

Next, we shall prove \eqref{eq:uI:TjMnk}.
For this, we fix a positive integer $j$ such that $n-k < j < n$.
Then, $s_j \in \symm_{\ive{n}-\ive{n-k}} = \symm_{(1^{n-k},k)}$.
A permutation $w \in \symm_{\ive{n}-\ive{n-k}}$ will be called \emph{southern} if it satisfies $w^{-1}\tup{j} < w^{-1}\tup{j+1}$, and \emph{northern} otherwise.
If $w$ is a southern permutation, then $\ell\tup{s_jw} = \ell\tup{w} + 1$ and thus $T_{s_jw} = T_j T_w$, so that
\begin{align}
    T_w + T_{s_jw} = T_w + T_j T_w = \tup{1 + T_j} T_w.
    \label{pf:TjMnk:1}
\end{align}
Moreover, the map $w \mapsto s_jw$ is a bijection from the set of all southern permutations in $\symm_{\ive{n}-\ive{n-k}}$ to the set of all northern permutations in $\symm_{\ive{n}-\ive{n-k}}$.
Hence,
\begin{equation}
    \sum_{\substack{w \in \symm_{\ive{n}-\ive{n-k}}; \\ w\text{ is northern}}} T_{w}
    =
    \sum_{\substack{w \in \symm_{\ive{n}-\ive{n-k}}; \\ w\text{ is southern}}} T_{s_jw}.
    \label{pf:TjMnk:2}
\end{equation}

Now, the sum $\M_{n,k} = \sum_{w \in \symm_{\ive{n}-\ive{n-k}}} T_w$ can be split into southern and northern permutations. Thus we find
\begin{align*}
\M_{n,k}
&= \sum_{w \in \symm_{\ive{n}-\ive{n-k}}} T_w
= \sum_{\substack{w \in \symm_{\ive{n}-\ive{n-k}}; \\ w\text{ is southern}}} T_w + \sum_{\substack{w \in \symm_{\ive{n}-\ive{n-k}}; \\ w\text{ is northern}}} T_{w} \\
&= \sum_{\substack{w \in \symm_{\ive{n}-\ive{n-k}}; \\ w\text{ is southern}}} T_w + \sum_{\substack{w \in \symm_{\ive{n}-\ive{n-k}}; \\ w\text{ is southern}}} T_{s_jw}
\qquad \left(\text{by \eqref{pf:TjMnk:2}}\right) \\
&= \sum_{\substack{w \in \symm_{\ive{n}-\ive{n-k}}; \\ w\text{ is southern}}} \left(T_w + T_{s_jw}\right)
= \sum_{\substack{w \in \symm_{\ive{n}-\ive{n-k}}; \\ w\text{ is southern}}} \left(1 + T_j\right) T_w
\end{align*}
(by \eqref{pf:TjMnk:1}).
Left multiplication by $T_j$ transforms this into
\[
T_j \M_{n,k}
= \sum_{\substack{w \in \symm_{\ive{n}-\ive{n-k}}; \\ w\text{ is southern}}} \underbrace{T_j \left(1 + T_j\right)}_{= q \left(1 + T_j\right)} T_w
= q \underbrace{\sum_{\substack{w \in \symm_{\ive{n}-\ive{n-k}}; \\ w\text{ is southern}}} \left(1 + T_j\right) T_w}_{=\M_{n,k}}
= q \M_{n,k}.
\]
Thus, \eqref{eq:uI:TjMnk} is proved.

Equation~\eqref{eq:uI:TjBnk} follows from \eqref{eq:uI:Bnk=} and \eqref{eq:uI:TjMnk}:
if $j$ is a positive integer such that $n - k < j < n$, then
\begin{equation*}
    T_j \B_{n, k} = T_j \M_{n, k} \X_{n, k} = q \M_{n, k} \X_{n, k} = q \B_{n, k}.
\end{equation*}

By repeatedly applying \eqref{eq:uI:TjMnk}, we easily see that
\begin{align}
    T_w \M_{n,k} = q^{\ell\tup{w}} \M_{n,k}
    \label{eq:uI:TwMnk}
\end{align}
for any $w \in \symm_{\ive{n}-\ive{n-k}}$ (since $w$ has a reduced expression containing only $s_j$ with $n-k < j < n$).
Since $\M_{n,k} = \sum_{w \in \symm_{\ive{n}-\ive{n-k}}} T_w$, we now have
\begin{align}
    \M_{n,k} \M_{n,k}
    &= \sum_{w \in \symm_{\ive{n}-\ive{n-k}}} T_w \M_{n,k} \nonumber\\
    &= \sum_{w \in \symm_{\ive{n}-\ive{n-k}}} q^{\ell\tup{w}} \M_{n,k}
    \qquad \left(\text{by \eqref{eq:uI:TwMnk}}\right)
    \nonumber\\
    &= \tup{\sum_{w \in \symm_{\ive{n}-\ive{n-k}}} q^{\ell\tup{w}}} \M_{n,k}.
    \label{pf:eq:uI:Mnk2:1}
\end{align}
However, the subgroup $\symm_{\ive{n}-\ive{n-k}}$ of $\symm_n$ is isomorphic to the symmetric group $\symm_k$ (via the isomorphism that sends each generator $s_i$ to $s_{n-i}$), and this isomorphism preserves the length of a permutation.
Hence,
\begin{align}
\sum_{w \in \symm_{\ive{n}-\ive{n-k}}} q^{\ell\tup{w}} = \sum_{w \in \symm_k} q^{\ell\tup{w}} = \ive{k}!_q
\label{pf:eq:ui:Mnk2:lensum}
\end{align}
by a well-known identity (\cite[Theorem 3.2.1]{Sagan2020}).
Thus, \eqref{pf:eq:uI:Mnk2:1} rewrites as
$\M_{n,k} \M_{n,k} = \ive{k}!_q \M_{n,k}$.
In other words, $\M_{n,k}^2 = \ive{k}!_q \M_{n,k}$.
This proves \eqref{eq:uI:Mnk2}.
\qedhere

\end{proof}

\begin{proof}[Proof of Lemma~\ref{TiJi}.]
The definition of $J_n$ yields
\begin{align}
J_n
& =\sum_{i=1}^{n-1}q^{i-n}T_{\left(  i,n\right)  }
=\underbrace{q^{\left(  n-1\right)  -n}}_{=q^{-1}}\underbrace{T_{\left(
n-1,n\right)  }}_{=T_{n-1}}
+\sum_{i=1}^{n-2}\underbrace{q^{i-n}}_{=q^{-1} q^{i-\tup{n-1}}}
T_{\left(  i,n\right)}\nonumber\\
& =q^{-1}T_{n-1}
+q^{-1}\sum_{i=1}^{n-2}q^{i-\tup{n-1}}T_{\left(  i,n\right)  }
.\label{pf.TiJi:0}
\end{align}

Let $i\in \ive{n-2}$. It is known that $s_{n-1}s_{n-2}\cdots
s_i\cdots s_{n-2}s_{n-1}$ (with the subscripts decreasing from $n-1$ to $i$
and then rising again to $n-1$) is a reduced expression for the transposition
$\left(  i,n\right)  $. Hence,
\begin{equation}
T_{\left(  i,n\right)  }
= T_{n-1}T_{n-2}\cdots T_i\cdots T_{n-2}T_{n-1}.
\label{pf.TiJi:1}
\end{equation}
Similarly,
\begin{equation}
T_{\left(  i,n-1\right)  }
= T_{n-2}T_{n-3}\cdots T_i\cdots T_{n-3}T_{n-2}.
\label{pf.TiJi:2}
\end{equation}
Now, \eqref{pf.TiJi:1} can be rewritten as
\[
T_{\left(  i,n\right)  }
=T_{n-1}\underbrace{T_{n-2}T_{n-3}\cdots
T_i\cdots T_{n-3}T_{n-2}}_{\substack{=T_{\left(  i,n-1\right)  }\\\text{(by
\eqref{pf.TiJi:2})}}}T_{n-1}
=T_{n-1}T_{\left(  i,n-1\right)  }T_{n-1}.
\]

We have proved this equality for all $i\in \ive{n-2}$. Multiplying
it by $q^{i-\tup{n-1}}$ and summing over all $i\in \ive{n-2}$, we obtain
\begin{align*}
\sum_{i=1}^{n-2}q^{i-\tup{n-1}}T_{\left(  i,n\right)  }  & =\sum_{i=1}^{n-2}
q^{i-\left(  n-1\right)  }T_{n-1}T_{\left(
i,n-1\right)  }T_{n-1}
 =T_{n-1}\underbrace{\sum_{i=1}^{n-2}q^{i-\left(  n-1\right)
}T_{\left(  i,n-1\right)  }}_{\substack{=J_{n-1}\\\text{(by the definition of
}J_{n-1}\text{)}}}T_{n-1}
 =T_{n-1}J_{n-1}T_{n-1}.
\end{align*}
Thus, we can rewrite \eqref{pf.TiJi:0} as
\[
J_n = q^{-1}T_{n-1}+q^{-1}T_{n-1}J_{n-1}T_{n-1}
= q^{-1}\left(  1+T_{n-1} J_{n-1}\right)  T_{n-1}.
\]
Multiplying this equality by $T_{n-1}-q+1$ from the right, we find
\begin{align*}
J_n \left(  T_{n-1}-q+1\right)    & =q^{-1}\left(  1+T_{n-1}J_{n-1}\right)
\underbrace{T_{n-1}\left(  T_{n-1}-q+1\right)  }_{\substack{=T_{n-1}^{2}
-\left(  q-1\right)  T_{n-1}=q\\\text{(since \eqref{eq:Hndef:1} yields
}T_{n-1}^{2}=\left(  q-1\right)  T_{n-1}+q\text{)}}}\\
& =q^{-1}\left(  1+T_{n-1}J_{n-1}\right)  q
= 1+T_{n-1}J_{n-1}.
\end{align*}
Thus, $T_{n-1}J_{n-1} = J_n \left(  T_{n-1}-q+1\right)  -1$.
This proves Lemma~\ref{TiJi}.
\end{proof}

For our next proof, we need a minor lemma:

\begin{lem}
\label{lem:commute1}
Let $n,i>0$ and $a \in \HH_n$. Then, $T_i$
commutes with $T_iaT_i + qa$.
\end{lem}

\begin{proof}
We have
\[
T_i\left(  T_iaT_i+qa\right)  =\underbrace{T_i^{2}}%
_{\substack{=\left(  q-1\right)  T_i+q\\\text{(by
\eqref{eq:Hndef:1})}}}aT_i+qT_ia
= \left(  q-1\right)  T_iaT_i + qaT_i + qT_ia.
\]
The same computation (just with the order of factors reversed) yields
\[
\left(  T_iaT_i+qa\right)  T_i
= \left(  q-1\right)  T_iaT_i + qT_ia + qaT_i.
\]
The right-hand sides of these two equalities are clearly equal. Thus, so are
the left-hand sides. In other words, $T_i\left(  T_iaT_i+qa\right)
=\left(  T_iaT_i+qa\right)  T_i$. This proves Lemma~\ref{lem:commute1}.
\end{proof}

\begin{proof}[Proof of Lemma~\ref{lem:TiJk}.]
%
We must prove that $T_i$ commutes
with $J_n$. If $i>n$, then $T_i$ commutes with all the generators
$T_{1},T_{2},\ldots,T_{n-1}$ of $\HH_n$ (by \eqref{eq:Hndef:2}), and
thus commutes with all elements of $\HH_n$, including $J_n$.
Hence, without loss of generality, we assume that $i\leq n$.
Since $i\notin\left\{  n-1,n\right\}  $,
this entails $i\leq n-2$. The definition of $J_n$ yields
\begin{align*}
J_n  & =\sum_{j=1}^{n-1}q^{j-n}T_{\left(  j,n\right)  }=\sum_{j\in
\ive{n-1} \setminus\left\{  i,i+1\right\}  }q^{j-n}T_{\left(  j,n\right)
}+\underbrace{q^{i-n}T_{\left(  i,n\right)  }+q^{\left(  i+1\right)
-n}T_{\left(  i+1,n\right)  }}_{=q^{i-n}\left(  T_{\left(  i,n\right)
}+qT_{\left(  i+1,n\right)  }\right)  }\\
& =\sum_{j\in \ive{n-1}  \setminus\left\{  i,i+1\right\}  }
q^{j-n}T_{\left(  j,n\right)  }+q^{i-n}\left(  T_{\left(  i,n\right)
}+qT_{\left(  i+1,n\right)  }\right)  .
\end{align*}
Hence, in order to show that $T_i$ commutes with $J_n$, it will suffice to
show that
\begin{equation}
T_i\text{ commutes with }T_{\left(  j,n\right)  }\text{ for each }
j\in \ive{n-1}  \setminus\left\{  i,i+1\right\}
\label{pf.lem:TiJk:goal1}
\end{equation}
and that
\begin{equation}
T_i\text{ commutes with }T_{\left(  i,n\right)  }+qT_{\left(  i+1,n\right)
}.
\label{pf.lem:TiJk:goal2}
\end{equation}

Let us first verify \eqref{pf.lem:TiJk:goal1}. Indeed, let $j\in
\ive{n-1}  \setminus\left\{  i,i+1\right\}  $. Then, the transpositions
$s_i=\left(  i,i+1\right)  $ and $\left(  j,n\right)  $ are disjoint (since
$i\leq n-2$) and thus commute: i.e., we have $s_i\left(  j,n\right)
=\left(  j,n\right)  s_i$. Hence, $T_{s_i\left(  j,n\right)  }=T_{\left(
j,n\right)  s_i}$. Moreover, $\left(  j,n\right)  \left(  i\right)  <\left(
j,n\right)  \left(  i+1\right)  $ (since this is just saying $i<i+1$). Hence,
$T_{\left(  j,n\right)  s_i}=T_{\left(  j,n\right)  }T_i$ (since
$T_{ws_i}=T_{w}T_i$ for any $w\in\mathfrak{S}_n$ satisfying $w\left(
i\right)  <w\left(  i+1\right)  $). Furthermore, $\left(  j,n\right)
^{-1}\left(  i\right)  <\left(  j,n\right)  ^{-1}\left(  i+1\right)  $ (again
since $i<i+1$), so that $T_{s_i\left(  j,n\right)  }=T_iT_{\left(
j,n\right)  }$  (since $T_{s_iw}=T_iT_{w}$ for any $w\in\mathfrak{S}_n$
satisfying $w^{-1}\left(  i\right)  <w^{-1}\left(  i+1\right)  $). Thus,
$T_iT_{\left(  j,n\right)  }=T_{s_i\left(  j,n\right)  }=T_{\left(
j,n\right)  s_i}=T_{\left(  j,n\right)  }T_i$. In other words, $T_i$
commutes with $T_{\left(  j,n\right)  }$. This proves \eqref{pf.lem:TiJk:goal1}.

It remains to verify \eqref{pf.lem:TiJk:goal2}.
Here, we observe that
$s_is_{i+1}\cdots s_{n-1}\cdots s_{i+1}s_i$
(where the subscripts rise from $i$ to $n-1$ and then decrease back to $i$)
is a reduced expression for the transposition $\tup{i, n}$.
Thus,
\begin{equation}
T_{\left(  i,n\right)  }
= T_iT_{i+1}\cdots T_{n-1}\cdots T_{i+1}T_i.
\label{pf.lem:TiJk:4}
\end{equation}
Similarly,
\begin{equation}
T_{\left(  i+1,n\right)  }
= T_{i+1}T_{i+2}\cdots T_{n-1}\cdots T_{i+2}T_{i+1}.
\label{pf.lem:TiJk:6}
\end{equation}
Now, we can rewrite \eqref{pf.lem:TiJk:4} as
\[
T_{\left(  i,n\right)  }
= T_i\underbrace{T_{i+1}T_{i+2}\cdots T_{n-1}\cdots
T_{i+2}T_{i+1}}_{\substack{=T_{\left(  i+1,n\right)  }\\\text{(by
\eqref{pf.lem:TiJk:6})}}}T_i
=T_iT_{\left(  i+1,n\right)  }T_i.
\]
Hence,
\begin{equation}
T_{\left(  i,n\right)  }+qT_{\left(  i+1,n\right)  }=T_iT_{\left(
i+1,n\right)  }T_i+qT_{\left(  i+1,n\right)  }.
\label{pf.lem:TiJk:7}
\end{equation}
But Lemma~\ref{lem:commute1} (applied to $a=T_{\left(  i+1,n\right)  }$) shows
that $T_i$ commutes with $T_iT_{\left(  i+1,n\right)  }T_i+qT_{\left(
i+1,n\right)  }$. In other words, $T_i$ commutes with $T_{\left(  i,n\right)
}+qT_{\left(  i+1,n\right)  }$ (by
\eqref{pf.lem:TiJk:7}). Thus, \eqref{pf.lem:TiJk:goal2} is proved, and we are
done proving Lemma~\ref{lem:TiJk}.
\end{proof}

\begin{proof}[Proof of Proposition~\ref{prop:eigenspace-dirsum}.]
(a) We assume that $S \neq \varnothing$, since otherwise the claim is just
saying that $0 = 0$.

The sum
$\bigoplus_{\eta\in S}\Eig_{\eta}(f)$ is indeed direct, since any list of
nonzero eigenvectors of $f$ corresponding to distinct eigenvalues is $\kk%
$-linearly independent (see, e.g., \cite[\S 8.12]{Axler}). Moreover, each
$\eta\in S$ satisfies $\Eig_{\eta}(f)=\ker\left(  f-\eta\id
\right)  $ by the definition of an eigenspace. Hence, it is easy to see that
\begin{equation}
\bigoplus_{\eta\in S}\Eig_{\eta}(f)\subseteq\ker\left(  \prod_{\eta\in
S}\left(  f-\eta\id\right)  \right)
\label{pf.prop:eigenspace-dirsum.2}
\end{equation}
(because a vector $v$ in any of the eigenspaces $\Eig_{\eta}(f)$ will be
annihilated by the operator $f-\eta\id$
for this $\eta$, and thus will clearly be annihilated by the product
$\prod_{\eta\in S}\left(  f-\eta\id\right)  $ of all these
operators). It remains to prove the converse inclusion.

The set $S$ is finite. For each $\eta\in S$, we define the polynomial%
\[
g_{\eta}:=\prod_{\xi\in S\setminus\left\{  \eta\right\}  }
\dfrac{x-\xi}{\eta-\xi} \in \kk\left[  x\right]  .
\]
This is a polynomial of degree $<\left\vert S\right\vert $ and satisfies
$g_{\eta}\left(  \xi\right)  =\delta_{\eta,\xi}$ for all $\xi\in S$ (where
$\delta$ denotes the Kronecker delta). Thus, the sum $\sum_{\eta\in S}g_{\eta
}$ of all these polynomials is a polynomial of degree $<\left\vert
S\right\vert $ that sends all the elements $\xi\in S$ to $\sum_{\eta\in
S}g_{\eta}\left(  \xi\right)  =\sum_{\eta\in S}\delta_{\eta,\xi}=1$; hence, it
must be the constant polynomial $1$ (since the polynomial $\sum_{\eta\in
S}g_{\eta}-1$ has degree $<\left\vert S\right\vert $ but at least $\left\vert
S\right\vert $ many roots, because each $\xi \in S$ is a root of it).
In other words, $\sum_{\eta\in S}g_{\eta}=1$, so
that $\sum_{\eta\in S}g_{\eta}\left(  f\right)  =\id$, and
therefore each $v\in V$ satisfies%
\begin{equation}
\sum_{\eta\in S}\left(  g_{\eta}\left(  f\right)  \right)  \left(  v\right)
=v.
\label{pf.prop:eigenspace-dirsum.5}
\end{equation}

Now, for each $v\in\ker\left(  \prod_{\xi\in S}\left(  f-\xi\id
\right)  \right)  $ and each $\eta\in S$, we have%
\[
g_{\eta}\left(  f\right)  =\prod_{\xi\in S\setminus\left\{  \eta\right\}
}\dfrac{f-\xi\id}{\eta-\xi}\qquad\left(  \text{by the
definition of }g_{\eta}\right)
\]
and thus%
\begin{align*}
\left(  f-\eta\id\right)  \left(  \left(  g_{\eta}\left(
f\right)  \right)  \left(  v\right)  \right)   & =\left(  f-\eta
\id\right)  \left(  \prod_{\xi\in S\setminus\left\{
\eta\right\}  }\dfrac{f-\xi\id}{\eta-\xi}\right)  \left(
v\right) \\
& =\dfrac{\left(  f-\eta\id\right)  \prod_{\xi\in
S\setminus\left\{  \eta\right\}  }\left(  f-\xi\id\right)
}{\prod_{\xi\in S\setminus\left\{  \eta\right\}  }\left(  \eta-\xi\right)
}\left(  v\right) \\
& =\dfrac{1}{\prod_{\xi\in S\setminus\left\{  \eta\right\}  }\left(  \eta
-\xi\right)  }\underbrace{\left(  \left(  f-\eta\id\right)
\prod_{\xi\in S\setminus\left\{  \eta\right\}  }\left(  f-\xi
\id\right)  \right)  \left(  v\right)  }_{\substack{=\left(
\prod_{\xi\in S}\left(  f-\xi\id\right)  \right)  \left(
v\right)  \\=0 \\\text{(since }v\in\ker\left(  \prod_{\xi\in S}\left(
f-\xi\id\right)  \right)  \text{)}}}\\
& =0,
\end{align*}
so that
\[
\left(  g_{\eta}\left(  f\right)  \right)  \left(  v\right)  \in
\operatorname*{Ker}\left(  f-\eta\id\right)  =\Eig_{\eta}(f).
\]
Hence, for each $v\in\ker\left(  \prod_{\xi\in S}\left(  f-\xi
\id\right)  \right)  $, we have%
\begin{align*}
v  & =\sum_{\eta\in S}\underbrace{\left(  g_{\eta}\left(  f\right)  \right)
\left(  v\right)  }_{\in\Eig_{\eta}(f)}\qquad\left(  \text{by
\eqref{pf.prop:eigenspace-dirsum.5}}\right) \\
& \in\sum_{\eta\in S}\Eig_{\eta}(f)=\bigoplus_{\eta\in S}\Eig_{\eta}(f).
\end{align*}
This proves that $\ker\left(  \prod_{\xi\in S}\left(  f-\xi\id
\right)  \right)  \subseteq\bigoplus_{\eta\in S}\Eig_{\eta}(f)$. Renaming the
index $\xi$ as $\eta$ and combining this with
\eqref{pf.prop:eigenspace-dirsum.2}, we obtain $\ker\left(  \prod_{\eta\in
S}\left(  f-\eta\id\right)  \right)  =\bigoplus_{\eta\in
S}\Eig_{\eta}(f)$, thus proving Proposition~\ref{prop:eigenspace-dirsum}(a).

(b) Part (a) yields
\begin{equation}
\ker\left(  \prod_{\eta\in S}\left(  f-\eta\id\right)  \right)
=\bigoplus_{\eta\in S}\Eig_{\eta}(f).
\label{pf.cor:diagableS.1}
\end{equation}
However, from $\prod_{\eta\in S}\left(  f-\eta\id\right)  =0$,
we obtain $\ker\left(  \prod_{\eta\in S}\left(  f-\eta\id
\right)  \right)  =\ker0=V$. Hence, \eqref{pf.cor:diagableS.1} rewrites as
\[
V=\bigoplus_{\eta\in S}\Eig_{\eta}(f).
\]
This shows that $f$ is diagonalizable, and its eigenvalues all belong to $S$.
\end{proof}

For the next two proofs, we need a further notation regarding Young tableaux:

\begin{itemize}

	\item If $\frakt$ is a standard tableau of shape $\lambda \vdash n$, and if $m \in \set{0,1,\ldots,n}$, then the \emph{restriction $\frakt \vert_m$} of $\frakt$ to $m$ is the tableau defined by removing all entries larger than $m$ from $\frakt$ (along with the cells that they occupy). This is again a standard tableau, but now of a shape $\nu$ for some partition $\nu \vdash m$. We write $\sh(\frakt \vert_m) = \nu$ in this case.
    
    Likewise we define $\frakt \vert_m$ when $\frakt$ is a standard tableau of shape $\lambda \sm \mu$ (as long as $m \geq \abs{\mu}$); this will then be a standard tableau of shape $\nu \sm \mu$.
    
\end{itemize}

\begin{proof}[Proof of Lemma~\ref{lem:Slam-glue}.]
We shall use the \emph{Young seminormal units} $w_{\frakt}\in
W^{\lambda}$ defined for all standard tableaux $\frakt\in\syt
\left(  \lambda\right)  $; see \cite[Definition 2.21]{r2r1}. By
\cite[Theorem 2.22 (1)]{r2r1}, these units form a basis of $S^{\lambda}$.
Applying this to $\lambda^{\prime}$ instead of $\lambda$, we see that the
analogous units $w_{\mathfrak{s}}$ for $\mathfrak{s}\in\syt\left(
\lambda^{\prime}\right)  $ form a basis of $S^{\lambda^{\prime}}$. Since the
claim we want to prove is linear in $v$, we can therefore, without loss of generality, assume that
$v=w_{\mathfrak{s}}$ for some tableau $\mathfrak{s}\in\syt\left(
\lambda^{\prime}\right)  $. 

We shall also use the \emph{Young idempotents} (also known as the
\emph{seminormal idempotents}) $p_{\frakt}\in\mathcal{H}_{n}$ defined
for all standard tableaux $\frakt\in\syt\left(  \lambda\right)
$; see \cite[Section 2.2.1]{r2r1}. From \cite[Proposition 2.11 (2)]{r2r1}, we
know that they satisfy%
\begin{equation}
p_{\frakt}\,J_{n}=\left[  \operatorname*{cont}\nolimits_{\frakt%
}\left(  n\right)  \right]  _{q}\,p_{\frakt},\label{pf.lem:Slam-glue.1}%
\end{equation}
where $\operatorname*{cont}\nolimits_{\frakt}\left(  n\right)  $ denotes
the content of the box of $\frakt$ that contains the entry $n$. Also,
\cite[Definition 2.21]{r2r1} yields%
\begin{equation}
w_{\frakt}=\operatorname*{word}\left(  \frakt\right)  \cdot
p_{\frakt}.\label{pf.lem:Slam-glue.ws}%
\end{equation}
Hence, for each $\frakt\in\syt\left(  \lambda\right)  $, we
have
\begin{align}
w_{\frakt}\,J_{n}  & =\operatorname*{word}\left(  \frakt\right)
\cdot p_{\frakt}\,J_{n}=\left[  \operatorname*{cont}%
\nolimits_{\frakt}\left(  n\right)  \right]  _{q}\,\operatorname*{word}%
\left(  \frakt\right)  \cdot p_{\frakt}\ \ \ \ \ \ \ \ \ \ \left(
\text{by
\eqref{pf.lem:Slam-glue.1}}\right)  \nonumber\\
& =\left[  \operatorname*{cont}\nolimits_{\frakt}\left(  n\right)
\right]  _{q}\,w_{\frakt}\ \ \ \ \ \ \ \ \ \ \left(  \text{by
\eqref{pf.lem:Slam-glue.ws}}\right)  .\label{pf.lem:Slam-glue.3}%
\end{align}

Now, let $\frakt\in\syt\left(  \lambda\right)  $ be the tableau
of shape $\lambda$ that agrees with $\mathfrak{s}$ on all boxes of
$\lambda^{\prime}$ while having an entry $n$ in the box of $\lambda
\sm\lambda^{\prime}$. Then, $\operatorname*{cont}\nolimits_{\frakt%
}\left(  n\right)  =\content\left(  \lambda\sm\lambda^{\prime}\right)
$. Moreover, $\mathfrak{s}=\frakt\mid_{n-1}$, so that \cite[Proposition
3.8]{r2r1} (applied to $\mu=\lambda^{\prime}$ and $\frakt^{\prime
}=\mathfrak{s}$) yields
\begin{equation}
w_{\frakt}=w_{\mathfrak{s}}\,\Phi_{\lambda\sm\sh
\left(  \frakt^{\prime}\right)  }\,p_{\lambda}=w_{\mathfrak{s}}%
\,\Phi_{r}\,p_{\lambda}\label{pf.lem:Slam-glue.2}%
\end{equation}
(since $\sh\left(  \frakt^{\prime}\right)  =\lambda
^{\prime}$ and thus $\Phi_{\lambda\sm\sh\left(
\frakt^{\prime}\right)  }=\Phi_{\lambda\sm\lambda^{\prime}}%
=\Phi_{r}$). Since $v=w_{\mathfrak{s}}$, we can rewrite this as%
\[
w_{\frakt}=v\,\Phi_{r}\,p_{\lambda}.
\]
Using this and the fact that $\operatorname*{cont}\nolimits_{\frakt%
}\left(  n\right)  =\content\left(  \lambda\sm\lambda^{\prime}\right)
$, we can rewrite
\eqref{pf.lem:Slam-glue.3} as%
\[
v\,\Phi_{r}\,p_{\lambda}\,J_{n}=\left[  \content\left(  \lambda
\sm\lambda^{\prime}\right)  \right]  _{q}\,v\,\Phi_{r}\,p_{\lambda},
\]
which is precisely what Lemma~\ref{lem:Slam-glue} claims.
\end{proof}

\begin{proof}[Proof of Lemma~\ref{lem:y-rec}.]
We will use the standard tableau
$\frakt^{\lambda\sm\mu}$ defined in \cite[\S 2.2.2]{r2r1}, and the
corresponding (non-central) element $p_{\frakt^{\lambda\sm\mu}}$
defined in \cite[Definition 2.14]{r2r1}. Note that the tableau
$\frakt^{\lambda\sm\mu}$ has its entry $n$ precisely in the box of
$\lambda\sm\lambda^{\prime}$; thus,
$\left.\frakt^{\lambda\sm\mu}\mid_{n-1}\right.
=\frakt^{\lambda^{\prime}\sm\mu}$. Hence,
\cite[Definition 2.14]{r2r1} yields
\begin{equation}
p_{\frakt^{\lambda\sm\mu}}
= p_{\frakt^{\lambda^{\prime}\sm\mu}}
\,p_{\lambda}.
\label{pf.lem:y-rec.1}
\end{equation}
Furthermore, $\Phi_{\lambda\sm\mu}=\Phi_{\lambda^{\prime}\sm\mu
}\,\Phi_{\lambda\sm\lambda^{\prime}}$ (keeping in mind that these maps
act from the right), again because the box of $\lambda\sm\lambda
^{\prime}$ is the last box of $\lambda\sm\mu$ in the row-by-row,
top-to-bottom listing. But \cite[Lemma 6.12]{r2r1} yields
\begin{equation}
u\,\Phi_{\lambda\sm\mu}\,p_{\frakt^{\lambda\sm\mu}}%
=u\,\Phi_{\lambda\sm\mu}\,p_{\lambda}
\label{pf.lem:y-rec.2}
\end{equation}
(since our $\Phi_{\lambda\sm\mu}$ is the
$\Phi_{\frakt^{\lambda\sm\mu}}$ of \cite{r2r1}). But the
definition of $y_{\mu\left( u\right)  }^{\lambda}$ yields
\begin{align}
y_{\mu\left(  u\right)  }^{\lambda}  & =u\,\Phi_{\lambda\sm\mu
}\,\mathcal{B}_{n,n-\left\vert \mu\right\vert }\,p_{\lambda}=u\,\Phi
_{\lambda\sm\mu}\,p_{\lambda}\,\mathcal{B}_{n,n-\left\vert
\mu\right\vert }\ \ \ \ \ \ \ \ \ \ \left(  \text{since }p_{\lambda}\text{ is
central}\right) \nonumber\\
& =u\,\Phi_{\lambda\sm\mu}\,p_{\frakt^{\lambda\sm\mu}%
}\,\mathcal{B}_{n,n-\left\vert \mu\right\vert }\ \ \ \ \ \ \ \ \ \ \left(
\text{by \eqref{pf.lem:y-rec.2}}\right)  .
\label{pf.lem:y-rec.3}
\end{align}
The same argument for $\lambda^{\prime}$ instead of $\lambda$ yields%
\[
y_{\mu\left(  u\right)  }^{\lambda^{\prime}}=u\,\Phi_{\lambda^{\prime
}\sm\mu}\,p_{\frakt^{\lambda^{\prime}\sm\mu}}\,\mathcal{B}%
_{n-1,n-1-\left\vert \mu\right\vert }.
\]
The latter equality yields
\begin{align*}
y_{\mu\left(  u\right)  }^{\lambda^{\prime}}\,\Phi_{\lambda\sm
\lambda^{\prime}}\,\mathcal{B}_{n}\,p_{\lambda}
& =u\,\Phi_{\lambda^{\prime
}\sm\mu}\,p_{\frakt^{\lambda^{\prime}\sm\mu}}\,\mathcal{B}%
_{n-1,n-1-\left\vert \mu\right\vert }\,\Phi_{\lambda\sm\lambda^{\prime}%
}\,\mathcal{B}_{n}\,p_{\lambda}\\
& =u\,\underbrace{\Phi_{\lambda^{\prime}\sm\mu}\,\Phi_{\lambda
\sm\lambda^{\prime}}}_{=\Phi_{\lambda\sm\mu}}\,
p_{\frakt^{\lambda^{\prime}\sm\mu}}\,
\underbrace{\mathcal{B}_{n-1,n-1-\left\vert
\mu\right\vert }\,\mathcal{B}_{n}\,p_{\lambda}}_{\substack{=p_{\lambda}\,
\mathcal{B}_{n-1,n-1-\left\vert \mu\right\vert}\,\mathcal{B}_{n} \\
\text{(since }p_{\lambda}\text{ is central)}}}
\ \ \ \ \ \ \ \ \ \ \left(  \text{by \eqref{eq:Phi-commut}}\right) \\
& =u\,\Phi_{\lambda\sm\mu}\,\underbrace{p_{\frakt^{\lambda
^{\prime}\sm\mu}}\,p_{\lambda}}_{\substack{
=p_{\frakt^{\lambda\sm\mu}} \\
\text{(by \eqref{pf.lem:y-rec.1})}}%
}\,\underbrace{\mathcal{B}_{n-1,n-1-\left\vert \mu\right\vert }\,
\mathcal{B}_{n}}_{=\mathcal{B}_{n,n-\left\vert \mu\right\vert }}
= u\,\Phi_{\lambda \sm\mu}\,p_{\frakt^{\lambda\sm\mu}}\,
\mathcal{B}_{n,n-\left\vert \mu\right\vert }\\
& =y_{\mu\left(  u\right)  }^{\lambda}\ \ \ \ \ \ \ \ \ \ \left(  \text{by
\eqref{pf.lem:y-rec.3}}\right)  .
\end{align*}
This proves Lemma~\ref{lem:y-rec}.
\end{proof}

\begin{proof}
[Proof of Proposition~\ref{prop:supersymrecursion}.]Recall two fundamental
recurrences for symmetric polynomials: For any $i\in\mathbb{N}$ and $j>0$, and
for any $j$ elements $z_{1},z_{2},\ldots,z_{j}$ of $A$, we have
\begin{equation}
e_{i}\left(  z_{1},z_{2},\ldots,z_{j}\right)  =e_{i}\left(  z_{1},z_{2}
,\ldots,z_{j-1}\right)  +z_{j}e_{i-1}\left(  z_{1},z_{2},\ldots,z_{j-1}
\right)
\label{pf.lem.super-id.erec}
\end{equation}
and
\begin{equation}
h_{i}\left(  z_{1},z_{2},\ldots,z_{j}\right)  =h_{i}\left(  z_{1},z_{2}
,\ldots,z_{j-1}\right)  +z_{j}h_{i-1}\left(  z_{1},z_{2},\ldots,z_{j}\right)
. \label{pf.lem.super-id.hrec}
\end{equation}
These follow easily from the definitions of $e_{i}$ and $h_{i}$.

Now, Definition~\ref{def:supersym} yields
\begin{align*}
e_{a}\left(  X_{c}/Y_{d}\right)   &  =\sum_{j=0}^{a}\underbrace{e_{j}\left(
X_{c}\right)  }_{\substack{=e_{j}\left(  X_{c-1}\right)  +x_{c}e_{j-1}\left(
X_{c-1}\right)  \\\text{(by \eqref{pf.lem.super-id.erec})}}}\cdot
\,h_{a-j}\left(  Y_{d}\right) \\
&  =\sum_{j=0}^{a}e_{j}\left(  X_{c-1}\right)  \cdot\underbrace{h_{a-j}\left(
Y_{d}\right)  }_{\substack{=h_{a-j}\left(  Y_{d-1}\right)  +y_{d}
h_{a-j-1}\left(  Y_{d}\right)  \\\text{(by \eqref{pf.lem.super-id.hrec})}
}}+\,x_{c}\underbrace{\sum_{j=0}^{a}e_{j-1}\left(  X_{c-1}\right)  \cdot
h_{a-j}\left(  Y_{d}\right)  }_{\substack{=\sum_{j=1}^{a}e_{j-1}\left(
X_{c-1}\right)  \cdot h_{a-j}\left(  Y_{d}\right)  \\\text{(here, we have
removed the addend for }j=0\text{,}\\\text{since this addend is }0\text{
(because }e_{0-1}=e_{-1}=0\text{))}}}\\
&  =\sum_{j=0}^{a}e_{j}\left(  X_{c-1}\right)  \cdot h_{a-j}\left(
Y_{d-1}\right)  +y_{d}\underbrace{\sum_{j=0}^{a}e_{j}\left(  X_{c-1}\right)
\cdot h_{a-j-1}\left(  Y_{d}\right)  }_{\substack{=\sum_{j=0}^{a-1}
e_{j}\left(  X_{c-1}\right)  \cdot h_{a-j-1}\left(  Y_{d}\right)
\\\text{(since the addend for }j=a\text{ is }0\text{)}}}\\
&  \ \ \ \ \ \ \ \ \ \ +x_{c}\underbrace{\sum_{j=1}^{a}e_{j-1}\left(
X_{c-1}\right)  \cdot h_{a-j}\left(  Y_{d}\right)  }_{\substack{=\sum
_{j=0}^{a-1}e_{j}\left(  X_{c-1}\right)  \cdot h_{a-j-1}\left(  Y_{d}\right)
\\\text{(here, we have substituted }j+1\text{ for }j\text{ in the sum)}}}\\
&  =\underbrace{\sum_{j=0}^{a}e_{j}\left(  X_{c-1}\right)  \cdot
h_{a-j}\left(  Y_{d-1}\right)  }_{\substack{=e_{a}\left(  X_{c-1}
/Y_{d-1}\right)  \\\text{(by Definition~\ref{def:supersym})}}}+\left(
x_{c}+y_{d}\right)  \underbrace{\sum_{j=0}^{a-1}e_{j}\left(  X_{c-1}\right)
\cdot h_{a-j-1}\left(  Y_{d}\right)  }_{\substack{=e_{a-1}\left(
X_{c-1}/Y_{d}\right)  \\\text{(by Definition~\ref{def:supersym})}}}\\
&  =e_{a}\left(  X_{c-1}/Y_{d-1}\right)  +\left(  x_{c}+y_{d}\right)
e_{a-1}\left(  X_{c-1}/Y_{d}\right)  .
\end{align*}
This proves Proposition~\ref{prop:supersymrecursion}.
\end{proof}

\begin{proof}
[Proof of Proposition~\ref{prop:superproduct}.]For $0\leq r\leq c$ and $0\leq
s\leq d$, write $X_{r}=(x_{1},\ldots,x_{r})$ and $Y_{s}=(y_{1},\ldots,y_{s})$.
Thus, $\mathsf{X}=X_{c}$ and $\mathsf{Y}=Y_{d}$, so we must prove that
\begin{equation}
e_{a}(X_{c}/Y_{d})=\sum_{1\leq j_{1}<j_{2}<\cdots<j_{a}\leq c}\ \ \prod
_{m=1}^{a}\left(  x_{j_{m}}+y_{j_{m}-m+1}\right)  .
\label{pf:prop:superproduct:goal1}
\end{equation}

To prove \eqref{pf:prop:superproduct:goal1}, we induct on $c$, simultaneously
for all $0\leq a\leq c$ and all lists of the required lengths. We first
observe that the $a=0$ case boils down to $1=1$, which is trivial. Thus, the
base case ($c=0$) of our induction is true, since $0\leq a\leq c$ then leads
to $a=0$.

Now, we fix $c\geq1$ and step from $c-1$ to $c$. Without loss of generality, assume that $a\geq1$,
since $a=0$ has already been handled. If $a<c$, then the induction hypothesis
(applied with $d-1=(c-1)-a+1$) yields
\begin{equation}
e_{a}(X_{c-1}/Y_{d-1})=\sum_{1\leq j_{1}<j_{2}<\cdots<j_{a}\leq c-1}
\ \ \prod_{m=1}^{a}\left(  x_{j_{m}}+y_{j_{m}-m+1}\right)
\label{pf:prop:superproduct:3}
\end{equation}
If $a=c$, then \eqref{pf:prop:superproduct:3} still holds: $d-1=0$, so its
left-hand side is $e_{c}(X_{c-1}/\varnothing)=e_{c}(X_{c-1})=0$ (since
$e_{c}\left(  z_{1},z_{2},\ldots,z_{k}\right)  =0$ whenever $k<c$), and its
right-hand side is an empty sum.

Furthermore, the induction hypothesis (applied with $d=(c-1)-(a-1)+1$) yields
\begin{align}
e_{a-1}(X_{c-1}/Y_{d}) &  =\sum_{1\leq j_{1}<j_{2}<\cdots<j_{a-1}\leq
c-1}\ \ \prod_{m=1}^{a-1}\left(  x_{j_{m}}+y_{j_{m}-m+1}\right)  \nonumber\\
&  =\sum_{1\leq j_{1}<j_{2}<\cdots<j_{a-1}<j_{a}=c}
\ \ \prod_{m=1}^{a-1}\left(  x_{j_{m}}+y_{j_{m}-m+1}\right)
\label{pf:prop:superproduct:4}
\end{align}
(here we have extended the summation index $\left(  j_{1}<j_{2}<\cdots
<j_{a-1}\right)  $ to $\left(  j_{1}<j_{2}<\cdots<j_{a}\right)  $ by setting
$j_{a}:=c$). For $a=1$, the first sum has one empty tuple and equals $1$; the
second sum has the single tuple $j_{1}=c$ and also equals $1$.

But Proposition~\ref{prop:supersymrecursion} yields
\begin{align*}
& e_{a}(X_{c}/Y_{d}) 
\\
&  =e_{a}(X_{c-1}/Y_{d-1})+\left(  x_{c}+y_{d}\right)
e_{a-1}(X_{c-1}/Y_{d})\\
&  =\sum_{1\leq j_{1}<j_{2}<\cdots<j_{a}\leq c-1}\ \ \prod_{m=1}^{a}\left(
x_{j_{m}}+y_{j_{m}-m+1}\right)  +\left(  x_{c}+y_{d}\right)  \sum_{1\leq
j_{1}<j_{2}<\cdots<j_{a-1}<j_{a}=c}\ \ \prod_{m=1}^{a-1}\left(  x_{j_{m}
}+y_{j_{m}-m+1}\right) \\
&  \ \ \ \ \ \ \ \ \ \ \ \ \ \ \ \ \ \ \ \ \left(  \text{by
\eqref{pf:prop:superproduct:3} and \eqref{pf:prop:superproduct:4}}\right) \\
&  =\sum_{1\leq j_{1}<j_{2}<\cdots<j_{a}\leq c-1}\ \ \prod_{m=1}^{a}\left(
x_{j_{m}}+y_{j_{m}-m+1}\right)  +\sum_{1\leq j_{1}<j_{2}<\cdots<j_{a-1}
<j_{a}=c}\underbrace{\left(  x_{c}+y_{d}\right)  }_{\substack{=x_{c}
+y_{c-a+1}\\\text{(since }d=c-a+1\text{)}}}\prod_{m=1}^{a-1}\left(  x_{j_{m}
}+y_{j_{m}-m+1}\right) \\
&  =\sum_{1\leq j_{1}<j_{2}<\cdots<j_{a}\leq c-1}\ \ \prod_{m=1}^{a}\left(
x_{j_{m}}+y_{j_{m}-m+1}\right)  +\sum_{1\leq j_{1}<j_{2}<\cdots<j_{a-1}
<j_{a}=c}\underbrace{\left(  x_{c}+y_{c-a+1}\right)  \prod_{m=1}^{a-1}\left(
x_{j_{m}}+y_{j_{m}-m+1}\right)  }_{\substack{=\prod_{m=1}^{a}\left(  x_{j_{m}
}+y_{j_{m}-m+1}\right)  \\\text{(since }j_{a}=c\text{ yields }x_{j_{a}
}+y_{j_{a}-a+1}=x_{c}+y_{c-a+1}\text{,}\\\text{so that }x_{c}+y_{c-a+1}\text{
is} \\ \text{ the missing }m=a\text{ factor of the product)}}}\\
&  =\sum_{1\leq j_{1}<j_{2}<\cdots<j_{a}\leq c-1}\ \ \prod_{m=1}^{a}\left(
x_{j_{m}}+y_{j_{m}-m+1}\right)  +\sum_{1\leq j_{1}<j_{2}<\cdots<j_{a-1}
<j_{a}=c}\ \ \prod_{m=1}^{a}\left(  x_{j_{m}}+y_{j_{m}-m+1}\right)  .
\end{align*}
On the other hand,
\begin{align*}
&  \sum_{1\leq j_{1}<j_{2}<\cdots<j_{a}\leq c}\ \ \prod_{m=1}^{a}\left(
x_{j_{m}}+y_{j_{m}-m+1}\right) \\
&  =\sum_{1\leq j_{1}<j_{2}<\cdots<j_{a}\leq c-1}\ \ \prod_{m=1}^{a}\left(
x_{j_{m}}+y_{j_{m}-m+1}\right)  +\sum_{1\leq j_{1}<j_{2}<\cdots<j_{a-1}
<j_{a}=c}\ \ \prod_{m=1}^{a}\left(  x_{j_{m}}+y_{j_{m}-m+1}\right)
\end{align*}
(here, we have split the sum according to whether $j_{a}\leq c-1$ or $j_{a}
=c$). Comparing these two equalities, we obtain
\[
e_{a}(X_{c}/Y_{d})=\sum_{1\leq j_{1}<j_{2}<\cdots<j_{a}\leq c}\ \ \prod
_{m=1}^{a}\left(  x_{j_{m}}+y_{j_{m}-m+1}\right)  .
\]
This completes the induction. Thus Proposition~\ref{prop:superproduct} is proved.
\end{proof}

\begin{proof}[Proof of Lemma~\ref{lem.d-ne0}]
The condition $d^\mu \neq 0$ is simply saying that
there exists a desarrangement tableau of shape $\mu$.
Thus, it is false when $\mu = (i)$ for some $i > 0$
(since the only standard tableau in this case is
\[ \frakt = \ytableaushort{123\cdots i},\] which is not
a desarrangement tableau because the smallest element
of $\ive{n} \setminus \Des\tup{\frakt}$ is $1$),
and is likewise false when $\mu = (1^i)$ for some odd
$i$ (here, the smallest element of
$\ive{n} \setminus \Des\tup{\frakt}$ is $i$).
In all other cases, it is true, since we can find
a desarrangement tableau:
If $\mu_1 > 1$, then we can put the numbers $1,2,3$
into the boxes $\tup{1,1},\tup{2,1},\tup{1,2}$ and
fill in the rest of $\mu$ arbitrarily (keeping the
tableau standard); the smallest element of
$\ive{n} \setminus \Des\tup{\frakt}$ will then be
$2$.
Otherwise, we simply take
\[ \frakt = \ytableaushort{1,2,3,\vdots,i}. \]
\end{proof}

\begin{proof}[Proof of Lemma~\ref{lem.horstrip-bound}.]
    By the definition of content, $\content(r,s) = s-r$. But $(r, s) \in \lambda \sm \mu \subseteq \lambda$. Thus, $s \leq \lambda_1 \leq \abs{\lambda} = n$ and $r \geq 1$, so that $s-r \leq n-1 < n$.

    On the other hand, $(r, s) \in \lambda$ also shows that $r \leq \ell(\lambda) \leq \ell(\mu) + 1$ (since $\lambda \sm \mu$ is a horizontal strip), whereas $\ell(\mu) \leq \abs{\mu} = j$. So $r \leq \ell(\mu) + 1 \leq j + 1$. Combining this with $s \geq 1$, we find $s - r \geq 1 - (j+1) = -j$.

    So $-j \leq s-r < n$. In view of $\content(r,s) = s-r$, this is precisely the claim of the lemma.
\end{proof}

\begin{proof}[Proof of Lemma~\ref{lem.e=e}.]
Vieta's formulas say that
\begin{align*}
\prod_{i=1}^{p}\left(  x-a_{i}\right)    & =\sum_{i=0}^{p}\left(  -1\right)
^{i}e_{i}\left(  a_{1},a_{2},\ldots,a_{p}\right)  x^{p-i}
\ \ \ \ \ \ \ \ \ \ \text{and}\\
\prod_{i=1}^{p}\left(  x-b_{i}\right)    & =\sum_{i=0}^{p}\left(  -1\right)
^{i}e_{i}\left(  b_{1},b_{2},\ldots,b_{p}\right)  x^{p-i}
\end{align*}
in the polynomial ring $\kk \left[  x\right]  $. The right-hand sides of
these two equalities agree, since we have assumed that $e_{i}\left(
a_{1},a_{2},\ldots,a_{p}\right)  =e_{i}\left(  b_{1},b_{2},\ldots
,b_{p}\right)  $ for all $i\in\left\{  0,1,\ldots,p\right\}  $. Thus, the
left-hand sides agree as well. In other words,
\[
\prod_{i=1}^{p}\left(  x-a_{i}\right)
= \prod_{i=1}^{p}\left(  x-b_{i}\right)
.
\]
Since $\kk \left[  x\right]  $ is a unique factorization domain, this
entails that the irreducible polynomials $x-a_{1},\ x-a_{2},\ \ldots
,\ x-a_{p}$ must agree with the irreducible polynomials $x-b_{1}
,\ x-b_{2},\ \ldots,\ x-b_{p}$ up to order.
It follows that (up to order), the elements $a_1, a_2,\ldots, a_p$ must
agree with the elements $b_1, b_2,\ldots, b_p$.
That is, $\sfA = \sfB$.
\end{proof}

\begin{proof}[Proof of Lemma~\ref{lem.A2}.]
(a) Since $S$ separates $R$, we have $\bigcap_{\chi \in S} \ker\chi = 0$.
But each $\ker\chi$ is a hyperplane in the $\kk$-vector space $R$
(since $\chi$ is a nonzero linear map $R \to \kk$), thus a subspace
of codimension $1$.
Hence, the intersection $\bigcap_{\chi \in S} \ker\chi$ of these
subspaces must have codimension $\leq \sum_{\chi \in S} 1 = \abs{S}$.
Since this intersection is $0$, we thus conclude that
$\dim R \leq \abs{S}$.
Therefore, $R$ is finite-dimensional, so that $\dim(R^*) = \dim R$
and therefore $\abs{S} \geq \dim R = \dim(R^*)$.

On the other hand,
Dedekind's theorem on linear independence of characters
(see \cite[Lemma 9.13.1 = tag 0CKL]{stacks-project}
or \cite[Theorem 2.1 and the remark after its proof]{KConrad-linearchar})
says that any set of (distinct) morphisms from an arbitrary monoid
to the monoid $(\kk, \cdot, 1)$ is $\kk$-linearly independent.
Thus, the distinct characters $\chi \in S$ are linearly independent
(since they are morphisms from monoid
$\tup{R, \cdot, 1}$ to the monoid $(\kk, \cdot, 1)$).
Since they lie in $R^*$, and since
there are at least $\dim(R^*)$ many of them (because 
$\abs{S} \geq \dim(R^*)$), they must therefore form a basis of
$R^*$.
In particular, $\abs{S} = \dim(R^*) = \dim(R)$, qed.

(b) Let $u\in R$. We must show that $u=f(r)$ for some polynomial
$f \in \kk [x]$.

We have assumed that any two distinct characters $\phi,\psi\in S$ satisfy
$\phi\left(  r\right)  \neq\psi\left(  r\right)  $. In other words, if
$\phi,\psi\in S$ satisfy $\phi(r)=\psi(r)$, then $\phi=\psi$ and therefore
$\phi(u)=\psi(u)$.

Hence, we can define a map
\begin{align*}
\Gamma:\left\{  \phi\left(  r\right)  \ \mid\ \phi\in S\right\}    &
\rightarrow\left\{  \phi\left(  u\right)  \ \mid\ \phi\in S\right\}  ,\\
\phi\left(  r\right)    & \mapsto\phi\left(  u\right)  .
\end{align*}
Note that the set $\left\{  \phi\left(  r\right)  \ \mid\ \phi\in S\right\}  $
is finite. Hence, by Lagrange interpolation, there is a polynomial
$f \in \kk[x]$ that sends each $c\in\left\{  \phi\left(  r\right)
\ \mid\ \phi\in S\right\}  $ to $\Gamma\left(  c\right)  $. Consider this
polynomial $f$. Then, each $\phi\in S$ satisfies $f\left(  \phi\left(
r\right)  \right)  =\Gamma\left(  \phi\left(  r\right)  \right)  =\phi\left(
u\right)  $ (by the definition of $\Gamma$) and thus
\[
\phi\left(  u\right)  =f\left(  \phi\left(  r\right)  \right)  =\phi\left(
f\left(  r\right)  \right)  \ \ \ \ \ \ \ \ \ \ \left(  \text{since }\phi\in
S\text{ is a } \kk\text{-algebra morphism}\right)  ,
\]
so that $u-f\left(  r\right)  \in\ker\phi$. In other words, $u-f(r)\in
\bigcap_{\phi\in S}\ker\phi=\bigcap_{\chi\in S}\ker\chi=0$ (since $S$
separates $R$). Thus, $u-f(r)=0$, so that $u=f(r)$, qed.
\end{proof}

\section{\label{sec:commut-pf2}Another proof of commutativity}

Another proof of Theorem~\ref{thm:commut} relies on the following lemma:

\begin{lem}
\label{lem:sunshine}
Assume that $\kk \subseteq \RR$ and
$q>0$. Let $a\in \B_n^\ast \HH_n$. If
$\B_n a=0$, then $a=0$.
\end{lem}

\begin{proof}
Assume that $\B_n a=0$. We know that $a\in\B_n^{\ast
}\HH_n$. In other words, $a = \B_n^\ast b$ for some
$b\in\HH_n$. Thus, $a^\ast a=\left(  \B_n^\ast
b\right)  ^\ast a = b^\ast \B_n a=0$ (since $\B_n a=0$),
so that \cite[Lemma 6.1 (1)]{r2r1}
yields $a=0$. This proves Lemma~\ref{lem:sunshine}.
\end{proof}

\begin{proof}[Proof of Theorem~\ref{thm:commut}.]
We induct on $n$ (the case $n=0$ being obvious).

Let $n>0$ and $i,j\geq0$. We must show that
$\R_{n,i}\R_{n,j}=\R_{n,j}\R_{n,i}$.
This is obvious if $i = 0$ (since $\R_{n,i} = 1$ in this case), obvious if $j = 0$ (for similar reasons), obvious if $i > n$ (since $\R_{n,i} = 0$ in this case, by \eqref{eq:Rnk:zero}) and obvious if $j > n$ (for similar reasons).
Hence, without loss of generality, assume that both $i$ and $j$ are positive integers $\leq n$.
Furthermore, since our claim ($\R_{n,i}\R_{n,j}=\R_{n,j}\R_{n,i}$) is a polynomial identity in $q$, we assume without loss of generality that the ground
field $\kk$ is $\RR$ and that $q \in \RR_{>0}$.

We shall first show that
$\B_n \R_{n,j}\R_{n,i} = \B_n \R_{n,i}\R_{n,j}$.

Indeed, Theorem~\ref{thm:crucialrecursion} yields
\[
\B_n \R_{n,i}=\left(  q^{i}\R_{n-1,i}
+\left(  \ive{n+1-i}_q +q^{n+1-i}J_n\right)  \R_{n-1,i-1}\right) \B_n.
\]
Hence,
\begin{align*}
\B_n \R_{n,i}\R_{n,j} &  =\left(  q^{i}
\R_{n-1,i}+\left(  \ive{n+1-i}_q +q^{n+1-i}J_n\right)
\R_{n-1,i-1}\right)  \B_n \R_{n,j}\\
&  =\left(  q^{i}\R_{n-1,i}+\left(  \ive{n+1-i}_q +q^{n+1-i}
J_n \right)  \R_{n-1,i-1}\right)  \\
&  \ \ \ \ \ \ \ \ \ \ \left(  q^{j}\R_{n-1,j}+\left(
\ive{n+1-j}_q +q^{n+1-j}J_n \right)  \R_{n-1,j-1}\right)
\B_n
\end{align*}
(here, we have applied Theorem~\ref{thm:crucialrecursion} again, this time to
$\B_n \R_{n,j}$). Interchanging $i$ and $j$, we obtain
the equality
\begin{align*}
\B_n \R_{n,j}\R_{n,i} &  =\left(  q^{j}
\R_{n-1,j}+\left(  \ive{n+1-j}_q +q^{n+1-j}J_n \right)
\R_{n-1,j-1}\right)  \\
&  \ \ \ \ \ \ \ \ \ \ \left(  q^{i}\R_{n-1,i}+\left(
\ive{n+1-i}_q +q^{n+1-i}J_n \right)  \R_{n-1,i-1}\right)
\B_n.
\end{align*}
Our intermediate goal is to show that the left-hand sides of these two equalities are the
same. Clearly, it suffices to show that the right-hand sides are the same. For
this, it suffices to show that the two factors
\begin{align*}
&  q^{i}\R_{n-1,i}+\left(  \ive{n+1-i}_q +q^{n+1-i}J_n \right)
\R_{n-1,i-1}
\qquad\text{and}\\
&  q^{j}\R_{n-1,j}+\left(  \ive{n+1-j}_q +q^{n+1-j}J_n \right)
\R_{n-1,j-1}
\end{align*}
commute. But this is clear, since the elements $\R_{n-1,i}$,
$\R_{n-1,i-1}$, $\R_{n-1,j}$ and $\R_{n-1,j-1}$
pairwise commute (by the induction hypothesis) and also commute with $J_n$
(since they lie in $\HH_{n-1}$, but any element of $\HH_{n-1}$
commutes with $J_n$ by Lemma~\ref{lem:TiHn-1}).
Thus, the left-hand sides are the same.
In other words, we have shown that
\begin{equation}
\B_n \R_{n,j}\R_{n,i} = \B_n \R_{n,i}\R_{n,j}.
\label{eq.lem:commutB}
\end{equation}

However,
$\R_{n,i}=\frac{1}{\ive{i}!_q}\B_{n,i}^\ast \B_{n,i}\in
\B_n^\ast \HH_n$ (since $i>0$ entails $\B_{n,i}^{\ast
}=\B_n^\ast \B_{n-1}^\ast \cdots\B_{n-i+1}^\ast
\in\B_n^\ast \HH_n$). Thus,
$\R_{n,i}\R_{n,j}\in \B_n^\ast \HH_n$.
Similarly, $\R_{n,j}\R_{n,i}\in \B_n^\ast \HH_n$.
Hence, the difference $\R_{n,i}\R_{n,j}
-\R_{n,j}\R_{n,i}$ belongs to $\B_n^{\ast
}\HH_n$ as well. But \eqref{eq.lem:commutB} shows that
$\B_n \left(  \R_{n,i}\R_{n,j}-\R_{n,j}\R_{n,i}\right)  =0$.
Thus, Lemma~\ref{lem:sunshine} (applied
to $a=\R_{n,i}\R_{n,j}-\R_{n,j}\R_{n,i}$)
yields $\R_{n,i}\R_{n,j}-\R_{n,j}\R_{n,i}=0$.
In other words, $\R_{n,i}\R_{n,j}=\R_{n,j}\R_{n,i}$.
This is exactly what we wanted to show.
Thus, the induction is complete.
\end{proof}

\begin{remark}
The above proof seems paradigmatic: It renders the commutativity of the $\R_{n,k}$ an almost predestined consequence of the recursion given in Theorem~\ref{thm:crucialrecursion}, provided that the general case can be reduced to the fair-weather case of $q > 0$.
It appears likely that the same strategy can be used to prove commutativity of other families in the Hecke algebra, provided that similar recursions are found.
\end{remark}

\section{A theory of split elements}\label{appendix.splitelements}

The above proof of Theorem~\ref{thm:big.eigenvalues} took us a long journey.
However, its ``catchiest'' part -- the fact that the eigenvalues are polynomials in $q$ with integer coefficients -- can be accessed through an easier route.
In this section, we shall show this route, hopefully signposting it for uses in other situations (e.g., graph spectra) where eigenvalues turn out to be integers.
Our goal is to prove the following proposition:

\begin{prop}
\label{prop:evals-qpols}
Let $n, k \geq 0$.
All eigenvalues of $\R_{n,k}$ (acting by right
multiplication on $\HH_n$ over $\kk$) are polynomials in
$q$ with integer coefficients. That is, the characteristic polynomial of
$\R_{n,k}$ (acting by right multiplication on $\HH_n$) factors as a product
of linear terms $x-p_{i}\left(  q\right)  $, where each $p_{i}$ is a
univariate polynomial over $\mathbb{Z}$.
\end{prop}

Our proof will rely on a notion of \emph{splitness} of an element over a
commutative ring. This notion is somewhat similar to the classical notion of
integrality over a commutative ring (\cite[Section 15.3]{DummitFoote},
\cite[Section 4.1]{ChambertLoir}), but also related to the notion of split
algebras (\cite{Bahran}, \cite[Definition 2.8]{burgess2024domain}). We
believe that it might have uses beyond the proof of
Proposition~\ref{prop:evals-qpols}; thus, we present it at a greater
generality than directly necessary.

In this section, $\kk$ shall be an arbitrary commutative ring (not
necessarily a field). Its elements will be called \emph{scalars}.

\begin{definition}
\label{def:split}
Let $A$ be a $\kk$-algebra. An element $a\in A$ is
said to be \emph{split} (over $\kk$) if there exist some scalars
$u_1, u_2, \ldots, u_n \in\kk$ (not necessarily distinct) such that
$\prod\limits_{i=1}^n \left(  a-u_i \right)  =0$.
\end{definition}
Note that the product here is well-defined, since all the $a-u_i$ commute.
The elements of $\kk$ itself are split, since each
$a\in\kk$ satisfies $\prod_{i=1}^{1} \left(  a-a\right)  =a-a=0$.
Split elements are integral over $\kk$, but not vice versa.
The relation between the notion of splitness and Proposition
\ref{prop:evals-qpols} is the following:

\begin{prop}
\label{prop:split:matrix}
Assume that $\kk$ is a field. Let $M$ be a
square matrix over $\kk$. Then, $M$ is split (regarded as an element of
the matrix ring) if and only if all eigenvalues of $M$ belong to $\kk$.
\end{prop}

\begin{proof}
$\Longrightarrow:$ Assume that $M$ is split. Thus, there exist some scalars
$u_1, u_2, \ldots, u_n \in\kk$ such that 
\[
\prod\limits_{i=1}^n \left(  M-u_i\right)  =0.
\]
Consider these scalars. Hence, the minimal
polynomial of $M$ must divide the polynomial 
\[
\prod\limits_{i=1}^n \left(  x-u_i\right)  .
\]
Therefore, all eigenvalues of $M$ belong to $\left\{
u_1, u_2, \ldots, u_n \right\}  $ (since all eigenvalues of a matrix are
roots of its minimal polynomial), and thus belong to $\kk$ as well.

$\Longleftarrow:$ Assume that all eigenvalues of $M$ belong to $\kk$.
Let $u_1, u_2, \ldots, u_n \in\kk$ be these eigenvalues (with their
algebraic multiplicities). Hence, by the Cayley--Hamilton theorem, we have
$\prod\limits_{i=1}^n \left(  M-u_i\right)  =0$. This shows that $M$ is split.
\end{proof}

Proposition~\ref{prop:split:matrix} reduces the proof of Proposition
\ref{prop:evals-qpols} to showing that $\R_{n,k}$ is split over $\mathbb{Z}
\left[  q\right]  $. To show this, we will now show general properties of
split elements.

\begin{lem}
\label{lem:split:1}
Let $A$ be a $\kk$-algebra. Let $a$ and $b$ be two
commuting elements of $A$. Let $u_1, u_2, \ldots, u_n \in\kk$ be
scalars such that $\prod_{i=1}^n \left(  a-u_i\right)  =0$. Let
$v_{1},v_{2},\ldots,v_{m}\in\kk$ be scalars such that
$\prod_{j=1}^{m}\left(  b-v_{j}\right)  =0$. Then,
\begin{equation}
\prod_{i=1}^n \ \ \prod_{j=1}^{m}\left(  a+b-u_i-v_{j}\right)  =0
\label{eq:lem:split:1:a+b}
\end{equation}
and
\begin{equation}
\prod_{i=1}^n \ \ \prod_{j=1}^{m}\left(  ab-u_iv_{j}\right)  =0.
\label{eq:lem:split:1:ab}
\end{equation}

\end{lem}

\begin{proof}
Let $B$ be the $\kk$-subalgebra of $A$ generated by $a$ and $b$. This
subalgebra $B$ is commutative, since its generators $a$ and $b$ commute.

For each $i\in\left[  n\right]  $, we set
$s_{i}:=\prod_{j=1}^{m}\left(  a+b-u_i-v_{j}\right)  $ and
$p_{i}:=\prod_{j=1}^{m}\left(  ab-u_i v_{j}\right)  $.
Thus, our goal is to show that $\prod_{i=1}^n s_{i}=0$ and
$\prod_{i=1}^n p_{i}=0$.

For each $i\in\left\{  0,1,\ldots,n\right\}  $, define the ideal
\[
U_{i}:=\left(  a-u_{1}\right)  \left(  a-u_{2}\right)  \cdots\left(
a-u_i\right)  B
\]
of $B$. Clearly, $U_0 \supseteq U_1 \supseteq U_2 \supseteq
\cdots \supseteq U_n$. Furthermore, $U_0 = B$
(since the product $\left(  a-u_{1}\right)
\left(  a-u_{2}\right)  \cdots\left(  a-u_{0}\right)  $ is empty and thus
equals $1$) and $U_n = 0$ (since $\left(  a-u_{1}\right)  \left(
a-u_{2}\right)  \cdots\left(  a-u_n \right)  =\prod_{i=1}^n \left(
a-u_i\right)  =0$).

Now, fix $i\in\left[  n\right]  $. Then, $U_{i-1}=\left(  a-u_{1}\right)
\left(  a-u_{2}\right)  \cdots\left(  a-u_{i-1}\right)  B$, so that
\begin{align*}
\left(  a-u_i\right)  U_{i-1}  &  =\underbrace{\left(  a-u_i\right)
\cdot\left(  a-u_{1}\right)  \left(  a-u_{2}\right)  \cdots\left(
a-u_{i-1}\right)  }_{=\left(  a-u_{1}\right)  \left(  a-u_{2}\right)
\cdots\left(  a-u_i\right)}B\\
&  =\left(  a-u_{1}\right)  \left(  a-u_{2}\right)  \cdots\left(
a-u_i\right)  B=U_{i}.
\end{align*}
Hence, the element $a-u_i$ acts as $0$ on the quotient $B$-module
$U_{i-1}/U_{i}$. In other words, the element $a$ acts as $u_{i}$ on this
$B$-module $U_{i-1}/U_{i}$. Therefore, for any polynomial $P\left(  x\right)
\in B\left[  x\right]  $, the element $P\left(  a\right)  \in B$ acts
as $P\left(  u_{i}\right)  $ on this $B$-module $U_{i-1}/U_{i}$. Applying this
to $P\left(  x\right)  =\prod_{j=1}^{m}\left(  x+b-u_i-v_{j}\right)  $, we
conclude that the element $\prod_{j=1}^{m}\left(  a+b-u_i-v_{j}\right)  \in
B$ acts as $\prod_{j=1}^{m}\underbrace{\left(  u_{i}+b-u_i-v_{j}\right)
}_{=b-v_{j}}=\prod_{j=1}^{m}\left(  b-v_{j}\right)  =0$ on this $B$-module
$U_{i-1}/U_{i}$. In other words, the element $s_{i}$ acts as $0$ on the
$B$-module $U_{i-1}/U_{i}$ (since $s_{i}=\prod_{j=1}^{m}\left(  a+b-u_i
-v_{j}\right)  $). In other words, $s_{i}U_{i-1}\subseteq U_{i}$.

A similar argument shows that $p_{i}U_{i-1}\subseteq U_{i}$ (here, we need to
observe that $p_{i}=\prod_{j=1}^{m}\left(  ab-u_iv_{j}\right)  $ acts as
$\prod_{j=1}^{m}\underbrace{\left(  u_{i}b-u_iv_{j}\right)  }_{=u_{i}\left(
b-v_{j}\right)  }=u_{i}^{m}\underbrace{\prod_{j=1}^{m}\left(  b-v_{j}\right)
}_{=0}=0$ on $U_{i-1}/U_{i}$).

Now, forget that we fixed $i$. We thus have shown that each $i\in\left[
n\right]  $ satisfies
\begin{align}
s_{i}U_{i-1}  &  \subseteq U_{i}\qquad\text{and}
\label{pf:lem:split:1:sU}\\
p_{i}U_{i-1}  &  \subseteq U_{i}.
\label{pf:lem:split:1:pU}
\end{align}

Now,
\begin{align*}
\prod_{i=1}^n s_{i}  &  = s_n s_{n-1}\cdots s_{3}s_{2}s_{1}=s_n
s_{n-1}\cdots s_{3}s_{2}s_{1}\underbrace{1}_{\in B=U_{0}}\\
&  \in s_n s_{n-1}\cdots s_{3}s_{2}
\underbrace{s_{1}U_{0}}_{\substack{\subseteq U_{1}\\
\text{(by \eqref{pf:lem:split:1:sU})}}}\subseteq
s_n s_{n-1}\cdots s_{3}\underbrace{s_{2}U_{1}}_{\substack{\subseteq
U_{2}\\\text{(by \eqref{pf:lem:split:1:sU})}}}
\subseteq s_n s_{n-1}\cdots\underbrace{s_{3}U_{2}}_{\substack{\subseteq
U_{3}\\\text{(by \eqref{pf:lem:split:1:sU})}}}
\subseteq\cdots\subseteq U_n =0.
\end{align*}
In other words, $\prod_{i=1}^n s_{i}=0$. Similarly, using
\eqref{pf:lem:split:1:pU}, we obtain $\prod_{i=1}^n p_{i}=0$. As explained
above, this completes the proof of Lemma~\ref{lem:split:1}.
\end{proof}

\begin{theorem}
\label{thm:split:a+b-ab}
Let $A$ be a $\kk$-algebra. Let $a$ and $b$ be
two commuting split elements of $A$. Then, the elements $a+b$ and $ab$ are
split as well.
\end{theorem}

\begin{proof}
This follows from Lemma~\ref{lem:split:1}.
\end{proof}

\begin{prop}
\label{prop:split:lama}
Let $A$ be a $\kk$-algebra. Let $a$ be a split
element of $A$. Let $\lambda\in\kk$. Then, $\lambda a$ is also split.
\end{prop}

\begin{proof}
Since $a$ is split, there exist scalars $u_1, u_2, \ldots, u_n \in
\kk$ satisfying $\prod_{i=1}^n \left(  a-u_i\right)  =0$. Consider
these scalars. Then, $\prod_{i=1}^n \underbrace{\left(  \lambda a-\lambda
u_{i}\right)  }_{=\lambda\left(  a-u_i\right)  }=\lambda^n
\underbrace{\prod_{i=1}^n \left(  a-u_i\right)  }_{=0}=0$, and this shows
that $\lambda a$ is split.
\end{proof}

\begin{theorem}
\label{thm:split:subalg}
Let $A$ be a commutative $\kk$-algebra that is
generated by split elements. Then, any element of $A$ is split.
\end{theorem}

\begin{proof}
Let $S$ be the set of all split elements of $A$. Then, $S$ is closed under
addition and multiplication (by Theorem~\ref{thm:split:a+b-ab}) and under
scaling (by Proposition~\ref{prop:split:lama}), and contains $\kk$ as a
subset (since each element of $\kk$ is split). Thus, $S$ is a
$\kk$-subalgebra of $A$. Since $S$ contains a set of generators of $A$
(because $A$ is generated by split elements), we thus see that $S$ must be $A$
itself. Hence, any element of $A$ lies in $S$, and thus is split.
\end{proof}

Theorem~\ref{thm:split:subalg} provides a convenient way to infer the splitness of certain elements from the splitness of other elements.
Another tool for such inferences (particularly useful in noncommutative $\kk$-algebras) is the following theorem:

\begin{theorem}
\label{thm:split:br=wb}
Let $A$ be a $\kk$-algebra. Let $b,c,d,f$ be
elements of $A$ such that $f$ is split and such that $bc=fb$ and $c=db$. Then,
$c$ is split.
\end{theorem}

\begin{proof}
Since $f$ is split, there exist scalars $u_1, u_2, \ldots, u_n \in
\kk$ such that $\prod_{i=1}^n \left(  f-u_i\right)  =0$. Consider
these scalars. For each $i\in\left[  n\right]  $, we have
\begin{equation}
b\left(  c-u_i\right)
=\underbrace{bc}_{=fb}- \,u_{i}b
=fb-u_ib
=\left(  f-u_i\right)  b.
\label{pf:thm:split:ab=bb:1}
\end{equation}

Now,
\begin{align*}
b\prod_{i=1}^n \left(  c-u_i\right)
&  =\underbrace{b\left(
c-u_{1}\right)  }_{\substack{=\left(  f-u_{1}\right)  b\\\text{(by
\eqref{pf:thm:split:ab=bb:1})}}}\left(  c-u_{2}\right)  \left(  c-u_{3}
\right)  \cdots\left(  c-u_n \right) \\
&  =\left(  f-u_{1}\right)  \underbrace{b\left(  c-u_{2}\right)
}_{\substack{=\left(  f-u_{2}\right)  b\\\text{(by
\eqref{pf:thm:split:ab=bb:1})}}}\left(  c-u_{3}\right)  \cdots\left(
c-u_n \right) \\
&  =\left(  f-u_{1}\right)  \left(  f-u_{2}\right)  \underbrace{b\left(
c-u_{3}\right)  }_{\substack{=\left(  f-u_{3}\right)  b\\\text{(by
\eqref{pf:thm:split:ab=bb:1})}}}\cdots\left(  c-u_n \right) \\
&  =\cdots\\
&  =\underbrace{\left(  f-u_{1}\right)  \left(  f-u_{2}\right)  \cdots\left(
f-u_n \right)  }_{=\prod_{i=1}^n \left(  f-u_i\right)  =0}b=0.
\end{align*}
But $c = db$, and thus
$c\prod_{i=1}^n \left(  c-u_i\right)
= d\underbrace{b\prod_{i=1}^n \left(  c-u_i\right)}_{=0}
= 0$.
In other words, $\prod_{i=0}^n \left( c-u_i\right)  =0$, where we set
$u_{0}:=0$. This shows that $c$ is split.
\end{proof}

We observe that Theorem~\ref{thm:split:br=wb} has two remarkable corollaries, which we shall not use:

\begin{cor}
\label{cor:split:ab-ba}
Let $A$ be a $\kk$-algebra. Let $a$ and $b$ be
two elements of $A$ such that the product $ba$ is split. Then, $ab$ is also split.
\end{cor}

\begin{proof}
We can apply Theorem~\ref{thm:split:br=wb} to $c = ab$ and $f = ba$ and $d = a$
(since $ba$ is split and since $b\cdot ab = ba \cdot b$ and $ab = ab$).
Thus we conclude that $ab$ is split.
\end{proof}

\begin{cor}
\label{cor:split:ab=bb}
Let $A$ be a $\kk$-algebra. Let $a$ and $b$ be
two elements of $A$ such that $b^{2}=ab$. Assume that $a$ is split. Then, $b$
is also split.
\end{cor}

\begin{proof}
We can apply Theorem~\ref{thm:split:br=wb} to $c = b$ and $f = a$ and $d = 1$
(since $a$ is split, and since $bb = b^2 = ab$ and $b = 1b$).
Thus we conclude that $b$ is split.
\end{proof}

To apply the above theory to the Hecke algebras, we need one input from
representation theory:

\begin{prop}
\label{prop:split:J}
Assume that $q$ is invertible in $\kk$. Then, the
Young--Jucys--Murphy elements $J_1, J_2, \ldots, J_n$ in $\HH_n$ are
split (over $\kk$).
\end{prop}

\begin{proof}
The formula \cite[(5.2)]{Murphy92} says that each $m\in\left[  n\right]  $
satisfies
\[
\prod_{c\in\mathcal{R}\left(  m\right)  }\left(  J_{m}-c\right)  =0,
\]
where\footnote{If some of the $q$-integers $\ive{k}_q$ are equal, then
this $\mathcal{R}\tup{m}$ should be understood as a multiset.}
\[
\mathcal{R}\left(  m\right)  :=
\begin{cases}
\left\{  \ive{k}_q \ \mid\ -m<k<m\right\}  , & \text{if }m \geq 4;\\
\left\{  \ive{k}_q \ \mid\ -m<k<m\right\}  \setminus\left\{
0\right\}  , & \text{if }1<m<4; \\
\left\{  0 \right\}  , & \text{if }m = 1.\\
\end{cases}
\]
Since all the $c\in\mathcal{R}\left(  m\right)  $ are scalars in $\kk$,
this equality shows that each $J_{m}$ is split.
Proposition~\ref{prop:split:J} is proved.
\end{proof}

We are now ready to prove the following:

\begin{lem}
\label{lem:split:Rnk}
Assume that $\kk$ contains the field of rational
functions $\mathbb{Q}\left(  q\right)  $. Let $n,k \geq 0$. Then, the
element $\R_{n,k}$ of $\HH_n$ is split (over $\kk$).
\end{lem}

The assumption on $\kk$ will soon be disposed of, but we need it for
the following proof:

\begin{proof}[Proof of Lemma~\ref{lem:split:Rnk}.]
We induct on $n$. The base case ($n=0$)
is obvious, so we pass to the induction step.

If $k=0$, then our claim is obvious, since $\R_{n,0}=1$. So is it for $k>n$,
since $\R_{n,k}=0$ in this case. Thus, without loss of generality,
we assume that $1\leq k\leq n$.
Hence, Theorem~\ref{thm:crucialrecursion} yields
\begin{equation}
\B_n \R_{n,k} = \mathcal{W}_{n,k} \B_n,
\label{eq:lem:split:Rnk:BR=WB}
\end{equation}
where
\[
\mathcal{W}_{n,k}:=q^{k}\R_{n-1,k}+\Big(\left[  n+1-k\right]  _q
+q^{n+1-k}J_n \Big)\R_{n-1,k-1}.
\]

Now we shall show that $\mathcal{W}_{n,k}$ is split. Indeed, by our induction
hypothesis, the elements $\R_{n-1,k}$ and $\R_{n-1,k-1}$ are split (in
$\HH_{n-1}$ and thus also in the larger algebra $\HH_n$). So is the element
$J_n$ (by Proposition~\ref{prop:split:J}). Moreover, these three elements
$\R_{n-1,k}$ and $\R_{n-1,k-1}$ and $J_n$ all commute (indeed, Theorem
\ref{thm:commut} shows that the elements $\R_{n-1,k}$ and $\R_{n-1,k-1}$
commute, whereas Lemma~\ref{lem:TiHn-1} yields that $J_n$ commutes with both
of them as well). Hence, the $\kk$-subalgebra $A$ of $\HH_n$
generated by these three elements $\R_{n-1,k}$ and $\R_{n-1,k-1}$ and $J_n$
is commutative. Since its generators $\R_{n-1,k}$ and $\R_{n-1,k-1}$ and
$J_n$ are split, we thus conclude by Theorem~\ref{thm:split:subalg} that any
element of this subalgebra $A$ must be split. In particular,
$\mathcal{W}_{n,k}$ must be split (since the definition of
$\mathcal{W}_{n,k}$ shows that $\mathcal{W}_{n,k}$ belongs to $A$).

Moreover, \eqref{eq:Rnk:def} yields
\begin{align*}
\R_{n,k}  &  =\frac{1}{\left[  k\right]  !_q}\B_{n,k}^\ast \B_{n,k}\\
&  =\frac{1}{\left[  k\right]  !_q}\B_{n,k}^\ast \B_{n-(k-1)}\cdots
\B_{n-1}\B_n\ \ \ \ \ \ \ \ \ \ \left(  \text{by \eqref{eq:Bnk:def}}\right)
\\
&  =\mathcal{U}_{n,k}\B_n,\ \ \ \ \ \ \ \ \ \ \text{where }
\mathcal{U}_{n,k}
:=\frac{1}{\left[  k\right]  !_q}\B_{n,k}^\ast \B_{n-(k-1)}
\cdots\B_{n-2}\B_{n-1}.
\end{align*}

Therefore, Theorem~\ref{thm:split:br=wb} (applied to $\HH_n$, $\B_n$,
$\R_{n,k}$, $\mathcal{U}_{n,k}$ and $\mathcal{W}_{n,k}$ instead of $A$, $b$,
$c$, $d$ and $f$) shows that $\R_{n,k}$ is split (because of
\eqref{eq:lem:split:Rnk:BR=WB}). This completes the induction step, and thus
proves Lemma~\ref{lem:split:Rnk}.
\end{proof}

We can now prove Proposition~\ref{prop:evals-qpols}:

\begin{proof}[Proof of Proposition~\ref{prop:evals-qpols}.]
We must
show that the characteristic polynomial of $\R_{n,k}$ (acting by right
multiplication on $\HH_n$ over $\kk$) can be factored into monic
linear factors of the form $x-p_{i}\left(  q\right)  $, where each $p_{i}$ is
a univariate polynomial over $\mathbb{Z}$.

Clearly, it suffices to show this in the case when our base ring $\kk$
is $\mathbb{Z}\left[  q\right]  $. Thus,
without loss of generality, we assume that we are in this
case. Note that $\mathbb{Z}\left[  q\right]  $ is a unique factorization
domain with quotient field $\mathbb{Q}\left(  q\right)  $; it is thus
integrally closed in $\mathbb{Q}\left(  q\right)  $ (since any unique
factorization domain is integrally closed in its quotient field --- see
\cite[Theorem (4.1.9)]{ChambertLoir}). In other words, any element of
$\mathbb{Q}\left(  q\right)  $ that is integral over $\mathbb{Z}\left[
q\right]  $ must belong to $\mathbb{Z}\left[  q\right]  $.

The equality \eqref{eq:uI:Rnk=} contains no denominators or negative powers of
$q$. Thus, the element $\R_{n,k}\in\HH_n$ is defined over the ring
$\mathbb{Z}\left[  q\right]  $. We shall view the Hecke algebra $\HH_n$
defined over the ring $\mathbb{Z}\left[  q\right]  $ as a subring of the
corresponding Hecke algebra $\HH_n$ defined over the field $\mathbb{Q}
\left(  q\right)  $. The element $\R_{n,k}$ belongs to both of these Hecke
algebras, and thus acts on them both by right multiplication; the actions
clearly have the same characteristic polynomial.
Let $\chi_{\R_{n,k}} = \det\left(x \id_{\HH_n} - \R_{n,k} \right)
\in \tup{\ZZ\ive{q}}\ive{x}$ be this
characteristic polynomial.

We must show that $\chi_{\R_{n,k}}$ can be factored into monic
linear factors of the form $x-p_{i}\left(  q\right)  $ with $p_{i}
\in\mathbb{Z}\left[  q\right]  $.

Lemma~\ref{lem:split:Rnk} shows that $\R_{n,k}$ is split over the field
$\mathbb{Q}\left(  q\right)  $. Hence, Proposition~\ref{prop:split:matrix}
(applied to the field $\mathbb{Q}\left(  q\right)  $) shows that all
eigenvalues of $\R_{n,k}$ belong to $\mathbb{Q}\left(  q\right)  $.
In other words, all roots of
$\chi_{\R_{n,k}}$ belong to $\mathbb{Q}\left(  q\right)  $
(since the eigenvalues of $\R_{n,k}$ are precisely the roots of
$\chi_{\R_{n,k}}$).
But at the same time, these roots are integral over $\ZZ\ive{q}$ (since
$\chi_{\R_{n,k}}$ is a monic polynomial with coefficients in $\ZZ\ive{q}$).
Hence, these roots must belong to $\ZZ\ive{q}$ (since any element of
$\mathbb{Q}\left(  q\right)  $ that is
integral over $\mathbb{Z}\left[  q\right]  $ must belong to $\mathbb{Z}\left[
q\right]  $).

But the polynomial $\chi_{\R_{n,k}}$ is monic, and thus can be factored into
monic linear factors of the form $x-p_{i}\left(  q\right)  $, where the $p_i$
are its roots. Since these roots (as we just saw) belong to $\mathbb{Z}\left[
q\right]  $,
we thus have shown that $\chi_{\R_{n,k}}$ can be factored into monic
linear factors of the form $x-p_{i}\left(  q\right)  $ with $p_{i}
\in\mathbb{Z}\left[  q\right]  $.
As explained above, this proves Proposition
\ref{prop:evals-qpols}.
\end{proof}
\end{verlong}

\begin{private}
\section{Further relations}
\label{sec:furtherrels}
\sarah[inline]{I don't think this appendix is necessary for the arxiv version}
\darij[inline]{Would it hurt though?}
In this appendix, we collect a few further properties of the elements of $\HH_n$ studied above.

\begin{prop}[Commutation relations] \label{prop.JnBn}
    Let $n > 1$. Then:
    \begin{align}
	J_n T_{n-1} - T_{n-1} J_n &= T_{n-1} J_{n-1} - J_{n-1} T_{n-1};
	\label{eq:prop.JnBn.1} \\
	J_n \B_n - \B_n J_n &= \Big( T_{n-1}J_{n-1} - J_{n-1} T_{n-1} \Big) \B_{n-1};
	\label{eq:prop.JnBn.2} \\
    \B^*_n J_n - J_n \B^*_n &= \B^*_{n-1} \Big( J_{n-1} T_{n-1} - T_{n-1}J_{n-1} \Big).
	\label{eq:prop.JnBn.3}
    \end{align}
\end{prop}

\begin{proof} We prove each item in turn:
\begin{enumerate}
    \item Lemma~\ref{TiJi} yields $T_{n-1} J_{n-1} = J_n \tup{T_{n-1} -q+1} - 1 = J_n T_{n-1} - (q-1)J_n -1$. Thus,
    \begin{equation}
    J_n  T_{n-1} = T_{n-1} J_{n-1}  + (q-1) J_n  + 1.
    \label{pf.prop.JnBn.1}
    \end{equation}
    By an analogous argument (or by applying the $a \mapsto a^*$ anti-involution of $\HH_n$ to the equality \eqref{pf.prop.JnBn.1} and recalling that $J_n^* = J_n$ and $T_{n-1}^* = T_{n-1}$), we find
    \begin{equation}
    T_{n-1} J_n  = J_{n-1} T_{n-1}  + (q-1) J_n  + 1.
    \label{pf.prop.JnBn.2}
	\end{equation}
	Subtracting this equality from \eqref{pf.prop.JnBn.1}, we find
	\[
	J_n T_{n-1} - T_{n-1} J_n = T_{n-1} J_{n-1} - J_{n-1} T_{n-1}.
	\]
	This proves \eqref{eq:prop.JnBn.1}.
    \item
	The element $J_n$ commutes with $\B_{n-1} $ (by Lemma~\ref{lem:TiHn-1}).
	That is, $\B_{n-1} J_n = J_n \B_{n-1}$.
	On the other hand, Definition~\ref{def.B-B*} easily
	yields $\B_n  = 1 + T_{n-1} \B_{n-1}$
	(since $T_{n-1} T_{n-2} \cdots T_i = T_{n-1} (T_{n-2} T_{n-3} \cdots T_i)$
	whenever $i < n$).
	Therefore,
    \begin{align*}
	J_n \B_n - \B_n J_n
	&= J_n \tup{1 + T_{n-1} \B_{n-1}} - \tup{1 + T_{n-1} \B_{n-1}} J_n \\
	&= J_n T_{n-1} \B_{n-1} - \, T_{n-1} \B_{n-1} J_n \\
	&= J_n T_{n-1} \B_{n-1} - \, T_{n-1} J_n \B_{n-1}
	\qquad \tup{\text{since } \B_{n-1} J_n = J_n \B_{n-1}}\\
	&= \tup{J_n T_{n-1} - T_{n-1} J_n} \B_{n-1} \\
	&= \tup{T_{n-1} J_{n-1} - J_{n-1} T_{n-1}} \B_{n-1}
	\qquad \tup{\text{by \eqref{eq:prop.JnBn.1}}}.
	\end{align*}
	Thus, \eqref{eq:prop.JnBn.2} is proved.
	\item Finally, applying $*$ to \eqref{eq:prop.JnBn.2}, we find
    \begin{align*}
    \Big(J_n  \B_n  - \B_n  J_n \Big)^* &= \Big( \big( T_{n-1} J_{n-1}  - J_{n-1}  T_{n-1} \big) \B_{n-1}  \Big)^*
	=\B_{n-1}^*  \big( J_{n-1}  T_{n-1} - T_{n-1} J_{n-1} \big),
    \end{align*}
	since $J_n ^* = J_n$ and $J_{n-1}^* = J_{n-1}$ and $T_{n-1}^* = T_{n-1}$.
	Comparing this with
    \[
	\Big(J_n  \B_n  - \B_n  J_n \Big)^*
	= \B_n^* J_n - J_n \B_n^* \qquad \tup{\text{since }J_n^* = J_n},
    \]
    we obtain \eqref{eq:prop.JnBn.3}. \qedhere
\end{enumerate}
\end{proof}
\end{private}

\begin{private}
\newpage

\section{The affine Hecke algebras}

\darij[inline]{This section has been removed from the paper, since there are no finished results or even concrete conjectures here.}

Assume that $q \neq 0$.
The Hecke algebra $\HH_n$ can be embedded in the \emph{affine Hecke algebra $\HH_n^{\aff}$}, which contains both $\HH_n$ and the Laurent polynomial ring $\kk\ive{x_1^{\pm1}, x_2^{\pm1}, \ldots, x_n^{\pm1}}$ as subalgebras and connects them by the additional relations
\begin{align}
    T_i x_i T_i &= q x_{i+1} \qquad \text{ for all } i \in \ive{n-1}; \label{eq:Hnaff:1} \\
    T_i x_j &= x_j T_i \qquad \text{ whenever } j \notin \set{i, i+1} \label{eq:Hnaff:2}.
\end{align}
In view of \eqref{eq:Hndef:1} and $q \neq 0$, we can replace the relation \eqref{eq:Hnaff:1} by the equivalent relations
\begin{align}
    T_i x_i &= x_{i+1} T_i + \tup{q-1} x_{i+1} \qquad \text{ for all } i \in \ive{n-1}; \label{eq:Hnaff:1a} \\
    T_i x_{i+1} &= x_i T_i - \tup{q-1} x_{i+1} \qquad \text{ for all } i \in \ive{n-1}. \label{eq:Hnaff:1b}
\end{align}
\darij[inline]{Note that $T_i$ commutes with $x_i + x_{i+1}$ and with $x_i x_{i+1}$ and also with all the other $x_j$'s. This allows us to use a lot of reasoning from the classical Hecke/SGA theory even if we can't easily commute $T_i$ past $x_i$.}

As a $\kk$-vector space, $\HH_n^{\aff}$ has a basis consisting of all products $T_w \mathfrak{m}$, where $w$ ranges over the permutations in $\symm_n$ while $\mathfrak{m}$ ranges over the Laurent monomials in $x_1,x_2,\ldots,x_n$.
The involutive anti-automorphism $a \mapsto a^*$ on $\HH_n$ can be extended to an involutive anti-automorphism of $\HH^{\aff}_n$ by setting $f^* := f$ for any Laurent polynomial $f$.
See \cite[Section 5]{Vazirani} for more about this algebra (note that $x_i$ is called $X_i$ in \cite{Vazirani}, and the algebra $\HH_n^{\aff}$ itself is called $H_n$).
When $q=1$, this algebra $\HH_n^{{\aff}}$ becomes the group algebra of the wreath product of the symmetric group $\symm_n$ with its natural $\ZZ$-module $\ZZ^n$.

Now, for any $n,k \ge 0$, we define
\begin{align}
		\Rt_{n,k} & := \frac{1}{\ive{k}!_q} \B^*_{n,k} \tup{x_{n-k+1} \cdots x_{n-1} x_n} \B_{n,k} \in \HH^{\aff}_n.
        \label{eq:Rtnk:def}
\end{align}
This can be viewed as a deformation of $\R_{n,k}$.
We conjecture that an analogue of Theorem~\ref{thm:commut} holds for these elements $\Rt_{n,k}$ as well:

\begin{conj}
    \label{conj:commut-x}
    Let $n\geq0$. The elements $\Rt_{n,k}$ for all $k \geq0$ commute.
\end{conj}


In the case $q = 1$, the elements $\Rt_{n,k}$ have a clear combinatorial meaning:

\begin{prop}
\label{prop:q=1-R2R-x}
Assume that $q=1$. Let $n \ge k \ge 0$.
Then,
\begin{align*}
\Rt_{n,n-k}
&= \sum_{w \in \symm_n} \ \ \sum_{\substack{1 \leq i_1 < i_2 < \cdots < i_k \leq n; \\ w\left(i_1\right) < w\left(i_2\right) < \cdots < w\left(i_k\right)}}
w \cdot \prod_{\substack{i \in \ive{n}; \\ i \notin \set{i_1,i_2,\ldots,i_k}}} x_i.
\end{align*}
\end{prop}

OUCH! Conjecture~\ref{conj:commut-x} fails for $n = 3$ when $q \neq 1$.
We probably need another deformation of the $q=1$ case...

\end{private}

\end{document}